\documentclass[11pt]{article}
\usepackage{mathrsfs}
\usepackage{amsfonts,amsmath,dsfont,amssymb,amsthm,stmaryrd,bbm,color,comment}
\usepackage[hidelinks]{hyperref}
\usepackage[pdftex]{graphicx}
\usepackage{bbm}
\usepackage{array}
\newcolumntype{L}{>{\centering\arraybackslash}m{2.5cm}}
\usepackage{multirow}
\usepackage{xcolor}
\usepackage[percent]{overpic}
\usepackage{algorithm}
\usepackage{algpseudocode}

\newtheorem{thm}{Theorem}
\newtheorem{prop}[thm]{Proposition}

\newtheorem{lem}[thm]{Lemma}
\newtheorem{coro}[thm]{Corollary}

\newtheorem{rmk}[thm]{Remark}
\numberwithin{equation}{section}

\newcommand{\ind}{\mathbbm{1}}

\newcommand{\maxload}{\mathrm{MaxLoad}}

\newcommand{\poss}{\mathrm{Poisson}}

\newcommand{\bernoulli}{\mathrm{Bernoulli}}

\newcommand{\R}{\mathbb{R}}
\newcommand{\N}{\mathbb{N}}
\newcommand{\Z}{\mathbb{Z}}

\newcommand{\E}{\mathbb{E}}

\newcommand{\cF}{\mathcal{F}}
\newcommand{\vdE}[3]{E_{#3}^{\,#1,#2}}
\newcommand{\vddE}[3]{\hat{E}_{#3}^{\,#1,#2}}
\newcommand{\vE}[3]{E_{#3}^{\,#1,#2}}
\newcommand{\vF}[3]{F_{#3}^{\,#1,#2}}
\newcommand{\vG}[3]{G_{#3}^{\,#1,#2}}
\newcommand{\vtG}[3]{\tilde{G}_{#3}^{\,#1,#2}}
\newcommand{\vH}[3]{H_{#3}^{\,#1,#2}}
\newcommand{\vU}[3]{U_{#3}^{\,#1,#2}}

\def\P{{\mathbb P}}

\begin{document}
\title{Long-term balanced allocation via thinning}
\author{Ohad N. Feldheim, Ori Gurel-Gurevich,
Jiange Li
\thanks{O. F. and O. G-G are with the Hebrew University of Jerusalem. J. L. is with the Harbin Institute of Technology.
E-mail: {\tt ohad.feldheim@mail.huji.ac.il, ori.gurel-gurevich@mail.huji.ac.il,  lijiange7@gmail.com}.}
}
\date{\today}
\maketitle

\begin{abstract}
We study the long-term behavior of the two-thinning variant of the classical balls-and-bins model. In this model, an overseer is provided with uniform random allocation of $m$ balls into $n$ bins in an on-line fashion. For each ball, the overseer could reject its allocation and place the ball into a new bin drawn independently at random. The purpose of the overseer is to reduce the \emph{maximum load} of the bins, which is defined as the difference between the maximum number of balls in a single bin and $m/n$, i.e., the average number of balls among all bins.

We provide tight estimates for three quantities: the lowest maximum load that could be achieved at time $m$, the lowest maximum load that could be achieved uniformly over the entire time interval $[m]:=\{1, 2, \cdots, m\}$, and the lowest \emph{typical} maximum load that could be achieved over the interval $[m]$, where the typicality means that the maximum load holds for $1-o(1)$ portion of the times in $[m]$. 

We show that when $m$ and $n$ are sufficiently large, a typical maximum load of $(\log n)^{1/2+o(1)}$ can be achieved with high probability, asymptotically the same as the optimal maximum load that could be achieved at time $m$. However, for any strategy, the maximal load among all times in the interval $[m]$ is $\Omega\big(\frac{\log n}{\log\log n}\big)$ with high probability. A strategy achieving this bound is provided.   

An explanation for this gap is provided by our optimal strategies as follows. To control the typical load, we restrain the maximum load for some time, during which we accumulate more and more bins with relatively high load. After a while, we have to employ for a short time a different strategy to reduce the number of relatively heavily loaded bins, at the expanse of temporarily inducing high load in a few bins.
\end{abstract}


\section{Introduction}

In the classical \emph{balls-and-bins} model, $m$ balls are independently and uniformly at random placed into $n$ bins one after another. In this paper, we are interested in the following variant, which is called the \emph{two-thinning} model. For each ball, after a uniformly random bin, which is called the \emph{primary allocation}, has been suggested, an overseer has the choice of either accepting this bin, or placing the ball into a new bin selected independently and uniformly at random, which is called the \emph{secondary allocation}. In this model, the overseer is oblivious to the secondary allocation before deciding whether to accept the primary allocation. In contrast, in the well-known \emph{two-choice} model, which was introduced in the seminal work \cite{ABKU99}, the overseer is aware of the secondary allocation and places the ball into the bin which contains fewer balls (break ties arbitrarily). 

\subsection{Main results}
 We define the \emph{load} of a bin  as the difference between the number of balls in this bin and the average number of balls among all bins. Given a two-thinning strategy $f$ (see Section~\ref{Strat:two-thin} for a formal definition), we denote by $\maxload^f(m)$ the \emph{single-time maximum load}, which is defined as the maximum load among all bins after allocating $m$ balls using the strategy $f$, and denote by $\maxload^f([m])$ the \emph{all-time maximum load}, which is the maximum of $\maxload^f(k)$ for all $k\in[m]:=\{1, 2, \dots, m\}$. In general, we can replace $[m]$ by a subset $S\subseteq [m]$, and define $\maxload^f(S)$ in a similar manner.

\begin{thm}\label{thm:single-time load discrepancy}
For all $m,n\in \N$, there exists an explicit two-thinning strategy $f:=f_{m,n}$ such that, with high probability,
$$
\maxload^{f}(m)=
\begin{cases}
\Theta\left(\sqrt{\frac{\log n}{\log\log n-2\log (m/n)}}\right) & \Omega(n)\le m \le o(n\sqrt{\log n}),\\
\Theta(\sqrt{\log n}) & m=\Theta(n\sqrt{\log n}),\\
(\log n)^{1/2+o(1)} & m=\omega(n\sqrt{\log n}).
\end{cases}
$$
Moreover, in the first two cases the maximum loads are optimal up to some multiplicative constants, while in the third case we have a lower bound of $\Omega(\sqrt{\log n})$ for all two-thinning strategies.
\end{thm}

\begin{thm}\label{thm:all-time load discrepancy}
There exists an explicit two-thinning strategy $f$ such that, with high probability, 
$$
\maxload^f([m])=
\begin{cases}
\Theta\left(\sqrt{\frac{\log n}{\log\log n-2\log (m/n)}}\right) & \Omega(n)\le m \le O(n\sqrt{\log n}),\\
\Theta\big(\big(\frac{m\log n}{n}\big)^{1/3}\big) & \Omega(n\sqrt{\log n})\leq m\leq o(n\log^2 n),\\
\Theta\big(\frac{\log n}{\log\log n}\big) &\Omega(n\log^{2} n)\le m\le n^{O(1)}.
\end{cases}
$$
Moreover, the all-time maximum load achieved by $f$ is optimal up to a multiplicative constant.
\end{thm}

For $\varepsilon>0$, we denote by $\maxload^f_\varepsilon([m])$  the \emph{$\varepsilon$-typical maximum load}, which is defined as the largest $\ell>0$ such that $\maxload^f(k)> \ell$ holds for at least $\varepsilon m$ many  $k\in[m]$. Clearly, we have $\maxload^f(m)\le \maxload^f_\varepsilon([m])\le \maxload^f([m])$.  

Theorems \ref{thm:single-time load discrepancy} and \ref{thm:all-time load discrepancy} show that for $m=O(n\sqrt{\log n})$, the difference between the optimal single-time and all-time maximum loads is at most a multiplicative constant and hence the optimal typical maximum load also has the same asymptotic behaviour. For $m=\omega(n\sqrt{\log n})$, however, there is a gap between the optimal single-time and all-time maximum loads. The next theorem shows that in this regime, the typical maximum load behaves like the single time maximum load, so is the gap between the optimal typical and all-time maximum loads. 

\begin{thm}\label{thm:typical load discrepancy}
Let $m,n\in \N$, and write $\varepsilon=e^{-\frac{1}{2}\sqrt{\log\log\log n}}$.
There exists an explicit two-thinning strategy $f=f_{m,n}$ such that for $n$ large enough and for all $m$,
$$
\maxload^f_\varepsilon([m])\le (\log n)^{\frac{1}{2}+o(1)},
$$
holds with high probability.
\end{thm}

It is worth pointing out that for $m\le n^{O(1)}$ our strategy actually governs the loads in some \textbf{predetermined}, large (i.e., $1-\varepsilon$ portion) set of times in $[m]$, with high probability (see Proposition~\ref{prop:typical-load-poly-time}).

\subsection{Discussion}

The classical balls-and-bins model and its two-choice variant have been extensively studied
in probability theory, random graph theory, and computer science. Many applications have been found
in various areas, such as hashing, load balancing and resource allocation in parallel and distributed
systems (see e.g., \cite{ABKU99}, \cite{ACMR98}, \cite{KLM92}, \cite{SEK03}, \cite{Ste96}). 
In the balls-and-bins model, it is known that for $m=\Theta(n)$, the maximum load is $(1+o(1))\frac{\log n}{\log\log n}$ with high probability, and for $m\gg n$, the maximum load is $\Theta\Big(\sqrt{\frac{m\log n}{n}}\Big)$ with high probability (see e.g. \cite{RS98}). In the seminal paper \cite{ABKU99}, Azar, Broder, Karlin and Upfal showed that in the two-choice model, for $m=\Theta(n)$, the maximum load is $\frac{\log\log n}{\log 2}+O(1)$ with high probability -- an exponential improvement over the balls-and-bins model. In fact, this phenomenon was first noticed by Karp, Luby and Meyer auf der Heide \cite{KLM92} in the context of PRAM simulations when switching from one hash function to two. In \cite{ABKU99}, the $d$-choice setting, where the overseer is given $d>2$ choices, was also considered.
In this setting, an optimal maximum load of $\frac{\log\log n}{\log d}+O(1)$ can be achieved with high probability; that is, compared with the case $d=2$, the performance improves by merely a multiplicative factor for larger values of $d$. We refer the reader to the survey \cite{MRS01} for more details about the two-choice model. 

The long-term behavior of the two-choice model, in which case the number of balls $m$ can be super linear in $n$, proved to be more challenging. In the seminal paper \cite{BCSV06}, Berenbrink, Czumaj, Steger and V\"{o}cking showed that for arbitrarily large $m$, one can achieve the maximum load of $\frac{\log\log n}{\log 2}+O(1)$ with high probability. A simpler proof of this result with a weaker lower order term was given by Talwar and Wieder  \cite{TW14}. Since this result is achieved via a single greedy strategy at all times, a simple union bound argument implies that this strategy also maintains this bound as the all-time and the typical maximum loads for $m$ polynomially large in $n$. 

Different variants of the two-choice model have been studied under weaker constraints from practical considerations.
These include load balancing with limited memory \cite{AGGL10, BM12, MPS02}, relaxation on the possible pairs the overseer may select from (known as two choices on graphs) \cite{KP06, PTW15} and a hypergraph variant of it \cite{God08}. Other relaxations include bins with different selection probabilities \cite{BBFN14} and balls with different weights \cite{TW07}. An important purpose of this course of study is to understand the robustness of the load reduction achieved by the power of two choices, understanding the impact of constraints on memory, information and choice patterns. Particularly, Peres, Talwar and Wieder \cite{PTW15} studied the setting of two choices with errors, which is known as the $(1+\beta)$-choice model. In this setting, with probability $\beta$ the ball is allocated using the two-choice model, and with probability $1-\beta$ the ball is assigned to a random bin as in the balls-and-bins model. The authors showed that, irrespective of $m$, the gap between the maximum load and the average is $O\big(\frac{\log n}{\beta}\big)$. Since this result is irrespective of $m$, a simple union bound argument implies that this bound is also valid for the all-time and the typical maximum loads for $m$ polynomially large in $n$.

The \emph{two-thinning} variant is a different relaxation of the two-choice model which arises naturally in a statistical scenario, where one collects samples one-by-one and is allowed to decide whether to keep each sample or not, under the constraint of never discarding two consecutive samples. In \cite{DFGR19}, Dwivedi, Ramdas and the first two authors showed that two-thinning could reduce the discrepancy of a sequence of random points selected independently and uniformly at random
from the interval $[0, 1]$ to be near optimal. The first two authors studied the two-thinning variant of the balls-and-bins model in \cite{FGG18}. They showed that for $m=\Theta(n)$, the optimal maximum load is $(2+o(1))\sqrt{\frac{2\log n}{\log\log n}}$ with high probability, a polynomial improvement over the balls-and-bins model. Hence, this model is in some sense more powerful than the $(1+\beta)$-choice model.
The authors also conjectured the upper bound $O\Big(\sqrt{\frac{\log n}{\log\log n}}\Big)$ for all $m\gg n$. Los and Sauerwald \cite{LS21} recently disproved this conjecture by showing a lower bound of $\Omega(\sqrt{\log n})$ for $m=\Theta(n\sqrt {\log n})$, a bound which we show here holds for all $m=\Omega(n\sqrt {\log n})$. They also showed that a load of $\Omega(\log n/ \log \log n)$ holds for at least $\Omega(n\log n / \log\log n)$ times in $[1, n\log^2 n]$. Our work sheds more light on this phenomenon. 
The results in \cite{FGG18} were extended by the first and third authors \cite{FL20} to the $d$-thinning setting and the optimal maximum load of $(d+o(1))\big(\frac{d\log n}{\log\log n}\big)^{1/d}$ could be achieved with high probability.

Another relaxation of the two-choice model was recently studied by Los and Sauerwald \cite{LS21}. They considered the situation that each ball is offered two random bins and is allowed to send up to $k$ binary queries, each to one of
the two bins. In one model, it inquires whether the absolute load crosses some threshold, and in the other model, it inquires whether the number of bins with loads higher than that of the queried bin is greater than some percentile. The $k=1$ case is equivalent to our two-thinning model. They showed that in both models a maximum load of $O(k (\log n)^{1/k})$ can be achieved with high probability.

Here, we study the long-term behavior of the two-thinning model. Our discussions above and Theorems \ref{thm:single-time load discrepancy} and \ref{thm:typical load discrepancy} show that, in the balls-and-bins and the two-choice models, the optimal single-time and the typical maximum loads are asymptotically nearly identical. However, in contrast with these two models, there is a big gap between the optimal typical and the all-time maximum loads in the two-thinning setting. We attribute this difference to the fact that in the two-thinning setting, short periods of relative high maximum loads are necessary for the process to ``release steam'' with the benefit of arriving at low maximum loads at the end of these periods. A comparison of the maximum loads in these three models is given in the following table. 

\begin{table}[h!]
\def\arraystretch{1.7}
\centering
\begin{tabular}{||p{23mm} | c | c | c | c||} 
 \hline
 \hline
 \phantom{\Big(}            &  $m=\Theta(n\log^\alpha n)$    & $\maxload^f(m)$  &   $\maxload_\varepsilon^f([m])$  &  $\maxload^f([m])$ \\
 \hline\hline
 \multirow{2}{*}{Balls-and-bins}  & $\alpha<1$ &  $\Theta\big(\frac{\log n}{\log\log n}\big)$   & $\Theta\big(\frac{\log n}{\log\log n}\big)$ & $\Theta\big(\frac{\log n}{\log\log n}\big)$ \\
        & $\alpha\ge 1$ & $\Theta\Big(\sqrt{\frac{m\log n}{n}}\Big)$   & $\Theta\Big(\sqrt{\frac{m\log n}{n}}\Big)$ & $\Theta\Big(\sqrt{\frac{m\log n}{n}}\Big)$\\[3pt]
 \hline
 \multirow{4}{*}{Two-thinning} & $\alpha\in[0,\tfrac12)$ & $\Theta\left(\sqrt{\frac{\log n}{\log\log n}}\right)$ & $\Theta\left(\sqrt{\frac{\log n}{\log\log n}}\right)$ &$\Theta\left(\sqrt{\frac{\log n}{\log\log n}}\right)$\\
                               & $\alpha=\tfrac12$ & $\Theta(\sqrt{\log n})$ & $\Theta(\sqrt{\log n})$ & $\Theta(\sqrt{\log n})$\\ 
                               & $\alpha\in(\tfrac{1}{2},2)$ &
                             $(\log n)^{1/2+o(1)}$& $(\log n)^{1/2+o(1)}$ & $\Theta\big((\log n)^{\frac{1+\alpha}{3}}\big)$\\
                             & $\alpha\geq 2$ &
                             $(\log n)^{1/2+o(1)}$& $(\log n)^{1/2+o(1)}$ & $\Theta\big(\frac{\log n}{\log\log n}\big)$\\
                             \hline
 Two-choice & $\alpha\ge  0$ & $\Theta(\log\log n)$ & $\Theta(\log\log n)$ & $\Theta(\log\log n)$\\
 \hline
 \hline
\end{tabular}
\caption{A comparison of the single-time, all-time and typical maximum loads. Here, we write $m=\Theta(n\log^\alpha n)$ and select $\varepsilon=o(1)$, where $\alpha$ is allowed to depend on $n$, but some of the results require $m$ to be at most polynomially large in $n$. In the two-thinning model, the results for $\alpha=0$ appear in \cite{FGG18} and the lower bound for $\alpha=1/2$ appears in \cite{LS21} and the remaining results are new. In the two-choice model, the results follow from \cite{BCSV06}.  In the balls-and-bins model, the results are classical (see e.g. \cite{RS98}).}
\label{table:1}
\end{table}
\pagebreak

\subsection{Upper bound strategies and lower bound techniques}

In the following, we give a brief description of our strategies that achieve the upper bounds in our main results as well as techniques for establishing the lower bounds. We write $m=nt$ for $t\in\N$. Different strategies are required for values of $t$ in different ranges. 

\textbf{The single-time maximum load}. For $t=O(\sqrt{\log n})$, our upper bound is achieved by the \emph{threshold strategy} employed in \cite{FGG18}, which retries a ball if the number of primary allocations accepted by the suggested bin reaches certain threshold. For $t\geq \omega(\sqrt{\log n})$, the threshold strategy alone is not sufficient since the optimal choice of the threshold would be $t+\Theta((t\log n)^{1/3})$ and this yields a maximum load of $O((t\log n)^{1/3})$, which is much larger than our desired upper bound $(\log n)^{1/2+o(1)}$. Instead, we divide the process into multiple shorter stages and, in each stage, apply the threshold strategy with a smaller threshold. It is likely that this will cause more retries and even a temporarily high maximum load. To prevent this from causing a high load at the end of the process, we always retry a ball if its primary allocation is a heavily loaded bin. The number of retries caused by this requirement is relatively small since the number of heavily loaded bins is small. This, together with a careful selection of time lengths of the stages, enables us to achieve the maximum load of $(\log n)^{1/2+o(1)}$ at the end of the process. We call this strategy the \emph{multi-stage threshold strategy}. For $t\geq \omega(\log n)$, we need another ingredient in the form of a \emph{drift strategy}. 
Under this strategy we retry a ball with positive probability if its primary allocation has a positive load,
and surely if its load is very high. This creates a drift in the load of positively loaded bins towards zero, resulting in a load distribution with exponential tail and a maximum load of $\Theta(\tfrac{\log n}{\log\log n})$ (in some sense, this is an improvement of a similar strategy in \cite{DFGR19}). For $t\geq \omega(\log n)$, we first apply this drift strategy up to $\Theta(\log n)$ time before the end, and then apply the aforementioned multi-stage threshold strategy  to allocate the remaining $\Theta(n\log n)$ balls. Our lower bound follows from the simple observation that if we retry too many balls, the secondary allocations will cause a high maximum load, and if we retry too few balls, the primary allocations will cause a high maximum load.

\textbf{The all-time maximum load}. Our upper bound strategy is a time-adaptive version of the threshold strategy for the single-time maximum load, which we call a \emph{relative threshold strategy}. 
We use a threshold strategy where the threshold after throwing $tn$ balls, is $t+\ell$ for a fixed $\ell>0$. This strategy results in a uniform control of the maximum load throughout the process. Our lower bound follows from the observation that a uniform bound on the maximum load in the process upper-bounds the number of retries in the allocation of each batch of $n$ balls, and hence -- the total number of retries in the entire process. Subject to this constraint, we consider the maximum load after all balls have been allocated and show it to be large. 

\textbf{The typical maximum load}. As mentioned before Theorem \ref{thm:typical load discrepancy}, it suffices to consider the case $t\geq \omega(\sqrt{\log n})$. For $\omega(\sqrt{\log n})\leq t\leq O(\log n)$, we apply a multi-scale strategy. Each scale consists of iterations of two strategies. In the first, longer part of each iteration, we apply the strategy of a smaller scale, while in the second, shorter part,  we use a different regulating strategy. The strategy in the smallest scale is simply the relative threshold strategy, while the regulating strategy is the multi-stage threshold strategy used to control the single-time maximum load. These regulating segments play the role of ``releasing steam'' from the process -- although they result in a high maximum load for a short period of time,  they yield good control the maximum loads at the end of these segments, so that we can re-initiate the next iteration. For $t\geq \omega(\log n)$, we iterate over long segments of this strategy, separated by short segments of the drift strategy followed by the multi-stage threshold strategy.

\subsection{Outline}
This paper is organized as follows. In the next section, we introduce two-thinning strategies that are used to achieve the desired bounds on three types of maximum loads as stated in Theorems \ref{thm:single-time load discrepancy}, \ref{thm:all-time load discrepancy} and \ref{thm:typical load discrepancy}. We provide some preliminary tools in Section 3, which are used in  the analysis of different two-thinning strategies and the proofs of the main results. The proof of Theorem \ref{thm:single-time load discrepancy} on the single-time maximum load is provided in Section \ref{sec:single-time load upper bound} (upper bound) and Section \ref{sec:single-time load lower bound} (lower bound). The proof of Theorem \ref{thm:all-time load discrepancy} on the all-time maximum load is provided in Section \ref{sec:all time upper bound} (upper bound) and \ref{sec:all time lower bound} (lower bound). In the last section, we prove Theorem \ref{thm:typical load discrepancy} on the $\varepsilon$-typical maximum load.


\numberwithin{thm}{section}

\section{Strategies for taming the maximum loads}\label{sec: strategies}


In this section, we provide two-thinning strategies that are used to control the maximum loads. We give the formal definition of a two-thinning strategy in Section~\ref{Strat:two-thin}, and provide an alternative, indirect way of describing a two-thinning strategy in Section~\ref{Strat:realizable}. Several basic two-thinning strategies are given in Section~\ref{sec:base strat}, which are building blocks of more advanced strategies in Section~\ref{sec:combin}. We provide an outline of how these strategies are used to obtain the main theorems in Section \ref{sec:opt strat}.

\subsection{Two-thinning strategy}\label{Strat:two-thin}


A \emph{decision strategy} is a function
$$ 
f:[n]\times [0,1] \to \{1,2\},
$$
which, given a primary allocation and an external random number in $[0,1]$, decides whether to accept (denoted by 1) or reject (denoted by 2) the suggested allocation.
Given $Z_1,Z_2$, a pair of independent random variables, uniform on $[n]$ and $U$ uniform on $[0,1]$, we can consider the \emph{output} of a decision strategy given by $Z_{f(Z_1,U)}$.

A \emph{thinning strategy} $f$ is a sequence of functions $\{f_k\}_{k\in \N}$, where the function
$$
f_k:\left([n]\times [n]\times \{1,2\}\right)^{k-1} \times [n] \times [0,1] \to \{1,2\},
$$ 
given the history $\mathcal H\in ([n]\times [n]\times\{1,2\})^{k-1}$ of the process up to time $k-1$ (that is, the primary allocations, the final allocations and the decisions of the first $k-1$ balls), the primary allocation at time $k$ and an external random number in $[0,1]$, decides whether to accept or reject the suggested allocation. Hence, given the history of the process, the thinning strategy provides a decision strategy for the next allocation.

A thinning strategy $f$ generates the \emph{decisions} sequence  $\{D_k\}_{k\in\N}$ and the \emph{allocations} sequence  $\{Z_k\}_{k\in\N}$ in the following way. We denote by $\{Z^1_k\}_{k\in\N}$ and $\{Z^2_k\}_{k\in\N}$ two independent sequences of random variables, which are independent and uniformly distributed in $[n]$. Here, $Z^1_k$ represents the primary allocation of the $k$-th ball, while $\{Z^2_k\}_{k\in\N}$ is used as a pool of secondary allocations. Set $R_0=0$ and we denote by $R_k$ the number of rejections among the first $k$ primary allocations. Let $\{U_k\}_{k\in\N}$ be a collection of uniform  random variables on $[0,1]$. For the $k$-th allocation, we can inductively define
\begin{align}
D_k&= f_k\left(\{Z^1_j\}_{j\in[k-1]},\{Z_j\}_{j\in[k-1]},\{D_j\}_{j\in[k-1]},Z^1_{k},U_k\right), \notag\\
R_{k}&=R_{k-1}+D_k-1, \label{eq:def-retry}\\
Z_k&=\begin{cases}
Z^1_k & \text{if }D_k=1,\notag\\
Z^2_{R_{k}} & \text{if }D_k=2.\notag
\end{cases}
\end{align}
In other words, we look at the history $\mathcal H$ of the process up to time $k-1$ and at the primary allocation $Z^1_k$ at time $k$ along with an additional source of randomness $U_k$ and apply $f$ to determine whether to accept $Z^1_k$ or not. If we reject $Z^1_k$, we will then allocate the $k$-th ball to the next unused secondary allocation $Z^2_{R_k}$ from our pool.

We allow bins to start with some initial loads $\{L_i(0)\}_{i\in [n]}$ satisfying $\sum_{i=1}^{n}L_i(0)=0$, where $L_i(0)$ is the initial load of the $i$-th bin. Let $m\in \N$ and let $i\in[n]$. The \emph{load} of bin $i$ after allocating $m$ balls using the thinning strategy $f$ is defined as 
\begin{equation}\label{eq:load of bin i}
L_i^f(m) =L_i(0)+\sum_{k=1}^m \ind_{\{Z_k=i\}}-\frac{m}{n}.
\end{equation}
For any $M\subseteq [m]$, we define
\begin{equation}\label{eq:primary and secondary load}
\begin{split}
L_{1, i}^f(M) &=\big|\big\{k\in M: Z_k^1=i ~\text{and}~ D_k=1\big\}\big|,\\
L_{2, i}^f(M) &=\big|\big\{k\in M: Z_{R_k}^2=i\big\}\big|.
\end{split}
\end{equation}
Hence, $L^f_{1, i}([m])$ represents the number of primary allocations accepted by bin $i$ after allocating $m$ balls, and $L_{2, i}([m])$ represents the number balls that bin $i$ receive from secondary allocations. It is clear that $L_i^f(m)=L_i(0)+L^f_{1, i}([m])+L^f_{2, i}([m])-m/n$. For any $S\subseteq[n]$ and $\ell\in \R$, we define
\begin{equation}\label{eq:maxload level set}
\phi_S^\ell(m)=\left |\left\{i\in S: L_i^f(m)\geq \ell\right\}\right |,
\end{equation}
which is the number of bins in $S$ with loads at least $\ell$ after allocating $m$ balls using the thinning strategy $f$, and
\begin{equation}\label{eq:primay allocation level set}
\psi_S^\ell(M)=\left |\left\{i\in S: \sum_{k \in M} \ind_{\{Z_k^0=i\}}\geq \ell\right\}\right |,
\end{equation}
which is the number of bins in $S$ that are suggested as primary allocations at least $\ell$ times during the allocations of balls in $M$. The maximum load over a set of bins $S$ after allocating $m$ balls using the thinning strategy $f$ is defined as
\begin{equation}\label{eq:maxload}
\maxload_S^f(m)=\max_{i\in S}L^f_i(m).
\end{equation}
We will omit the index $S$ in these notations when $S=[n]$. For any $M\subseteq [m]$, we define the  maximum load achieved during the allocation of balls in $M$ as
\begin{equation}\label{eq:all-time-S-maxload}
\maxload^f(M)=\max_{k\in M}\maxload^f(k).
\end{equation}
The $\varepsilon$-typical maximum load $\maxload_\varepsilon^f(M)$ over the set $M$ is defined as 
\begin{equation}\label{eq:typical-S-maxload}
\maxload_\varepsilon^f(M)=\max\big\{\ell>0: \big|\big\{k\in M: \maxload^f(k)\geq \ell\big\}\big|\geq \varepsilon|M|\big\}.
\end{equation}







\subsection{A realizability criterion}\label{Strat:realizable}

Under certain circumstances, instead of providing an explicit, formal description of a two-thinning strategy, we only show the realizability. The following result provides a criterion for a probability distribution to be realized by some two-thinning strategy.

\begin{lem}\label{lem:realizability}
Any probability distribution $\mathcal{P}$ on $[n]$ with probability mass function $\{p_i\}_{i\in [n]}$ for which
$$
\frac{c}{n}\leq p_i\leq \frac{1+c}{n}
$$
for some $c>0$ and  for every $i\in[n]$, is the distribution of the output of a two-thinning decision strategy.
\end{lem}

\begin{proof}
Let $Z_1, Z_2, U$ be independent random variables uniformly distributed in $[n]$. Here, $U$ is the external randomness.  We define the two-thinning function $f: [n]\times [0, 1]\to \{1, 2\}$ as
$$
f(i, u)=
\begin{cases}
1, \ np_i-c\geq u,\\
2, \ np_i-c<u.
\end{cases}
$$
Let $Z=Z_{f(Z_1,U)}$ be the output of $f$. For any $i\in [n]$, we have
\begin{align*}
\P(Z=i) &= \P(Z_1=i, f(Z_1, U)=1)+\P(Z_2=i, f(Z_1, U)=2)\\
&=\P(Z_1=i)\P(f(i, U)=1)+\P(Z_2=i)\sum_{j=1}^n\P(Z_1=j)\P( f(j, U)=2)\\
&= \frac{1}{n}\cdot (np_i-c)+\frac{1}{n}\cdot \frac{1}{n}\sum_{j=1}^n(1+c-np_j)\\
&= p_i.
\end{align*}
The second identity follows from the joint independence among $Z_1, Z_2, U$.
\end{proof}

\subsection{The basic strategies}\label{sec:base strat}

Here, we introduce some basic two-thinning strategies, which are building blocks of more advanced strategies in the next section. The first two thinning strategies are  deterministic and rather natural.

\textbf{The threshold strategy.} The \emph{$\ell$-threshold strategy} accepts the primary allocation of a given ball whenver the suggested bin has accepted thus far less than $\ell$ primary allocations. 
In other words,
$$f_k(\mathcal H, i,u)=\begin{cases}
1 & \text{if }L^f_{1,i}(k)<\ell,\\
2 & \text{if }L^f_{1,i}(k)\ge \ell.
\end{cases}$$
This strategy is used to control the single-time maximum load of allocating $O(n\sqrt {\log n})$ balls.

\textbf{The relative threshold strategy.} The \emph{$\ell$-relative threshold strategy} accepts the $k$-th primary allocation if the suggested bin has accepted less than $\ell+\frac{k-1}n$ primary allocations or if the load of the suggested bin is below $- \log n$. In other words,
$$f_k(\mathcal H, i,u)=\begin{cases}
1 & \text{if }L^f_{1,i}(k)<\ell+\frac{k-1}n\text{ or }L^f_i(k)<- \log n,\\
2 & \text{if }L^f_{1,i}(k)\ge\ell +\frac{k-1}n\text{ and }L^f_i(k)\ge - \log n.
\end{cases}$$
This strategy is designed to control the all-time maximum load of allocating $o(n \log^2 n)$ balls.

\textbf{The drift strategy}. The third strategy relies on a coupling of the allocation process and a continuous time random process. This strategy can be used to achieve appropriate initial conditions for other strategies as it is very robust and can rather quickly reduce the load vector to a stationary distribution with an exponential tail. We denote by $\{X_i(t)\}_{i\in [n]}$ a collection of  independent regular point processes with initial values $X_i(0)=L_i(0)$ and conditional intensity functions 
\begin{equation}\label{eq:intensity xi}
\lambda_i(t)=
\begin{cases}
1+\theta, & X_i(t)<t,\\
1-\theta, & X_i(t)\geq t.
\end{cases}
\end{equation}
Write $X(t)=\sum_{i=1}^nX_i(t)$. 
We define the random process $\{Z_k\}_{k\in\N}$ as follows. For any $k\in\N$, we set 
\begin{equation}\label{eq:def zk}
\text{$Z_k=i$ if the $k$-th point of $X(t)$ for $t>0$ is a point of $X_i(t)$}.    
\end{equation}
We will show that, conditioned on $Z_1,\dots,Z_{k-1}$, the variable $Z_k$ meets the conditions of Lemma~\ref{lem:realizability}.
Hence $\{Z_k\}_{k\in\N}$ is realizable as the  output of a two-thinning strategy.
We call this strategy the \emph{$\theta$-drift strategy}.

We write $\cF_t$ for the natural filtration of $X(t)$ and denote by  $T_k=\inf\{t : X(t)=k\}$.
To see that the conditions of Lemma \ref{lem:realizability} are indeed satisfied, it suffices to show that there exists some $c>0$ such that
\begin{equation}\label{eq:realizability}
\frac{c}{n}\leq \P(Z_k=i\ |\  Z_1, \cdots, Z_{k-1},\cF_{T_{k-1}})\leq \frac{1+c}{n}
\end{equation}
holds for all $k\in\N$ and all $i\in [n]$. By the definition of $\{Z_k\}_{k\in\N}$, we have
$$
\frac{1-\theta}{n(1+\theta)}=\frac{\inf\limits_{t\geq 0}\lambda_i(t)}{n\max\limits_{j\in [n]}\sup\limits_{t\geq 0}\lambda_j(t)} 
\leq \P(Z_k=i\ |\  Z_1, \cdots, Z_{k-1},\cF_{T_{k-1}})
\leq\frac{\sup\limits_{t\geq 0}\lambda_i(t)}{n\min\limits_{j\in [n]}\inf\limits_{t\geq 0}\lambda_j(t)}\leq\frac{1+\theta}{n(1-\theta)}.
$$
One can check that the criterion \eqref{eq:realizability} holds for all $0<\theta\leq \sqrt 5-2 $.

\textbf{A varying drift strategy}. Our forth strategy is a modified drift strategy where the downwards drift is extremely strong for bins with loads above certain level $\ell$. We denote by $\{X_i(t)\}_{i\in[n]}$ a collection of independent regular point processes with initial values $X_i(0)=0$ and conditional intensity functions given by
\begin{equation}\label{eq:intensity xi a}
\lambda_i(t)=
\begin{cases}
1+\theta_1, & X_i(t)<t,\\
1-\theta_2, & t\leq X_i(t)\leq t+\ell,\\
\theta_3, & X_i(t)> t+\ell.
\end{cases}
\end{equation}
Here, we set $\theta_1=\theta_2=\frac{1}{\sqrt{\log n}}$ and $\theta_3=\frac{12}{\sqrt{\log n}}$.  We write $X(t)=\sum_{i=1}^nX_i(t)$.  For any $k\in\N$, we set $Z_k=i$ if the $k$-th point of $X(t)$ for $t>0$ is a point of the process $X_i(t)$. We write $\cF_t$ for the natural filtration of $X(t)$ and denote by $T_k=\inf\{t : X(t)=k\}$. 
Unlike in the case of the drift strategy, in certain situations, the distribution of $Z_k$ given $Z_0,\dots,Z_{k-1}$ is not the output of any two-thinning decision strategy. However, as the next lemma shows, this does not happen as long as the number of bins with very high load is not too large. We call the strategy which realizes $Z_k$ for as long as possible (and, say, accepts all primary allocations from that time and on, for the sake of completion), the \emph{$\ell$-varying drift strategy}.


\begin{lem}\label{lem:realizibility criterion}
For sufficiently large $n$, for any $k\in \N$, if 
\begin{equation}\label{eq:realizibility criterion}
\big|\big\{i\in[n]: X_i(T_{k-1})>T_{k-1}+\ell\big\}\big|\leq \frac{n}{\sqrt{\log n}},
\end{equation}
then the distribution of $Z_k$ given $Z_0,\dots,Z_{k-1}$ can be realized by a two-thinning decision strategy.
\end{lem}

\begin{proof}
We need to verify that the distribution of $Z_k$ given $Z_0,\dots,Z_{k-1}$ satisfies the condition of Lemma \ref{lem:realizability}. To this end, it is enough to show that there exists some $c>0$, which could depend on $n$, such that for sufficiently large $n$, for all $i\in [n]$ we have, 
\begin{equation}\label{eq:realizability criterion}
\frac{c}{n}\leq \P(Z_k=i~|~ Z_1, \cdots, Z_{k-1},\cF_{T_{k-1}})\leq \frac{1+c}{n}.
\end{equation}
Denote $n_0=|\{i\in[n]: X_i(T_{k-1})>T_{k-1}+\ell\}|$. Then, the condition \eqref{eq:realizibility criterion} says that $n_0\leq \frac{n}{\sqrt{\log n}}$. By the definition of $Z_k$, we have
$$
 \frac{\theta_3}{(n-n_0)(1+\theta_1)+n_0\theta_3} 
\leq \P(Z_k=i~|~ Z_1, \cdots, Z_{k-1},\cF_{T_{k-1}})
\leq\frac{1+\theta_1}{(n-n_0)(1-\theta_2)+n_0\theta_3}
$$
Using the fact that the denominators above are maximized when $n_0=0$ and are minimized when $n_0=\frac{n}{\sqrt{\log n}}$, we obtain
$$\frac{6}{n\sqrt{\log{n}}}\le\P(Z_k=i~|~ Z_1, \cdots, Z_{k-1},\cF_{T_{k-1}})\le
\frac{1+\frac{1}{\sqrt {\log n}}}{n\big(1-\frac{1}{\sqrt {\log n}}\big)^2}\leq\frac{1}{n}\left(1+\frac{4}{\sqrt{\log n}}\right)
$$ 
for all $n$ sufficiently large. Thus, inequality \eqref{eq:realizability criterion} holds with $c=\frac{4}{\sqrt{\log n}}$.
\end{proof}

We shall see in Section~\ref{subsection:case 2} that for $\ell=\frac{2\log n}{\log\log n}$, the condition of Lemma~\ref{lem:realizibility criterion} is indeed satisfied with high probability over polynomially long time in $n$.

\subsection{Combinations of the basic strategies}\label{sec:combin}

In many scenarios, particularly when the number of balls is large, we need to adjust and combine the basic strategies in an appropriate way to obtain the upper bounds in our main results. The following are several such combinations. 


\textbf{The multi-stage $(t,L_0,\ell)$-threshold strategy}.
Set $t_0=0$ and $k=\big\lfloor\frac{\log\log n}{3\log\log\log n}\big\rfloor$. We divide the process into $k$ stages, where the $i$-th stage proceeds from time $ n t_{i-1}$ to time $n t_i$, where the definition of $t_i$ as a function of $t$ is given at the end of this description. We write $H_0$ for the set of bins with loads greater than $L_0$ at time $t_0$. We inductively define $H_i$ as the set of bins in $(\cup_{j=1}^{i-1}H_j)^c$ (or in $H^c_0$ in the case $i=1$) whose loads at the end of the $i$-th stage are at least $L_0+2i\ell$. 
Then our strategy can be stated as follows. In the first stage, we retry a ball if its primary allocation bin has a load of at least $-\log n$ and either it is in $H_0$ or it has accepted $t_1-t_0+\ell$ primary allocations in the first stage so far. In $i$-th stage for $i\geq 2$, we retry a ball if its primary allocation bin has a load of at least $-\log n$ and either it is in $\cup_{j=1}^{i-1}H_j$, or it is a bin that has accepted $t_i-t_{i-1}+\ell$ primary allocations during the $i$-th stage so far.

Now we conclude the description with the definitions of $\{t_i\}_{i\in [k]}$.  Denote $\alpha=\frac{\log t}{\log\log n}$.  Given $\eta\in[0,\tfrac{\alpha-1/2}{4k-2}]$, we set $\beta=\alpha+\eta$, $\varepsilon=\frac{2\beta-1}{2(k+1)}$, and $\beta_i=\beta-\frac{(2\beta-1-\varepsilon)i}{2k+1}$. We then define $t_i=\lfloor t-\log^{\beta_i}\! n\rfloor$ for $1\leq i\leq k-1$, and $t_k=t$.

\begin{rmk}
It might be worthwhile to point out that after the first stage, we do not retry primary allocations that are bins in $H_0$ unless they consist of bins with load at least $-\log n$ and already accepted $\ell$ primary allocations more than the average in the current stage. Hence, the initial set of heavily loaded bins $H_0$ will play the same role as any other bins from stage two and on.
\end{rmk}

This multi-stage threshold strategy is designed to control the single-time maximum load for time $t\geq \omega(\sqrt{\log n})$, in which case the threshold strategy alone is not sufficient. Indeed, optimizing the choice of the threshold in the threshold strategy gives $t+\Theta((t\log n)^{1/3})$, which, in turn, yields a maximum load of $O((t\log n)^{1/3})$; much larger than the desired upper bound $(\log n)^{1/2+o(1)}$. Hence, we divide the process into multiple shorter stages and in each stage apply the threshold strategy with a smaller threshold. This is likely to cause more retries and even a temporarily higher maximum load. To prevent this from causing high load at the end of the process, we identify at the beginning of every stage heavily loaded bins ($H_i$) and from this time and on retry a ball if its primary allocation is one of these. The number of retries caused by this requirement is relatively small since the number of heavily loaded bins is small. This, together with a careful selection of time lengths of the stages, will effectively reduce the maximum load to $(\log n)^{1/2+o(1)}$ at the end of the process.

A sketch of the analysis of the strategy is as follows. We first control the maximum load after the first stage, and the number of relatively heavily loaded bins at the end of it (i.e., $H_1$).  In every subsequent stage $i$ there are two causes for retries: either the suggested bin already accumulated $\ell$ primary allocations more than the average in this stage, or it was marked as heavily loaded in previous stages (i.e., it is in $\cup_{j=1}^{i-1}H_{j}$). By inductive bounds on these, we are able to control the number of such retries. For a bin to be included into $H_{i}$, it must accumulate at least $2\ell$ allocations above average, so that at least $\ell$ of them are secondary. Using binomial estimates we can control the number of such bins with high probability and establish our bound on $H_i$. Similar computations also allow us to control the maximum load in bins $\cup_{j=1}^{i}H_{j}$, taking advantage of the negative drift of the load in $\cup_{j=1}^{i-1}H_{j}$, caused by the fact that they are always rejected as primary allocations (except if the load is already lower than $-\log n$).

\textbf{The drift-multi-stage $(\theta,t',t,L_0,\ell)$-threshold.}
This strategy is a combination of the drift strategy and the multi-stage  threshold strategy. It is designed 
 to control the single-time maximum load for $t\gg \log n$. This is simply done by applying the $\theta$-drift strategy up to time $t'$ followed by the multi-stage $(t,L_0,\ell)$-threshold strategy starting at time $t'$ and ending at time $t'+t$.

\textbf{The $Q$-multi-scale strategy}. This strategy is designed for controlling the typical maximum load for about $n(\log n)^{1+o(1)}$ time. The strategy is formed by multiple scales, each of which extends the previous one and consists of multiple iterations of the previous scale strategy separated by a different regulating strategy. Whenever we initiate a new strategy at some time, we treat this time point as the initiation time and the current loads as the initial loads for the new strategy. To avoid countless rounding operations, each strategy is applied for a not-necessarily integer time, and our policy is that if an integer point falls within the time domain of a strategy, then this strategy is applied to it.

We now give the exact description of the strategy, which is accompanied by an algorithmic description and a demonstration of the first three scales in Figure~\ref{fig:multi-key-fig}. 
We postpone the technical definitions of the parameters $L>0, k\in\N, \{\alpha_i,\alpha'_i,\ell_i\}_{i\in\N}$ after the description. We write $N_{i}=\lceil\frac{L}{3k\ell_i}\rceil$ and $Q^{i,j}=(2k+1)(j-1)\ell_i$.
In the first scale, we simply apply the $L$-relative threshold strategy up to time $n\lfloor\log^{\alpha_1} \!n\rfloor$.
In the second scale, we apply $N_1$ iterations of the first scale strategy (the last iteration may be incomplete) and the $j$-th iteration is followed by the multi-stage $(\log^{\alpha_1'}\!n,Q+Q^{1, j}+\ell_1,\ell_1)$-threshold strategy. The value of $Q$ in the $j$-th iteration of the first scale strategy is increased by $Q^{1, j}$. Generally, in the $(i+1)$-th scale, we apply $N_{i}$ iterations of $i$-th scale strategy and the $j$-th iteration is followed by  the multi-stage $(\log^{\alpha_i'}\!n,Q+Q^{i, j}+\ell_i,\ell_i)$-threshold strategy. In the $j$-th iteration, all values of $Q$ in the nested multi-scale strategies are increased by $Q^{i, j}$ (in comparison with the value of $Q$ in the current scale).

The technical definitions of the aforementioned parameters are given as follows. We set $\alpha_1=\frac{1}{2}+\frac{2}{\lfloor\sqrt{\log\log\log n}\rfloor+1/4}$, $L=(\log n)^{\frac{1+\alpha_1}{3}}$ and $k=\big\lfloor\frac{\log\log n}{3\log\log\log n}\big\rfloor$. We inductively define the sequences $\{\alpha_i,\alpha'_i,\ell_i\}_{i\in\N}$ via the following equations
\begin{align}
\varepsilon_i&=\frac{2\alpha_i-1}{2(k+1)},\notag \\
\ell_i&=(\log n)^{\frac{1}{2}+\frac{\alpha_i-1/2+k\varepsilon_i}{2k+1}}, \notag\\
\alpha_i'&=\alpha_i-\frac{1}{5}\cdot\frac{2\alpha_i-1-\varepsilon_i}{2k+1},
\label{eq:alpha-i-prime-iteration}\\
\log^{\alpha_{i+1}} \!n&=N_i(\lfloor\log^{\alpha_i}\!n\rfloor+\lfloor\log^{\alpha_i'}\! n\rfloor).\label{eq:alpha-i-iteration}
\end{align}
According to the description of our strategy, the first part of each iteration runs for $n\log^{\alpha_i}\! n$ time, and the second part runs for $n\log^{\alpha_i'}\!n$ time, so that the $(i+1)$-th scale runs for $n\log^{\alpha_{i+1}}\!n$ time in total.  

The idea behind this strategy is as follows. In each scale of the strategy, most of the time we apply the lower scale strategy, which yields a good control of the typical maximum load. However, the number of bins with loads close to the threshold will accumulate along the time. In order to mitigate this effect, we need to apply the multi-stage threshold strategy with a low threshold for a short period of time. This enables us to dramatically reduce the number of such relatively high  loaded bins at the end of each regulating period, although it is possible that during these regulating periods, certain bins may temporarily accumulate very high loads. Once the regulating period is over, the small number of relatively high load bins allows us to iterate the lower scale strategy once again.

In the following figure, we provide an algorithmic description of the $Q$-multi-scale strategy and a demonstration of the first three scales of the strategy.


\begin{algorithm}
\caption{\emph{$Q$-multi-scale} (Scale=$i+1$) }\label{alg:cap}
\begin{algorithmic}
\If{$i=0$}
    \State \textbf{Run} \emph{$L$-relative threshold} for $\log^{\alpha_1} n$ time
\Else
    \For{$j=1$ to $N_{i}$}:
        \State \textbf{Run} \emph{$\big(Q+Q^{i, j}\big)$-multi-scale}\big($i$\big) for $\log^{\alpha_{i}}\!n$ time 
        \State \textbf{Run} \emph{multi-stage} $(\log^{\alpha_i'}\!n,Q+Q^{i,j}+\ell_i,\ell_i)$\emph{-threshold} 
    \EndFor
\EndIf
\end{algorithmic}
\end{algorithm}
\begin{figure}[h!]
 \centering
 \vspace{-10pt}
  \caption{\textbf{Above:} an algorithmic description of the $Q$-multi-scale strategy. \textbf{Below:} the first three scales of this strategy. The first scale is the $L$-relative threshold strategy. The second scale consists of $N_1$ iterations, the $j$-th of which 
 incorporates the strategy of the first scale followed by the multi-stage $(\log^{\alpha'_1}n,Q+Q^{1, j}+\ell_1, \ell_1)$-threshold strategy. The third scale consists of $N_2$ iterations, each of which  consists of the second scale strategy with its $Q$ set to be $Q+Q^{2, j}$, followed by the multi-stage $(\log^{\alpha'_2}n, Q+Q^{2, j}+\ell_2, \ell_2)$-threshold strategy.
 }
 \vspace{12pt}
 \begin{overpic}[width=1\textwidth,tics=5]{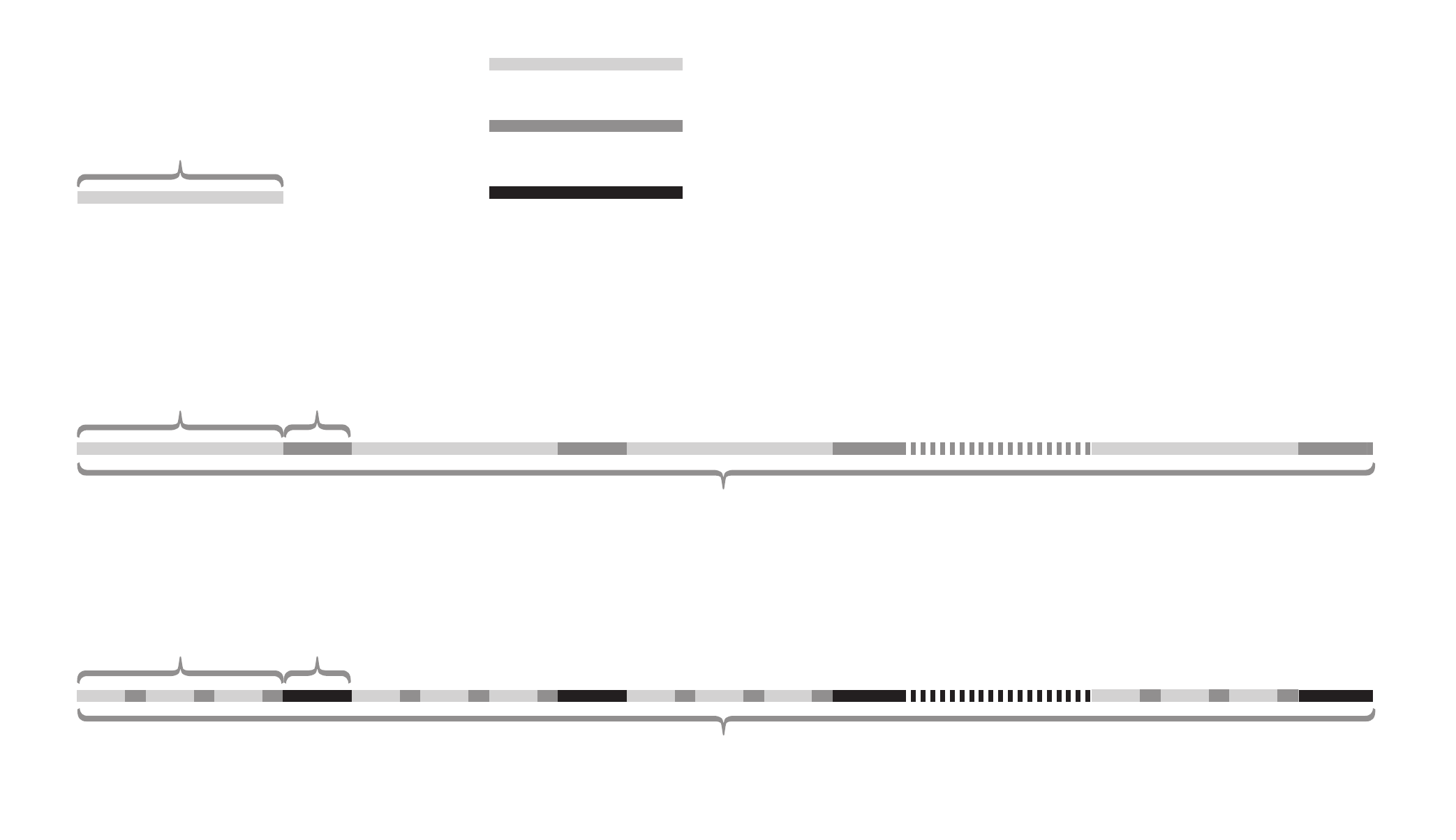}
\put(50,53){$L$-relative threshold strategy}
\put(50,48.5){$(\cdot,\cdot,\ell_1)$-multi-stage threshold strategy}
\put(50,43.9){$(\cdot,\cdot,\ell_2)$-multi-stage threshold strategy}
 \put(0,54){\Large 1\textsuperscript{st} scale}
 \put(7,47.5){$\lfloor \log^{\alpha_1} n\rfloor$}
 \put(0,37){\Large 2\textsuperscript{nd} scale}
 \put(7,30.5){$\lfloor \log^{\alpha_1} n\rfloor$}
 \put(17,30.5){$\lfloor \log^{\alpha_1'} n\rfloor$}
 \put(43,22){$N_1$-iterations}
  \put(0,19){\Large 3\textsuperscript{rd} scale}
 \put(7,13.5){$\lfloor \log^{\alpha_2} n\rfloor$}
 \put(17,13.5){$\lfloor \log^{\alpha_2'} n\rfloor$}
 \put(43,5){$N_2$-iterations}
 \end{overpic}\label{fig:multi-key-fig}
 \end{figure}
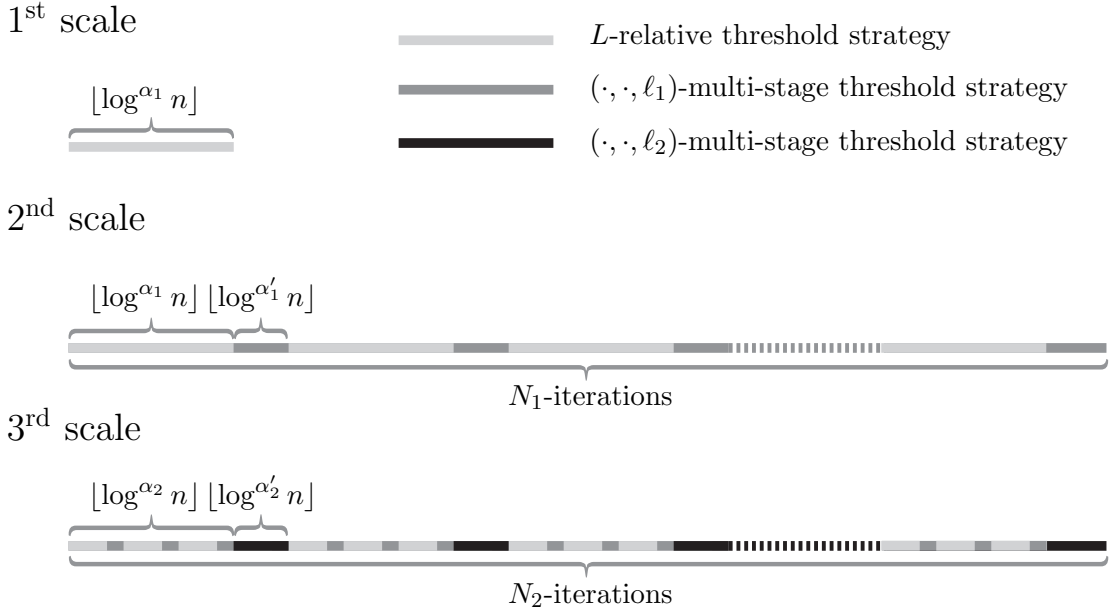

\pagebreak
\textbf{The $d$-multi-scale long-term  combined strategy}. 
This strategy is used to control the typical maximum load for arbitrarily long time and it consists of multiple iterations. As in the $Q$-multi-scale strategy, we set $\alpha_1=\frac{1}{2}+\frac{2}{\lfloor\sqrt{\log\log\log n}\rfloor+1/4}$, $L=(\log n)^{\frac{1+\alpha_1}{3}}$ and $k=\big\lfloor\frac{\log\log n}{3\log\log\log n}\big\rfloor$. The sequence $\{\alpha_i\}_{i\in\N}$ is defined in \eqref{eq:alpha-i-prime-iteration} and \eqref{eq:alpha-i-iteration}.
We denote by $i_{\max}=\max\{i\in \N: \alpha_i\leq 1\}$. We set
\begin{equation}\label{eq:Q-L-A}
Q=L=(\log n)^{\frac{1+\alpha_1}{3}},~~A=\sqrt{6d (\log n)^{1+\alpha_{i_{\max}+1}}},
\end{equation}
\begin{equation}\label{eq:ms}
m_0=\lfloor200dn\log n\rfloor,~~m_1 =n(\log n)^{\alpha_{i_{\max}+1}},~~m_2 =\lceil16nA\rceil.
\end{equation}
\begin{equation}\label{eq:L0}
L_0=\big\lfloor(\log n)^{\frac{1}{2}+\left(2-\frac{1}{2k+1}\right)\frac{\alpha-1/2}{2k+1}}\big\rfloor,~\text{where}~\alpha=\frac{\log (m_0/n)}{\log\log n}.
\end{equation}
In this strategy, a \emph{standard iteration} consists of three phases: The first one consists of 
the allocation of $m_0$ balls according to the multi-stage ($m_0/n, L_0, L_0$)-threshold strategy defined in Section \ref{sec:combin}; the second phase consists of the allocation of $m_1$ balls using the $Q$-multi-scale strategy; the third phase consists of the allocation of balls according to the 1/5-drift strategy given in Section \ref{sec:base strat}, until the first time $m$ when the following three conditions are satisfied
\begin{itemize}
\item At least $m_2$ balls were allocated during this phase,
\item $\max_{i\in[n]}\big|L_i^f(m)\big|\le 100d\log n$,
\item $\left|\left\{i\in[n] : L_i^f(m)>L_0\right\}\right|< 4000ne^{-L_0 /15}$.
\end{itemize}
The strategy itself consists of applying such iterations indefinitely, with the exception that we skip the first phase in the first iteration. The purpose of this exception is to make this strategy an extension of the $Q$-multi-scale strategy. 

\subsection{Optimal strategies}\label{sec:opt strat}

We summarize in Table \ref{table:2} the strategies and the time intervals where these strategies are employed to control the single-time, all-time and typical maximum loads. Notice that strategies that work for larger values of $m$ encapsulate those that work for smaller values so that the more advanced strategy could be also used for smaller values of $m$.

\begin{table}[H]
\def\arraystretch{1.7}
\centering
\resizebox{\columnwidth}{!}{
\begin{tabular}{||L | L  | L | L | L | L ||} 
 \hline
 \hline
 \phantom{\Big(}            &  \scalebox{0.97}{$m\!\le\!O(n\sqrt{\log n})$}    & \scalebox{0.97}{$m\le O(n\log n)$}  &  \scalebox{0.97}{$m\le O(n\log^2 n)$}&  \scalebox{0.97}{$m\le n^{O(1)}$}  &  \hspace{-50pt}generic $m$ \\
 \hline\hline
 {Maximum load at time $m$}  & Threshold strategy &  Multi-stage threshold strategy   & \multicolumn{3}{c||}{Drift multi-stage threshold strategy}  \\[3pt]
 \hline
 {Maximum load up to time $m$}  & \multicolumn{2}{c|}{Relative threshold strategy} &  \multicolumn{2}{c|}{Varying drift strategy }   & - \\[3pt]
                             \hline
 Typical load up to time $m$ & Relative threshold Strategy & \multicolumn{2}{c|}{$Q$-multi-scale threshold strategy} & \multicolumn{2}{c||}{$d$-multi-scale long-term  combined strategy}\\
 \hline
 \hline
\end{tabular}}
\caption{Optimal strategies for the single-time, all-time and typical maximum loads.}
\label{table:2}
\end{table}


%

\section{Preliminaries}\label{eq:preliminaries}

\subsection{Poisson approximation}

One difficulty of analyzing the balls-and-bins model is the correlation among the loads of different bins. The following result shows that the joint distribution of the loads of different bins can be well approximated by assuming that the loads of these bins are independent $\poss (m/n)$ random variables.

Let $\N_0=\N\cup\{0\}$. Given $x, y\in(\N_0)^n$, we say that $x\leq y$ if $x_i\leq y_i$ for all $i\in[n]$. A subset $S\subset(\N_0)^n$ is called \emph{monotone decreasing (resp. increasing)} if $x\in S$ implies that $y\in S$ for all $y\leq x$ (resp. $x\leq y$).

\begin{lem}[\cite{MU05}, Theorem 5.10]\label{lem:poisson approximation}
Let $\{X_i\}_{i\in[n]}$ be the number of balls in bins $i\in[n]$ when $m$ balls are independently and uniformly placed into $n$ bins. Let $\{Y_i\}_{i\in[n]}$ be independent $\poss (m/n)$ random variables. For any monotone set $S\subseteq[n]$, we have
$$
\P((X_1, \cdots, X_n)\in S)\leq 2\P((Y_1, \cdots, Y_n)\in S).
$$
\end{lem}

We borrow the following lemma from \cite{FGG18}, which provides a concentration bound on the maximum load
over a subset of bins.

\begin{lem}[\cite{FGG18}, Lemma 2.2] \label{lem:max load}
Let $\{X_i\}_{i\in[n]}$ be the number of balls in bins $i\in[n]$ when  $\lfloor\theta n\rfloor, 0\leq\theta\leq1$, balls are independently and uniformly placed into $n$ bins. For $k\in\lfloor\theta n\rfloor$ and $S\subseteq [n]$, we have
$$
\P\left(\max_{i\in S}X_i<k\right)\leq 2\exp\left(-\frac{\theta^k|S|}{ek!}\right).
$$
\end{lem}




\subsection{Poisson tail estimate}

 Let $X$ be a $\poss(\lambda)$ random variable. When $\lambda$ is an integer, $X$ can be seen as the sum of $\lambda$ independent $\poss(1)$ random variables. As a consequence of Cram\'er's Theorem (e.g., \cite{DZ10:book}, Theorem 2.2.3), $\lambda^{-1} X$ satisfies the Large Deviation Principle (LDP), namely, for any closed set $F\subset \R$,
$$
\limsup_{\lambda\rightarrow\infty}\frac{1}{\lambda}\log \P(\lambda^{-1}X\in F)\leq -\inf_{x\in F}\Lambda^*(x),
$$
and for any open set $J\subset \R$,
$$
\liminf_{\lambda\rightarrow\infty}\frac{1}{\lambda}\log\P(\lambda^{-1}X\in J)\geq -\inf_{x\in J}\Lambda^*(x),
$$
where the rate function
$$
\Lambda^*(x)=
\begin{cases}
1-x+x\log x, & x>0\\
+\infty, & \text{otherwise}.
\end{cases}
$$
The statement actually holds for general $\lambda$ that is not necessarily an integer. This LDP readily implies the following Poisson tail bounds.

\begin{lem}\label{lem:poss tail}
Let $X$ be a $\poss(\lambda)$ random variable. For sufficiently large $\lambda$ and any $\kappa>0$,
\begin{equation}\label{eq:poss upper tail}
e^{-2\lambda I(\kappa/\lambda)}\leq\P(X\geq \lambda+\kappa)\leq e^{-\lambda I(\kappa/\lambda)},
\end{equation}
and  for any $0<\kappa<\lambda$, 
\begin{equation}\label{eq:poss lower tail}
e^{-2\lambda I(-\kappa/\lambda)}\leq\P(X\leq \lambda-\kappa)\leq e^{-\lambda I(-\kappa/\lambda)},
\end{equation}
where $I(x)=\Lambda^*(1+x)=(1+x)\log(1+x)-x$ for $x\in(-1, \infty)$.

\end{lem}

\begin{rmk}
In fact, the upper bounds hold for any $\lambda>0$ and this readily follows from Chernoff's argument. As the name, LDP, indicates, Lemma \ref{lem:poss tail} provides a good approximation of the Poisson tail when $\kappa$ is larger than the standard deviation $\sqrt \lambda$.  The following approximation of the rate function $I(x)$ will be repeatedly used. For $0\leq x\leq 4$, we have
\begin{equation}\label{eq:approx-I}
\frac{x^2}{4}\leq I(x)\leq \frac{x^2}{2}
\end{equation}
and, for $x\geq 4$, we have
\begin{equation}\label{eq:approx-I-1}
x\log\frac{x}{e}\leq I(x)\leq 3x\log\frac{x}{e}.
\end{equation}
\end{rmk}

The following result will be repeatedly used in later sections to estimate the number of retries in the allocation of balls using the threshold strategy.

\begin{lem}\label{lem:retry bound}
Let $\{X_i\}_{i\in[n]}$ be independent $\poss(\lambda)$ random variables. Let $\ell>0$. We define $Y_i=\max\{0, X_i-\lambda-\ell\}$ and $Y=\sum_{i=1}^nY_i$. Set $r^*=6ne^{-\lambda I(\ell/\lambda)}/\log(1+\ell/\lambda)$, where the function $I(x)$ is given in Lemma \ref{lem:poss tail}. Then we have
\begin{equation}
\P(Y>r^*)<\exp\left(-ne^{-\lambda I(\ell/\lambda)}\right).
\end{equation}
\end{lem}

\begin{proof}
The statement follows from the classical Chernoff's argument. For any $u>0$, we have
\begin{align*}
\E e^{u Y_1} &< 1+e^{-u \ell}\sum_{k=\lceil\lambda+\ell\rceil }^\infty e^{u (k-\lambda)}\cdot\P(X_1= k)\\
&=1+e^{-u \ell}\sum_{k=\lceil \lambda+\ell\rceil }^\infty e^{u (k-\lambda)}\Big(\P(X_1\ge k)-\P(X_1\ge k+1)\Big)\\
&=1+e^{-u \ell}\left(\sum_{k=\lceil \lambda+\ell\rceil }^\infty e^{u (k-\lambda)}\cdot\P(X_1\geq k)-e^{-u}\sum_{k=\lceil \lambda+\ell\rceil+1 }^\infty e^{u (k-\lambda)}\cdot\P(X_1\geq k) \right)\\
&=1+e^{-u \ell}\left(e^{u (\lceil \lambda+\ell \rceil-\lambda)}\cdot\P(X_1\geq \lceil\lambda+\ell \rceil)+(1-e^{-u})\sum_{k=\lceil \lambda+\ell \rceil+1}^\infty e^{u (k-\lambda)}\cdot\P(X_1\geq k) \right).
\end{align*}
Write $\ell^*=\lceil \lambda+\ell \rceil-\lambda$ and $j_k=k-\lambda$. We obtain 
\begin{equation}\label{eq:laplace trans}
\E e^{u Y_1}< 1+e^{-u \ell}\left(e^{u \ell^*}\cdot\P(X_1\geq \lambda + \ell^* )+(1-e^{-u})\sum_{k=\lambda+\ell^*+1}^\infty e^{u j_k}\cdot\P(X_1\geq \lambda + j_k) \right).
\end{equation}
For any $k>0$, we apply Lemma \ref{lem:poss tail} to obtain
\begin{equation}\label{eq:g-lambda func}
e^{u k}\cdot\P(X_1\geq \lambda+k)\leq e^{\lambda g_u(k/\lambda)},
\end{equation}
where $g_u(x)=(1+u)x-(1+x)\log(1+x)$. One can check that $g_u'(x)=u-\log(1+x)$ and that $g_u''(x)=-(1+x)^{-1}<0$. Let $u^*=\frac{1}{2}\log(1+\ell/\lambda)$. Then, $g_{u^*}(x)$ is a decreasing and concave function for $x\geq\ell/\lambda$. Hence, we have for any $k\geq \ell$ that
\begin{equation}\label{eq:ratio bound}
\frac{e^{\lambda g_{u^*}((k+1)/\lambda)}}{e^{\lambda g_{u^*}(k/\lambda)}}=\exp\left(\frac{g_{u^*}((k+1)/\lambda)-g_{u^*}(k/\lambda)}{1/\lambda}\right)<e^{g_{u^*}'(k/\lambda)}\leq e^{g_{u^*}'(\ell/\lambda)}=e^{-u^*},
\end{equation}
where the second equality follows from the formula for $g'_{u^*}(x)$ and our choice of $u^*$. Combining \eqref{eq:laplace trans}, \eqref{eq:g-lambda func} and \eqref{eq:ratio bound}, we have
\begin{align*}
\E e^{u^* Y_1} &<1+e^{-u^* \ell}\left(e^{\lambda g_{u^*}(\ell^*/\lambda)}+(1-e^{-u^*})\sum_{k=\lambda+\ell^*+1}^\infty e^{\lambda g_{u^*}(j_k/\lambda)} \right)\\
&< 1+e^{-u^* \ell}\left(e^{\lambda g_{u^*}(\ell^*/\lambda)}+e^{\lambda g_{u^*}((\ell^*+1)/\lambda)}\right)\\
&<1+2e^{-u^* \ell}\cdot e^{\lambda g_{u^*}(\ell/\lambda)}=1+2e^{-\lambda I(\ell/\lambda)}\\
&<\exp\left(2e^{-\lambda I(\ell/\lambda)}\right),
\end{align*}
where the second last inequality follows from the fact that $\ell^*\geq \ell$ and that $g_{u^*}(x)$ is  decreasing for $x\geq\ell/\lambda$. Then we apply Markov's inequality to obtain for any $r>0$ that
$$
\P(Y>r)\leq e^{-u^*r}\E e^{u^*Y}=e^{-u^*r}\left(\E e^{u^* Y_1}\right)^n<\exp\left(2ne^{-\lambda I(\ell/\lambda)}-u^*r\right).
$$
Recall that $u^*=\frac{1}{2}\log(1+\ell/\lambda)$. In particular, for $r^*=6ne^{-\lambda I(\ell/\lambda)}/\log(1+\ell/\lambda)$, we  have
$$
\P(Y>r^*)<\exp\left(-ne^{-\lambda I(\ell/\lambda)}\right).
$$
This concludes the proof.
\end{proof}


\subsection{Concentration bounds for the drift strategy}

As our drift strategy is based on a coupling of the allocation process and a continuous time random process, our concentration bounds for the drift strategy rely on the study of a particular type of temporal point processes. We refer the interested readers to \cite{DVJ03:book, DVJ08:book} for more details of general temporal point processes.

\textbf{$\theta$-standardizing point process.} 
A temporal point process $X(t)$ is called \emph{$\theta$-standardizing} if the conditional intensity function $\lambda(t)$ satisfies 
\begin{align}
   \lambda(t)&< 1-\theta,~~ \text{if}~X(t)\geq t, \label{eq:upper-theta-stand}\\
   \lambda(t)&\geq 1+\theta,~~ \text{if}~X(t)<t \label{eq:lower-theta-stand}.
\end{align}
We say that $X(t)$ is \emph{upper $\theta$-standardizing} if \eqref{eq:upper-theta-stand} holds, and that $X(t)$ is \emph{lower $\theta$-standardizing} if  \eqref{eq:lower-theta-stand} holds.

\begin{lem}\label{lem:concentration-theta-standarizing}
Let $\{X(t)\}_{t\geq0}$ be a temporal point process adapted to the filtration $\{\cF_t\}_{t\geq 0}$. Let $s\geq 0$ be a stopping time with respect to $\{\cF_t\}_{t\geq 0}$ and let $\eta\in[0, 1]$ be a $\cF_s$ measurable random variable. Denote $Y(t)=X(t)-t$.
\begin{enumerate}

\item If $X(t)$ is upper $2\theta$-standarizing, then 
we have 
\begin{equation}\label{eq:con-upper-theta}
    \E \big[ e^{\theta Y(s+\eta)}\ |\ \cF_s\big]\leq e^{-\theta^2\eta}\cdot e^{\theta Y(s)}+e^{2\theta},
\end{equation}
and for any $\lambda$ satisfying $(1-2\theta)e^\lambda<\lambda/2$, 
we have 
\begin{equation}\label{eq:con-upper-theta-0}
    \E \big[ e^{\lambda Y(s+\eta)}\ |\ \cF_s\big]\leq e^{-\frac{\lambda}2\eta}\cdot e^{\lambda Y(s)}+e^{2\lambda}.
\end{equation}

\item If $X(t)$ is lower $2\theta$-standarizing, then we have
\begin{equation}\label{eq:con-lower-theta}
    \E \big[ e^{-\theta Y(s+\eta)}\ |\ \cF_s\big]\leq e^{-\theta^2\eta}\cdot e^{-\theta Y(s)}+e^\theta.
\end{equation}

\item If $X(t)$ is $2\theta$-standarizing, then we have
\begin{equation}\label{eq:con-theta}
    \E \big[ e^{\theta|Y(s+\eta)|}\ |\ \cF_s\big]\leq e^{-\theta^2\eta}\cdot e^{\theta |Y(s)|}+3e^{2\theta}.
\end{equation}
\end{enumerate}
\end{lem}

\begin{proof}
We denote by $Z(\beta)$ a $\poss(\beta)$ random variable throughout the proof. We first prove inequalities \eqref{eq:con-upper-theta} and \eqref{eq:con-upper-theta-0}. We need to estimate the Laplace transform of $Z(\alpha(1-2\theta))$ for any $\alpha>0$ as follows
\begin{align}
\E e^{\lambda[Z(\alpha(1-2\theta))-\alpha]}&=e^{\alpha(1-2\theta)(e^{\lambda}-1)-\alpha\lambda}\notag\\
&\leq 
\begin{cases}\label{eq:Z-laplace}
e^{\alpha(1-2\theta)(\lambda+\lambda^2)-\alpha\lambda}\le e^{\alpha(\lambda^2-2\lambda\theta)} & 0\leq \lambda\leq1,\\
e^{\alpha(1-2\theta)e^{\lambda}-\alpha\lambda}\leq e^{-\lambda\alpha/2} & (1-2\theta)e^\lambda\le \lambda/2.
\end{cases}
\end{align}
We define $s_*=\min\{t\in [s,s+\eta]\ :\ Y(t)\geq1\}$ and set $s_*=s+\eta$ if the minimum is taken over an empty set. Then, $s_*$ is a stopping time with respect to $\{\cF_t\}_{t\geq 0}$. 
We have
 \begin{align}\label{eq:con-upper-theta-1}
 \E \big[e^{\lambda Y(s+\eta)}\ |\ \cF_{s_*}\big] 
 &=e^{\lambda Y(s_*)}\cdot\E \big[e^{\lambda[Y(s+\eta)-Y(s_*)]}\ |\ \cF_{s_*}\big] \notag\\
 &\le e^{\lambda Y(s_*)}\cdot\E \big[e^{\lambda[Z((1-2\theta)(s+\eta-s_*))-(s+\eta-s_*)]}\ |\ \cF_{s_*}\big]\notag\\
 &\leq \begin{cases}
 e^{-\theta^2(s+\eta-s_*)}\cdot e^{\theta Y(s_*)} & \lambda=\theta,\\
 e^{-\frac{\lambda}{2}(s+\eta-s_*)}\cdot e^{\lambda Y(s_*)} & (1-2\theta)e^\lambda\le \lambda/2. 
 \end{cases} \notag\\
 &\leq  \begin{cases}
 e^{-\theta^2\eta}\cdot e^{\theta Y(s)}+ e^{2\theta}& \lambda=\theta,\\
 e^{-\frac{\lambda}{2}\eta}\cdot e^{\lambda Y(s)}+e^{2\lambda} & (1-2\theta)e^\lambda\le \lambda/2. 
 \end{cases}
 \end{align}
To see the first inequality, observe that $Y(t)=X(t)-t\geq 0$ for $t\in[s_*, s+\eta]$. Since $X(t)$ is upper $2\theta$-standardizing, $Y(s+\eta)-Y(s_*)=X(s+\eta)-X(s_*)-(s+\eta-s_*)$ is dominated by $Z((1-2\theta)(s+\eta-s_*))-(s+\eta-s_*)$. The second inequality follows from \eqref{eq:Z-laplace}. In each case of \eqref{eq:con-upper-theta-1}, the first term is an upper bound for the case $s_*=s$, while the second term uses the fact that $Y(s_*)< 2$ when $s_*\neq s$. Inequalities \eqref{eq:con-upper-theta} and \eqref{eq:con-upper-theta-0} follow from the tower property of conditional expectation and \eqref{eq:con-upper-theta-1}.

Next we prove \eqref{eq:con-lower-theta}. Write $E=\{Y(t)\le 0~ \text{for all}~ t\in [s,s+\eta]\}$. Observe that, whenever $E^c$ occurs, we have $Y(s+\eta)\ge -1$. Hence, 
\begin{align*}
\E \big[ e^{-\theta Y(s+\eta)}\ |\ \cF_s\big] 
&=\E \big[e^{-\theta Y(s+\eta)}\ind_E+e^{-\theta Y(s+\eta)}\ind_{E^c} \ |\ \cF_s\big]\\
&\leq \E \big[e^{-\theta Y(s+\eta)}\ind_E\ |\ \cF_s\big]+e^\theta\\
&= e^{-\theta Y(s)} \cdot\E \big[e^{-\theta [Y(s+\eta)-Y(s)]}\ind_E\ |\ \cF_s\big]+e^\theta\\
&\le e^{-\theta Y(s)} \cdot \E \big[e^{-\theta [Z((1+2\theta)\eta)-\eta]}\ |\ \cF_s\big]+e^\theta\\
&\le e^{-\theta^2 \eta}\cdot e^{-\theta Y(s)}+e^\theta.
\end{align*}
To see the second inequality, observe that, whenever $E$ occurs, we have $Y(t)=X(t)-t\leq 0$ for all $t\in [s, s+\eta]$. Since $X(t)$ is $2\theta$-standardizing,  $Y(s+\eta)-Y(s)=X(s+\eta)-X(s)-\eta$ dominates $Z((1+2\theta)\eta)-\eta$. The last inequality follows from that for any $\alpha> 0$, 
$$
\E e^{-\theta [Z(\alpha(1+2\theta))-\alpha]}=e^{\alpha(1+2\theta)(e^{-\theta}-1)+\alpha\theta}<e^{\alpha(1+2\theta)(-\theta+\theta^2/2)+\alpha\theta}<e^{-\alpha\theta^2}.
$$

When $X(t)$ is $2\theta$-standarizing, it is both upper and lower $2\theta$-standarizing. Hence, inequalities \eqref{eq:con-upper-theta} and \eqref{eq:con-lower-theta} hold. Observe that 
\begin{align*}
 \E \big[ e^{\theta |Y(s+\eta)|}\ |\ \cF_s\big] &\leq \E \big[ e^{\theta Y(s+\eta)}\ |\ \cF_s\big]+\E \big[ e^{-\theta Y(s+\eta)}\ |\ \cF_s\big].
\end{align*}
This, together with \eqref{eq:con-upper-theta} and \eqref{eq:con-lower-theta}, yields \eqref{eq:con-theta}.
\end{proof}

\begin{coro}\label{coro:concentration-theta-standarizing-a}
Let $\{X(t)\}_{t\geq0}$ be a temporal point process adapted to the filtration $\{\cF_t\}_{t\geq 0}$. Denote $Y(t)=X(t)-t$. 
\begin{enumerate}
\item If $X(t)$ is upper $2\theta$-standarizing,  we have for any $t\geq s$, 
\begin{equation}\label{eq:con-upper-theta-a}
    \E \big[ e^{\theta Y(t)}\ |\ \cF_s\big]\leq e^{-\theta^2(t-s)}\cdot e^{\theta Y(s)}+\frac{2e^{2\theta}}{\theta^2},
\end{equation}
and for any $\lambda$ satisfying $(1-2\theta)e^\lambda<\lambda/2$,
\begin{equation}\label{eq:concentration-theta-one-sided-a}
\E \big[ e^{\lambda Y(t)}\ |\ \cF_s \big]<e^{-\frac{\lambda}{2}(t-s)}\cdot e^{\lambda Y(s)}+\frac{2e^{2\lambda}}{1-e^{-\lambda/2}}.
\end{equation}

\item If $X(t)$ is lower $2\theta$-standarizing, we have for any $t\geq s$,
\begin{equation}\label{eq:con-lower-theta-a}
    \E \big[ e^{-\theta Y(t)}\ |\ \cF_s]\leq e^{-\theta^2(t-s)}\cdot e^{-\theta Y(s)}+\frac{2e^{\theta}}{\theta^2}.
\end{equation}

\item If $X(t)$ is $2\theta$-standarizing, we have for any $t\geq s$,
\begin{equation}\label{eq:con-theta-a}
    \E \big[ e^{\theta |Y(t)|}\ |\ \cF_s\big]\leq e^{-\theta^2(t-s)}\cdot e^{\theta |Y(s)|}+\frac{6e^{2\theta}}{\theta^2}.
\end{equation}
\end{enumerate}
\end{coro}

\begin{proof}
We only prove \eqref{eq:con-upper-theta-a} and inequalities \eqref{eq:concentration-theta-one-sided-a}, \eqref{eq:con-lower-theta-a}, \eqref{eq:con-theta-a} can be proved in a similar manner.
Lemma~\ref{lem:concentration-theta-standarizing} yields that for any $k\in\N$,
$$
\E\left[ \left(e^{\theta Y(s+k)}-\frac{e^{2\theta}}{1-e^{-\theta^2}}\right)e^{\theta^2(s+k)}\ \big|\ \cF_{s+k-1}\right]
\leq  \left(e^{\theta Y(s+k-1)}-\frac{e^{2\theta}}{1-e^{-\theta^2}}\right)e^{\theta^2 (s+k-1)}.
$$
Hence, $\left\{\left(e^{\theta Y(s+k)}-\frac{e^{2\theta}}{1-e^{-\theta^2}}\right)e^{\theta^2(s+k)}\right\}_{k\in\N}$ is a supermartingale and for any $k\in\N$, we have
\begin{equation}\label{eq:concentration-theta-standarizing-c}
\E\big[e^{\theta Y(s+k)}~|~\cF_s \big]\leq e^{-\theta^2k}\cdot e^{\theta Y(s)}+\frac{e^{2\theta}(1-e^{-\theta^2k})}{1-e^{-\theta^2}}.
\end{equation} 
For any $t\geq s$, we have
\begin{align*}
\E\big[e^{\theta Y(t)}~|~\cF_s \big] &= \E\big[\E\big[e^{\theta Y(t)}~|~\cF_{s+\lfloor t-s\rfloor} \big]~|~\cF_s\big]\\
&\leq e^{-\theta^2(t-s-\lfloor t-s\rfloor)}\cdot\E\big[e^{\theta Y(s+\lfloor t-s\rfloor)}~|~\cF_s\big]+e^{2\theta}\\
&\leq e^{-\theta^2(t-s)}\cdot e^{\theta Y(s)}+\frac{2e^{2\theta}}{1-e^{-\theta^2}}\\
&\leq e^{-\theta^2(t-s)}\cdot e^{\theta Y(s)}+\frac{2e^{2\theta}}{\theta^2}.
\end{align*}
In the first inequality, we use Lemma~\ref{lem:concentration-theta-standarizing}, and in the second inequality, we use \eqref{eq:concentration-theta-standarizing-c}. The last inequality follows from $e^{-x}>1-x$. 
\end{proof}

\begin{coro}\label{coro:con-theta-sum}
We denote by $\{X_i(t)\}_{i\in[n]}$ independent  $2\theta$-standarizing point processes with initial values $\{X_i(0)\}_{i\in [n]}$ such that $|X_i(0)|\leq L$ for all $i\in [n]$. For all $t\geq L/\theta$, we have
\begin{equation}\label{eq:con-theta-b}
\E e^{\theta |X_i(t)-t|}\leq \frac{20}{\theta^2}.
\end{equation}
Write $Y(t)=\frac{1}{n}\sum_{i=1}^nX_i(t)-t$. For all $t\geq L/\theta$, we have
\begin{equation}\label{eq:con-theta-sum}
\E e^{\theta |Y(t)|}\leq \frac{20}{\theta^2}~~~\text{and}~~~\E e^{\theta n|Y(t)|}\leq \left(\frac{20}{\theta^2}\right)^n.
\end{equation}
In addition, for $0\le t<L/\theta$, we have
\begin{equation}\label{eq:con-theta-sum2}
\E e^{\theta |X_i(t)-t|}\leq e^{\theta L}+\frac{20}{\theta^2}
~~~\text{and}~~~\E e^{\theta |Y(t)|}\leq e^{\theta L}+\frac{20}{\theta^2}.
\end{equation}
\end{coro}

\begin{proof}
Inequality \eqref{eq:con-theta-a} and the assumption that $|X_i(0)|\leq L$ imply that 
$$
\E e^{\theta |X_i(t)-t|}\leq e^{-\theta^2t+\theta L}+\frac{6e^{2\theta}}{\theta^2}.
$$
For $t\ge L/\theta$, the RHS of the above inequality is at most $1+\frac{6e^{2\theta}}{\theta^2}\leq \frac{20}{\theta^2}$; for $0\leq t<L/\theta$, it can be trivially bounded above by $e^{\theta L}+\frac{20}{\theta^2}$. This proves inequality \eqref{eq:con-theta-b} and the first inequality of \eqref{eq:con-theta-sum2}. Then we can use inequality \eqref{eq:con-theta-b} to obtain for $t\ge L/\theta$ that
\begin{align*}
\E e^{\theta |Y(t)|} &\leq \E e^{\frac{\theta}{n}\sum_{i=1}^n|X_i(t)-t|}=\left(\prod_{i=1}^n\E e^{\theta |X_i(t)-t|}\right)^{1/n} 
\leq \frac{20}{\theta^2},
\end{align*}
and
\begin{align*}
\E e^{\theta n|Y(t)|} &\leq \E e^{\theta\sum_{i=1}^n|X_i(t)-t|}=\prod_{i=1}^n\E e^{\theta |X_i(t)-t|}\leq \left(\frac{20}{\theta^2}\right)^n.
\end{align*}
Similarly, we can use the first inequality of \eqref{eq:con-theta-sum2} to obtain the second inequality of  \eqref{eq:con-theta-sum2}.
\end{proof}

Consider a collection independent regular point processes $\{X_i(t)\}_{i\in [n]}$  with the initial value $\{L_i(0)\}_{i\in [n]}$ and conditional intensity functions $\{\lambda_i(t)\}_{i\in [n]}$ given in \eqref{eq:intensity xi}. The
process $\{Z_k\}_{k\in\N}$ defined in \eqref{eq:def zk} is the 
output of the $\theta$-drift strategy $f$ as per Section~\ref{sec:base strat}. We show the following concentration bounds on the load vector $\{L_i^f(m)\}_{i\in [n]}$.

\begin{lem}\label{lem:single time upper bound}
Suppose that $|L_i(0)|\leq L$ for all $i \in [n]$. Set $\theta=1/5$. The $\theta$-drift strategy $f$ satisfies that for any $m\geq \big( \frac{3L}{\theta}+\frac{10}{\theta}\log\frac{80}{\theta^2}\big)n$, any $i\in [n]$ and any $k>0$,
\begin{equation}\label{eq:load-discrepancy-drift}
\P\left( \big|L_i^f(m)|>k\right)\leq\frac{320}{\theta^2}\exp\left(-\frac{\theta k}{5}\right).
\end{equation}
Taking the union bound, we have
\begin{equation}\label{eq:maxload-drift}
\P\left( \max_{i\in [n]}|L_i^f(m)|>k+\frac{5}{\theta}\log\frac{320n}{\theta^2}\right)\leq \exp\left(-\frac{\theta k}{5}\right).
\end{equation}
\end{lem}

\begin{proof}

Set $t^*=m/n+k/2$ and $t_*=\max(m/n-k/2,0)$. We denote by $E=\{X(t^*)\geq m\}$ and $F=\{X(t_*)\leq m\}$.
Using the law of total probability, we obtain
\begin{align}\label{eq:load-discrepancy-drift-a}
\P\left( |L_i^f(m)|>k\right) &=\P\left( L_i^f(m)>k\right)+\P\left( L_i^f(m)<-k\right)\notag\\
&\leq \P\left( L_i^f(m)>k, E\right)+\P(E^c)+\P\left( L_i^f(m)<-k, F\right) +\P(F^c).
\end{align}

We now estimate the first two terms of \eqref{eq:load-discrepancy-drift-a}.
Since $X_i(t)$ given in \eqref{eq:intensity xi} is $\theta$-standardizing, we apply the first inequality of \eqref{eq:con-theta-sum} and Markov's inequality to obtain
\begin{equation}\label{eq:e complement}
\P(E^c) = \P\left(\frac{X(t^*)}{n}<t^*-\frac{k}{2}\right)\leq e^{-\frac{\theta k}{4}} \cdot\E \exp\left(\frac{\theta}{2}\Big|\frac{X(t^*)}{n}-t^*\Big|\right)\leq \frac{80}{\theta^2} \exp\left(-\frac{\theta k}{4}\right).
\end{equation}
Whenever $E$ occurs, we have $L_i^f(m)\leq X_i(t^*)-m/n$. This, together with inequality \eqref{eq:con-theta-b} and Markov's inequality, yields
\begin{align}\label{eq:load-discrepancy-drift-a-1st term}
\P\left( L_i^f(m)>k,  E\right) &\leq \P\left( X_i(t^*)>\frac{m}{n}+k\right)=\P\left( X_i(t^*)>t^*+\frac{k}{2}\right)\notag \notag\\
&\leq e^{-\frac{\theta k}{4}} \cdot\E \exp\left(\frac{\theta}{2}|X_i(t^*)-t^*|\right)\leq \frac{80}{\theta^2}\exp\left({-\frac{\theta k}{4}}\right).
\end{align}

We next estimate the last two terms of \eqref{eq:load-discrepancy-drift-a}. We first estimate $\P(F^c)$. For $k\ge2m/n$ we have $t_*=0$ and $X(t_*)=\sum_{i\in [n]}L_i(0)=0$. This yields $\P(F^c)=0$. For $k< 2m/n$, we use the fact that $t_*=m/n-k/2$ to rewrite $F=\{X(t_*)/n\le t_*+k/2\}$. Set $k_0=2m/n-4L/\theta$. One can check that $t_*>2L/\theta$ for $0<k<k_0$ and that $0<t_*<2L/\theta$ for $k_0<k<2m/n$. We apply the first inequality of \eqref{eq:con-theta-sum}, the second inequality of \eqref{eq:con-theta-sum2} and Markov's inequality to obtain
\begin{align}\label{eq:f complement}
\P(F^c) & =\P\left(\frac{X(t_*)}{n}>t_*+\frac{k}{2}\right)\leq e^{-\frac{\theta k}{4}} \cdot\E \exp\left(\frac{\theta}{2}\Big|\frac{X(t_*)}{n}-t_*\Big|\right)\notag\\
&\leq \begin{cases}
\frac{80}{\theta^2}e^{-\theta k/4}, & 0<k\le k_0 \\
\left(e^{\theta L/2}+\frac{80}{\theta^2}\right)e^{-\theta k/4}, & k_0<k< 2m/n 
\end{cases}\notag\\
&\le \begin{cases}
\frac{80}{\theta^2}e^{-\theta k/4}, & 0<k\le k_0, \\
e^{-\theta k/5}, & k_0<k< 2m/n, 
\end{cases}
\end{align}
where the second case of inequality \eqref{eq:f complement} follows from $e^{\theta L/2}+80/\theta^2\le e^{\theta k_0/20}<e^{\theta k/20}$. To see this, we observe that our assumption on $m$ and our choice of $\theta=1/5$ imply that 
$$
\frac{2L}{\theta}+\frac{10}{\theta}\log\left(e^{\theta L/2}+\frac{80}{\theta^2}\right)\leq \frac{2L}{\theta}+\frac{10}{\theta}\log\big(e^{\theta L/2}\big)+\frac{10}{\theta}\log\frac{80}{\theta^2}=\frac{3L}{\theta}+\frac{10}{\theta}\log\frac{80}{\theta^2}\leq \frac{m}{n}.
$$
This can be rewritten as $k_0/2\geq \frac{10}{\theta}\log\left(e^{\theta L/2}+\frac{80}{\theta^2}\right)$, which is equivalent to the desired statement.

We now estimate the third term of \eqref{eq:load-discrepancy-drift-a}. For $k\ge2m/n$, we derive from the assumption on $m$ that
$$
L_i^f(m)\ge -L-\frac{m}{n}>-\frac{4m}{3n}>-k.
$$
In this case, we have $\P\big(L_i^f(m)< -k, F\big)=0$. We now deal with the case that $k<2m/n$. Whenever $F$ occurs, we have $L_i^f(m)\geq X_i(t_*)-m/n$. Together with $t_*= m/n-k/2$, this yields
$$
\P\left( L_i^f(m)<-k,  F\right) \leq \P\left( X_i(t_*)<\frac{m}{n}-k\right)=\P\left( X_i(t_*)<t_*-\frac{k}{2}\right).
$$
Recall that $k_0=2m/n-4L/\theta$ and the fact that $t_*>2L/\theta$ for $0<k<k_0$ and that $t_*<2L/\theta$ for $k_0<k<2m/n$. We  apply inequality \eqref{eq:con-theta-b}, the first inequality of \eqref{eq:con-theta-sum2} and Markov's inequality to obtain
\begin{align}\label{eq:load-discrepancy-drift-a-2nd term}
\P\left( L_i^f(m)<-k,  F\right) 
&\leq e^{-\frac{\theta k}{4}} \cdot\E \exp\left(\frac{\theta}{2}|X_i(t_*)-t_*|\right)\notag\\
&\le \begin{cases}
\frac{80}{\theta^2}e^{-\theta k/4}, & 0<k\le k_0 \\
\left(e^{\theta L/2}+\frac{80}{\theta^2}\right)e^{-\theta k/4}, & k_0<k<2m/n 
\end{cases}\notag\\
&\le \begin{cases}
\frac{80}{\theta^2}e^{-\theta k/4}, & 0<k\le k_0, \\
e^{-\theta k/5}, & k_0<k<2m/n,
\end{cases}
\end{align}
where the second case of inequality \eqref{eq:load-discrepancy-drift-a-2nd term} again uses $e^{\theta L/2}+80/\theta^2\le e^{\theta k_0/20}<e^{\theta k/20}$.

Combining \eqref{eq:load-discrepancy-drift-a}-\eqref{eq:load-discrepancy-drift-a-2nd term}, we obtain \eqref{eq:load-discrepancy-drift}.
\end{proof}

\begin{lem}\label{lem:level set bound}
Suppose that $|L_i(0)|\leq L$ for all $i \in [n]$. Set $\theta=1/5$ and $k_0=1+\frac{2}{\theta}\log\frac{80}{\theta^2}$. The $\theta$-drift strategy $f$ satisfies that for any $m\geq 2nL/\theta$ and any $k\geq 3k_0$,
\begin{equation}\label{eq:level set bound}
\P\left(\left |\left\{i\in [n]: L_i^f(m)>k\right\}\right|\geq \frac{160}{\theta^2}ne^{-\frac{\theta k}{3}}\right)\leq 2\exp\left(-2n\left(\frac{80}{\theta^2}\right)^2e^{-\frac{2\theta k}{3}}\right).
\end{equation}
\end{lem}

\begin{proof}
Set $t^*=m/n+k_0$. Let $E=\{X(t^*)\geq m\}$. Denote $S_k=\big\{i\in[n]: L_i^f(m)\geq k\big\}$. By the law of total probability, we have
\begin{equation}\label{eq:level set est}
\P\left(|S_k|\geq \frac{160}{\theta^2}e^{-\frac{\theta k}{4}}\right)\leq\P\left(|S_k|\geq \frac{160}{\theta^2}e^{-\frac{\theta k}{4}}, E\right)+\P(E^c).
\end{equation}
The second inequality of \eqref{eq:con-theta-sum}, Markov's inequality and our choice of $k_0$ yield
\begin{equation}\label{eq:e complement-1}
\P(E^c) = \P(X(t^*)<nt^*-nk_0)\leq \left(\frac{80}{\theta^2}\right)^n \exp\left(-\frac{n\theta k_0}{2}\right)=\exp\left(-\frac{\theta n}{2}\right).
\end{equation}
To estimate the first term in \eqref{eq:level set est}, we introduce independent Bernoulli random variables $W_i$, which are indicator functions of the events that $X_i(t^*)>m/n+k$. Hence,
\begin{align*}
\P(W_i=1) &=\P\left(X_i(t^*)>\frac{m}{n}+k\right)=\P\left(X_i(t^*)>t^*+k-k_0\right)\\
&\leq \P\left(X_i(t^*)>t^*+\frac{2k}{3}\right)\leq \frac{80}{\theta^2}\exp\left({-\frac{\theta k}{3}}\right),
\end{align*}
where in the first inequality, we use the assumption that $k\geq 3k_0$, and in the second inequality, we use the fact that $X_i(t)$ is $\theta$-standarizing and \eqref{eq:con-theta-b}. Observe that, when the event $E$ occurs, we have $L_i^f(m)\leq X_i(t^*)-m/n$, which implies that $|S_k|\leq \sum_{i=1}^nW_i$. This, together with Hoeffding's inequality, yields
\begin{equation*}\label{eq:level set est-e}
\P\left(|S_k|\geq \frac{160}{\theta^2}ne^{-\frac{\theta k}{3}}, E\right)\leq \P\left(\sum_{i=1}^nW_i \geq  \frac{160}{\theta^2}ne^{-\frac{\theta k}{3}}\right)\leq \exp\left(-2n\left(\frac{80}{\theta^2}\right)^2e^{-\frac{2\theta k}{3}}\right).
\end{equation*}
This, along with \eqref{eq:level set est} and \eqref{eq:e complement-1}, gives
\begin{equation*}
\P\left(|S_k|\geq \frac{160}{\theta^2}e^{-\frac{\theta k}{4}}\right)\leq
\exp\left(-2n\left(\frac{80}{\theta^2}\right)^2e^{-\frac{2\theta k}{3}}\right)+\exp\left(-\frac{\theta n}2\right).
\end{equation*}
This, together with the condition that $k\ge 3+\frac{6}\theta\log\tfrac{80}{\theta^2}$,
yields \eqref{eq:level set bound}.
\end{proof}
We also provide a concentration bound on the time it takes the drift strategy to bring certain quantities close to stationarity.
\begin{lem}\label{lem:expected normalization time}
Suppose that $|L_i(0)|\le L$ for all $i\in[n]$. Set $\theta=1/5$. 
Denote 
\begin{align*}
    A_m&=\left\{\max_{i\in[n]}\big|L_i^f(m)\big|\le k_a+\frac{5}{\theta}\log\frac{320n}{\theta^2}\right\},\\
    B_m&=\left\{\big|\big\{i\in[n] : L_i^f(m)>k_b\big\}\big|< \frac{160}{\theta^2}ne^{-\frac{\theta k_b}{3}}\right\},
\end{align*}
and assume that
$$ \exp\left(-\frac{\theta k_a}{5}\right)+2\exp\left(-2n\left(\frac{80}{\theta^2}\right)^2e^{-\frac{2\theta k_b}{3}}\right)<\frac12.$$
If $T=\min\big\{m\in\N : A_m \cap B_m\text{ holds}\big\}$, then under the $\theta$-drift strategy, we have
\[\E(T)<Cn(L+\log n)\]
for some absolute constant $C>0$ and all large enough $n$.
\end{lem}
\begin{proof}
Set $m_0=0$ and recursively define
$$ 
m_{j+1}=m_j+\left\lceil \tfrac{3n}{\theta}\max_{i\in[n]}\big|L_i^f(m_j)\big|+\frac{10n}{\theta}\log\frac{80}{\theta^2}\right\rceil.
$$
Denote
$$
J=\min\{j : A_{m_j} \cap B_{m_j}\}.
$$
It is obvious that $T\le m_J$. By Lemma 
\ref{lem:single time upper bound}, Lemma \ref{lem:level set bound} and the union bound, we have, conditioned on the history of the process until $m_j$ balls have been allocated, that
$$ 
\P\Big(A_{m_{j+1}}^c \cup B_{m_{j+1}}^c \ |\ \cF_{m_j}\Big) \le \exp\left(-\frac{\theta k_a}{5}\right)+2\exp\left(-2n\left(\frac{80}{\theta^2}\right)^2e^{-\frac{2\theta k_b}{3}}\right)<\frac12.
$$
Thus, we have $\P(J>j)\le 2^{-j}$ and hence
$$\E(J) \le 2.$$
For $j\geq 1$, we have by Lemma \ref{lem:single time upper bound} that
$$
\P\left( \max_{i\in [n]}\big|L_i^f(m_{j})\big|>k_a+\frac{5}{\theta}\log\frac{320n}{\theta^2} \ \Big|\ \cF_{m_{j-1}} \right)\leq \exp\left(-\frac{\theta k_a}{5}\right),
$$
which implies that
$$
\E\big(m_{j+1}-m_j \ |\ \cF_{m_{j-1}}\big) \le \left\lceil\frac{3n}{\theta} \left(\frac{5}{\theta}\log\frac{320n}{\theta^2} + \frac{1}{1-e^{-\theta/4}}\right)+\frac{10n}{\theta}\log\frac{80}{\theta^2}\right\rceil.
$$
Putting all these together, we obtain
$$
\E(T)\le \E(m_J)\le Cn(L+ \log n)
$$
for some $C>0$ and $n$ large enough. 
\end{proof}


\section{Single-time load discrepancy: upper bound}\label{sec:single-time load upper bound}

In this section, we investigate two-thinning strategies that can achieve the upper bounds on the single-time load discrepancy as stated in Theorem~\ref{thm:single-time load discrepancy}. Write $t=m/n$. Observe that for any thinning strategy $f$ and any $m\in\N$,
\begin{equation}\label{eq:integerization}
\maxload^f(\lfloor t\rfloor n)-1\leq \maxload^f(m)\leq \maxload^f(\lceil t\rceil n)+1.
\end{equation}
Hence, at the expense of an additive constant to the maximum load, we can always assume that $m$ is divisible by $n$, and then it suffices to study $\maxload^f(tn)$ for $t\in \N$.

\subsection{Case 1: 
\texorpdfstring{$t\le O(\sqrt{\log n})$
}{t ≤ O(√(log n))}} \label{sec:case 1}

In this case, we apply the $(t+\ell)$-threshold strategy introduced in \cite{FGG18} (see Section~\ref{sec:base strat}). 
Recall that this strategy retries a ball if its primary allocation is a bin which has accepted at least $t+\ell$ primary allocations. 

\begin{prop}\label{prop:single time upper bound case 1}
Assume that $L_i(0)=0$ for all $i\in [n]$ and that $t\leq e^{-9}\sqrt{\log n}$. We set $\ell=\sqrt{\frac{3\log n}{\log\log n-2\log t}}$. For any $\varepsilon>0$ and sufficiently large $n$, the   $(t+\ell)$-threshold strategy  $f$ satisfies
\begin{equation}\label{eq:single time upper bound case 1}
\P\left(\maxload^f(tn)>(2+\varepsilon)\ell\right)<3n^{-\varepsilon}.
\end{equation}
\end{prop}

\begin{proof}
We write $r:=R_{nt}$ for the total number of retries throughout the process. 
The strategy $f$ guarantees that no bins accept more than $t+\ell$ primary allocations, i.e., $L^f_{1, i}([tn])\leq t+\ell$. This, together with the equation $L_i^f(tn)=L^f_{1, i}([tn])+L^f_{2, i}([tn])-t$, implies that
\begin{equation}\label{eq:upper bound maxload}
\P\left(\maxload^f(tn)>(2+\varepsilon)\ell\right) \leq \P\left(\max_{i\in [n]}L^f_{2, i}([tn])>(1+\varepsilon)\ell\right),
\end{equation}
where $L^f_{2, i}([tn])$ defined in \eqref{eq:primary and secondary load} represents the number balls that bin $i$ receives from secondary allocations. Set $r^*=6ne^{-t I(\ell/t)}/\log(1+\ell/t)$. By the law of total probability, we have
\begin{align}\label{eq:upper bound maxload a}
\P\left(\max_{i\in [n]}L^f_{2, i}([tn])>(1+\varepsilon)\ell\right) &\leq \P\left(\max_{i\in [n]}L^f_{2, i}([tn])>(1+\varepsilon)\ell,  r\leq r^*\right)+\P(r>r^*).
\end{align}

First, we estimate the second term of \eqref{eq:upper bound maxload a}. We write $\{X_i\}_{i\in[n]}$ for independent $\poss(t)$ random variables. Define $Y_i=\max\{0, X_i-t-\ell\}$ and $Y=\sum_{i=1}^nY_i$.  Lemmata \ref{lem:poisson approximation} and \ref{lem:retry bound} provide the following tail bound
\begin{align}\label{eq:retry bound}
\P(r>r^*) &\leq 2\P(Y>r^*)<2\exp\left(-ne^{-t I(\ell/t)}\right) \notag\\
&<2\exp\left(-n\left(\frac{et}{\ell}\right)^{3\ell}\right)=\exp\left(-n^{1-o(1)}\right),
\end{align}
where the last inequality follows from the upper bound in \eqref{eq:approx-I-1} and the fact that $\ell\geq 4t$ for large enough $n$.


Next, we estimate the first term of \eqref{eq:upper bound maxload a}. Again, using  the lower bound in \eqref{eq:approx-I-1},  we obtain $r^*<6n(et/\ell)^{\ell}$ for $n$ large enough. Set $\lambda=6(et/\ell)^{\ell}$. We denote by $\{W_i\}_{i\in[n]}$ independent $\poss(\lambda)$ random variables. Lemma \ref{lem:poisson approximation} and the union bound argument yield
\begin{align}
\P\left(\max_{i\in [n]}L^f_{2, i}([tn])>(1+\varepsilon)\ell,  r\leq r^*\right)&\le \P\left(\max_{i\in [n]}|\{s\le r^*\ :\ Z^2_s=i\}|>(1+\varepsilon)\ell\right)\notag\\
&\leq 2\P\left(\max_{i\in[n]}W_i>(1+\varepsilon)\ell\right) \notag\\
&\leq 2n\P(W_1>(1+\varepsilon)\ell). \label{eq:retry-maxload}
\end{align}
Apply Lemma \ref{lem:poss tail} and  the lower bound of $I(x)$ in \eqref{eq:approx-I-1} to obtain
\begin{equation}\label{eq:single bin retry load}
\P(W_1>(1+\varepsilon)\ell)\leq e^{-\lambda I((1+\varepsilon)\ell/\lambda)}<\left(\frac{6e}{(1+\varepsilon)\ell}\left(\frac{et}{\ell}\right)^{\ell}\right)^{(1+\varepsilon)\ell}<\left(\frac{et}{\ell}\right)^{(1+\varepsilon)\ell^2}.
\end{equation}
One can check that
$$
\left(\frac{et}{\ell}\right)^{(1+\varepsilon)\ell^2}=\exp\left(-(1+\varepsilon)\cdot\frac{3}{2}\left(1-\frac{\log(\log\log n-2\log t)+2-\log 3}{\log\log n-2\log t}\right)\log n\right).
$$
Our assumption of $t$ yields that $\log\log n-2\log t\geq 18$. This, together with the fact that $x^{-1}\log x$ is decreasing for $x>e$, yields that
$$
\frac{\log(\log\log n-2\log t)+2-\log 3}{\log\log n-2\log t}\leq \frac{\log (18)+2-\log 3}{18}<\frac{1}{3}.
$$
Hence, we obtain
$$
\left(\frac{et}{\ell}\right)^{(1+\varepsilon)\ell^2}\leq n^{-(1+\varepsilon)}.
$$
This, combined with \eqref{eq:retry-maxload}, \eqref{eq:single bin retry load}, yields
\begin{equation}\label{eq:retry-maxload 1}
\P\left(\max_{i\in [n]}L^f_{2, i}([tn])>(1+\varepsilon)\ell,  r\leq r^*\right)<2  n^{-\varepsilon}.
\end{equation}
The desired statement \eqref{eq:single time upper bound case 1} follows from \eqref{eq:upper bound maxload}, \eqref{eq:upper bound maxload a}, \eqref{eq:retry bound} and \eqref{eq:retry-maxload 1}.
\end{proof}

Our next result complements the proof of the case $t\leq O(\sqrt{\log n})$. Moreover, it also provides a tight upper bound for the maximum load for $t=(\log n)^{1/2+o(1)}$. 

\begin{prop}\label{prop:single time upper bound case 2}
Assume that $L_i(0)=0$ for all $i\in [n]$ and that $\Omega(\log^{1/2} n)\le t\le o(\log^2 n)$. We set $\ell=(ct\log n)^{1/3}$, where $c$ is an absolute constant such that $\ell\leq t$. For any $\varepsilon>0$ and sufficiently large $n$, the $(t+\ell)$-threshold strategy $f$ satisfies
\begin{equation}\label{eq:single time upper bound case 2}
\P\left(\maxload^f(tn)>\left(\frac{4(1+\varepsilon)}{c}+1\right)\ell\right)<3n^{-\varepsilon}.
\end{equation}
\end{prop}

\begin{proof}
We slightly modify the proof of Proposition \ref{prop:single time upper bound case 1}. Set $r^*=6ne^{-t I(\ell/t)}/\log(1+\ell/t)$. As before, we define independent random variables $\{X_i\}_{i\in[n]}$, $\{Y_i\}_{i\in[n]}$ and $\{W_i\}_{i\in[n]}$ where $X_i\sim \poss(t)$, $Y_i=\max\{0,X_i-t-\ell\}$ and $W_i\sim \poss(\lambda)$ for $\lambda=r^*/n$. As before, we set $r:=R_{nt}$. Similar to \eqref{eq:upper bound maxload}, \eqref{eq:upper bound maxload a} and  \eqref{eq:retry-maxload}, we have
\begin{align}\label{eq:upper bound maxload case 2}
\P\left(\maxload^f(tn)>\left(\tfrac{4(1+\varepsilon)}{c}+1\right)\ell\right) &\leq \P\left(\max_{i\in [n]}L^f_{2, i}([tn])>\frac{4(1+\varepsilon)\ell}{c}\right)\notag\\
& \leq \P\left(\max_{i\in [n]}L^f_{2, i}([tn])>\frac{4(1+\varepsilon)\ell}{c},r<r^*\right)+\P(r>r^*)\notag\\
&\leq 2\P\left(\max_{i\in[n]}W_i>\frac{4(1+\varepsilon)\ell}{c}\right)+\P(r>r^*)\notag\\
&\leq 2n\P\left(W_1>\frac{4(1+\varepsilon)\ell}{c}\right)+\P(r>r^*).
\end{align}
Similar to \eqref{eq:retry bound}, Lemmata \ref{lem:poisson approximation} and \ref{lem:retry bound} yield that
\begin{align}\label{eq:retry bound 1}
\P(r>r^*)&<2\exp\left(-ne^{-t I(\ell/t)}\right)<2\exp\left(-n\exp\left(-\frac{\ell^2}{2t}\right)\right)\notag\\
&<2\exp\big(-ne^{-\ell}\big)=\exp\big(-n^{1-o(1)}\big),
\end{align}
where the last two inequalities follow from the upper bound of $I(x)$ in \eqref{eq:approx-I} and the fact that $\ell\leq t$. Using the lower bound of $I(x)$ in \eqref{eq:approx-I} and $\log(1+x)>x/2$ for $0< x< 1$, one can check that $\lambda=r^*/n<\frac{12t}{\ell}\exp\left(-\frac{\ell^2}{4t}\right)=o(1)$. This, together with Lemma \ref{lem:poss tail} and inequality $I(x)>x\log(x/e)$ for $x>4$, yields
\begin{align*}
\P\left(W_1>\frac{4(1+\varepsilon)\ell}{c}\right) &\leq \exp\left(-\lambda I\left(\frac{3(1+\varepsilon)\ell}{c\lambda}\right)\right)\leq \left(\frac{ce\lambda}{3(1+\varepsilon)\ell}\right)^{\frac{3(1+\varepsilon)\ell}{c}}\notag\\
&\leq \left(\frac{4cet}{(1+\varepsilon)\ell^2}\exp\left(-\frac{\ell^2}{3t}\right)\right)^{\frac{3(1+\varepsilon)\ell}{c}}\notag\\
&<\exp\left(-\frac{(1+\varepsilon)\ell^3}{ct}\right)=n^{-(1+\varepsilon)}.
\end{align*}
Combining this with \eqref{eq:upper bound maxload case 2} and \eqref{eq:retry bound 1}, we can obtain \eqref{eq:single time upper bound case 2}.
\end{proof}


\subsection{Case 2:  \texorpdfstring{$\Omega(\sqrt{\log n})\le t\le O(\log n)$}{Ω(√(log n))≤t≤O(log n)}}\label{sec:case 2}

For $t=O\big((\log n)^{\frac{1}2+\frac{1}{\sqrt{\log\log\log n}}}\big)$, Theorem~\ref{thm:single-time load discrepancy} follows from Proposition~\ref{prop:single time upper bound case 2}. Thus, here we treat
$$\Omega\big((\log n)^{\frac{1}2+\frac{1}{\sqrt{\log\log\log n}}}\big) \le t\le O(\log n).$$
In this subsection, we study the allocation problem in a more general setting. The initial loads are not necessarily perfectly balanced (i.e., allowing $L_i(0)\neq 0$).  
This will play an important role in Sections \ref{sec:case 4} and \ref{sec:typical load discrepancy}.

Recall that $k=\big\lfloor\frac{\log\log n}{3\log\log\log n}\big\rfloor$. Set $\ell=\lfloor \log^{\beta_k}n\rfloor$, where $\beta_k$ is defined in Section~\ref{sec:combin}. One can check that $\ell=\big\lfloor(\log n)^{\frac{1}{2}+\left(2-\frac{1}{2k+1}\right)\frac{\alpha+\eta-1/2}{2k+1}}\big\rfloor$. Then we have the following result.

\begin{prop}\label{prop:single time upper bound case 3}
Let $t>0$ and $\alpha=\frac{\log t}{\log\log n}$ satisfying $\alpha\in\big[\tfrac{1}{2}+\tfrac{1}{\sqrt{\log\log\log n}},1+\tfrac{\sqrt{\log\log\log n}}{\log\log n}\big].$ Suppose that for $L_0 \ge0$ the following conditions hold: \phantom{$\big]$}
\begin{enumerate}
\item $\maxload (0)<c t$ for some constant $0<c<1$,
\item $|H_0| \le 3n \exp\left(-\frac{\ell^2}{4\log^{\alpha+\eta} n}\right)$, where $H_0=\{i\in [n]: L_i(0)>  L_0\}$ is the set of bins with load greater than $L_0$.
\end{enumerate} 
Then the multi-stage $(t,L_0,\ell)$-threshold strategy $f$ (as defined in Section~\ref{sec:combin}), with the parameters above, satisfies that
$$
\P\left(\maxload^f(tn)>L_0+2k\ell\right)\leq n^{-e^{\sqrt{\log\log\log n}}}.
$$
\end{prop}


For
$\omega(n \sqrt{\log n}) \le m\le O(\log n)$, Theorem~\ref{thm:single-time load discrepancy} follows as an immediate consequence of the following corollary.

\begin{coro}\label{prop:single time upper bound case 3 for thm 1.1}
Let $t>0$ and $\alpha$ as above, satisfying $\alpha\in\big[\frac{1}{2}+\frac{1}{\sqrt{\log\log\log n}},1+\frac{\sqrt{\log\log\log n}}{\log\log n}\big]$. The multi-stage $(t,0, \ell)$-threshold strategy $f$
satisfies that
$$
\P\left(\maxload^f(tn)>(\log n)^{\frac{1}{2}+o(1)}\right)\leq n^{-e^{\sqrt{\log\log\log n}}}.
$$
\end{coro}
\begin{proof}
Apply Proposition~\ref{prop:single time upper bound case 3} with $\eta=0$ and $L(0)=L_0=0$ and observe that the two conditions of Proposition~\ref{prop:single time upper bound case 3} trivially hold. Hence the corollary follows from the fact that $\ell =(\log n)^{\frac{1}{2}+o(1)}$ and $k=\log^{o(1)} n$.  
\end{proof}

For $1\leq i\leq k$, we denote by $r_i$ be the number of retries in stage $i$ of the multi-stage $(t,L_0,\ell)$-threshold strategy.  
Recall our notation $H_i$ for the set of bins in $(\cup_{j=0}^{i-1}H_j)^c$ whose loads after the $i$-th stage are at least $L_0+2i\ell$. To establish Proposition~\ref{prop:single time upper bound case 3}, we use the following lemma, to inductively bound the number of retries in every stage and the size of $H_i$, the set of heavily loaded bins. 

\begin{lem}\label{lem:bound ri-hi}
Under the assumptions of Proposition~\ref{prop:single time upper bound case 3}, for all $1\leq i\leq k$, we have
\begin{equation}\label{eq:bound ri}
\P(r_i>r_i^*)\leq \exp\left(-n^{1/2-o(1)}\right),
\end{equation}
where
\begin{equation}\label{eq:r_i^*}
r_i^*:=\frac{20n\log^{\beta_{i-1}}n}{\ell}\exp\left(-\frac{\ell^2}{5\log^{\beta_{i-1}}\! n}\right).
\end{equation}
In addition, for $1\leq i\leq k-1$, we have
\begin{equation}\label{eq:bound hi}
\P\left(|H_i|> \frac{4n\lambda_i^\ell}{\ell!}\right)\leq \exp\left(-n^{1/2-o(1)}\right),
\end{equation}
where $\lambda_i:=r_i^*/n$.
\end{lem}

Before presenting the proof, we first make some technical observations. Using $\ell=\lfloor \log^{\beta_k}\! n\rfloor$ and $\beta_k=\beta-\frac{k(2\beta-1-\varepsilon)}{2k+1}$,
it is easy to check that
\begin{equation}\label{eq:technical-bound-1}
\ell\cdot \log\ell\leq(\log n)^{\frac{k(1+\varepsilon)+\beta}{2k+1}}
\cdot \frac{k(1+\varepsilon)+\beta}{2k+1}\log\log n=(\log n)^{\frac{1}{2}+o(1)}.
\end{equation}
For $n$ large enough, we have
\begin{align}\label{eq:technical-bound-2}
\frac{\ell^2}{\log^\beta\! n} 
&=\frac{\ell^2}{\log^{2\beta_k}\! n}\cdot (\log n)^{2\beta_k-\beta}>\frac{1}{2}(\log n)^{\frac{2k(1+\varepsilon)-(2k-1)\beta}{2k+1}}\notag\\
&> \frac{1}{2}(\log n)^{\frac{1-(2k-1)(\beta-1)}{2k+1}}\geq(\log\log n)^{\frac{9}{8}-o(1)},
\end{align}
where the last equality follows from that $\beta-1<\frac{1/4+o(1)}{2k-1}$ and  our choice of $k$. We also have
\begin{align}\label{eq:technical-bound-3}
\frac{\ell^3}{\log^\beta\! n}&\leq(\log n)^{2\beta-\frac{3k(2\beta-1-\varepsilon)}{2k+1}}=(\log n)^{1-\frac{2k-2}{2k+1}\big(\beta-\frac{1}{2}\big)+\frac{3k\varepsilon}{2k+1}} \notag\\
&<(\log n)^{1-\frac{2k-2}{2k+1}\big(\beta-\frac{1}{2}\big)+\frac{3}{2k+1}\big(\beta-\frac{1}{2}\big)}\notag\\
&=(\log n)^{1-\big(\beta-\frac{1}{2}\big)}=o(\log n),
\end{align}
where the last equality uses $\beta\geq\alpha>\frac{1}{2}+\frac{1}{\sqrt{\log\log\log n}}$. 
For $1\leq i\leq k$ and $n$ large enough, we have
\begin{equation}\label{eq:technical-bound-4}
0<\lambda_i<1
\end{equation}
To see this, notice that $\{\beta_i\}_{i=0}^k$ is a decreasing arithmetic progression, hence, $\{\lambda_i\}_{i=1}^k$ is a decreasing sequence and it suffices to show that  $0<\lambda_1<1$.  Observe that
$$
\lambda_1=\frac{20\log^\beta\! n}{\ell}\exp\left(-\frac{\ell^2}{5\log^\beta\! n}\right)=\exp\left(-\frac{\ell^2}{5\log^\beta\! n}+\log\frac{20\log^\beta\! n}{\ell}\right),
$$
and
$$
\log\frac{\log^\beta\! n}{\ell}<\beta\log\log n<2\log\log n.
$$
This, together with \eqref{eq:technical-bound-2}, yields that $0<\lambda_1<1$ and hence \eqref{eq:technical-bound-4}.

\begin{proof}
We prove \eqref{eq:bound ri} and \eqref{eq:bound hi} inductively by establishing the $i$-th case of \eqref{eq:bound ri} on condition that \eqref{eq:bound hi} holds for all $j<i$, and by establishing the $i$-th case of \eqref{eq:bound hi} on condition that \eqref{eq:bound ri} holds for the same $i$.
The case $i=1$ is treated separately.

\textbf{Bounding $\P(r_i>r_i^*)$ assuming that $|H_j|\leq 4n\lambda_j^\ell/\ell!$ for $j<i$}.
We denote by $r_{i, 1}$ the number of balls in the $i$-th stage whose primary allocations are bins that, at the time of the allocation, already accepted $t_i-t_{i-1}+\ell$ primary allocations during stage $i$. We write $r_{1, 2}$ for the number balls  in the first stage whose primary allocations are bins from $H_0$, and write $r_{i, 2}$ for $i\geq 2$, for the number of balls in the $i$-th stage whose primary allocations are bins from $\cup_{j=1}^{i-1}H_j$. By the definition of the strategy, we thus have $r_i\leq r_{i, 1}+r_{i, 2}$.

\textbf{Estimating $r_{i, 1}$.}
Recall that $t_i=\lfloor t-\log^{\beta_i}\! n\rfloor$ for $1\leq i\leq k-1$, $t_k=t$, $\ell=\lfloor\log^{\beta_k}\! n\rfloor$, where $\beta_i=\beta-\frac{(2\beta-1-\varepsilon)i}{2k+1}$, and observe that $\ell<t_i-t_{i-1}$ for $1\leq i\leq k$.  
Also, recall that $I(x)$ defined in Lemma \ref{lem:poss tail} is the rate function of the large deviation bound of a Poisson random variable.
We have
\begin{equation}\label{eq:bound ri-1-a}
\frac{3}{5}r_i^*> \frac{12n(t_i-t_{i-1})}{\ell}\exp\left(-\frac{\ell^2}{4(t_i-t_{i-1})}\right)\ge\frac{6ne^{-(t_i-t_{i-1})I(\ell/(t_i-t_{i-1}))}}{\log(1+\ell/(t_i-t_{i-1}))}=:r^*,
\end{equation}
where the first inequality follows from the definition of $r_i^*$ in \eqref{eq:r_i^*}, and the second inequality follows from the lower bound of $I(x)$ in \eqref{eq:approx-I} and that $\log(1+x)\geq x/2$ for $0<x<1$ (indeed  $\ell/(t_i-t_{i-1})<1$). Define $Y_j^{(i)}=\max\big\{0, X_j^{(i)}-(t_i-t_{i-1}+\ell)\big\}$, where 
$\big\{X_j^{(i)}\big\}_{j\in[n]}$ is a collection of independent $\poss(t_i-t_{i-1})$ random variables, and write $Y=\sum_{j=1}^nY_j^{(i)}$.
 By Lemmata  \ref{lem:poisson approximation}, \ref{lem:retry bound} and inequality \eqref{eq:bound ri-1-a}, we have
\begin{align}\label{eq:bound ri-1}
\P\left(r_{i, 1}>\frac{3}{5}r_i^*\right) &\leq 2\P\left(Y>\frac{3}{5}r_i^*\right) \leq 2\P\left(Y>r^*\right)\notag\\
&\leq2\exp\left(-n\exp\left(-(t_i-t_{i-1})I\left(\frac{\ell}{t_i-t_{i-1}}\right)\right)\right) \notag\\
& \leq 2\exp\left(-n\exp\left(-\frac{\ell^2}{2(t_i-t_{i-1})}\right)\right)\notag\\
&\leq 2\exp\big(-ne^{-\ell}\big)=\exp\big(-n^{1-o(1)}\big),
\end{align}
where the  last two inequalities follow from the upper bound of $I(x)$ in \eqref{eq:approx-I} and the fact that $\ell/(t_i-t_{i-1})<1$. The last identity follows from the fact that $\ell=(\log n)^{\frac{1}{2}+o(1)}$.

\textbf{Estimating $r_{1, 2}$.}
Here we estimate the number of balls  in the first stage whose primary allocations are bins from $H_0$. Using the assumption $|H_0|\leq 3n\exp\big(-\frac{\ell^2}{4\log^\beta\! n}\big)$, we have for $n$ large enough
\begin{equation}\label{eq:r-1-2-a}
2(t_1-t_0)|H_0|\leq 6n(t_1-t_0)\exp\left(-\frac{\ell^2}{4\log^\beta\! n}\right)\leq \frac{6n\log^{\beta}\! n}{\ell}\exp\left(-\frac{\ell^2}{5\log^\beta\! n}\right)<\frac{2}{5}r_1^*,
\end{equation}
where $r_1^*$ is given in \eqref{eq:r_i^*}, and the second inequality follows from $t_1-t_0<\log^\beta\! n$ and the observation that $\ell=o\big(\exp\big(\frac{\ell^2}{\log^{\beta}\! n}\big)\big)$ by \eqref{eq:technical-bound-2}. 
We denote by $\big\{X_j^{(i)}\big\}_{j\in[n]}$  independent $\poss(t_i-t_{i-1})$ random variables, and write $Z$ for a
Poisson random variable with the parameter $3n(t_1-t_0)\exp\left(-\frac{\ell^2}{4\log^\beta\! n}\right)$. Lemmata \ref{lem:poisson approximation}, \ref{lem:poss tail} and inequality \eqref{eq:r-1-2-a} yield
\begin{align}\label{eq:r1-2}
\P\left(r_{1, 2}>\frac{2}{5}r_1^*\right) &\leq 2\P\left(\sum_{j\in H_0}X_j^{(1)}>\frac{2}{5}r_1^*\right)\leq 2\P\left(Z>\frac{2}{5}r_1^*\right) \notag\\
&\leq 2\P\left(Z>6n(t_1-t_0)\exp\left(-\frac{\ell^2}{4\log^\beta\! n}\right)\right)\notag\\
&\leq 2\exp\left(-n(t_1-t_0)\exp\left(-\frac{\ell^2}{4\log^\beta\! n}\right)\right) \notag\\
&\leq 2\exp\left(-n(t_1-t_0)e^{-\ell}\right)\notag\\
&=\exp\big(-n^{1-o(1)}\big).
\end{align}
This, together with the $i=1$ case of \eqref{eq:bound ri-1}, implies the base case of \eqref{eq:bound ri}, i.e., $i=1$.

\textbf{Estimating $r_{i, 2}$ for $i\ge 2$, assuming that $|H_j|\leq \frac{4n\lambda_j^\ell}{\ell!}$  for $1\leq j\leq i-1$.} Recall that $r_{i, 2}$ is the number of balls in stage $i$ whose primary allocations are bins from $\cup_{j=1}^{i-1}H_j$. 
Again, write $\big\{X_j^{(i)}\big\}_{j\in[n]}$ for independent $\poss(t_i-t_{i-1})$ random variables, and write $Z$ for a Poisson random variable with parameter $\frac{4n(t_i-t_{i-1})}{\ell!}\sum_{j=1}^{i-1}\lambda_j^\ell$.
Define $E=\big\{|H_j|\leq \frac{4n\lambda_j^\ell}{\ell!}, 1\leq j\leq i-1\big\}$. Lemmata \ref{lem:poisson approximation} and  \ref{lem:poss tail} imply that
\begin{align}\label{eq:bound ri-2}
\P\left(r_{i, 2}> \frac{8n(t_i-t_{i-1})}{\ell!}\sum_{j=1}^{i-1}\lambda_j^\ell, ~E\right) &\le 
2\P\left(\sum_{m \in \cup_{j=1}^{i-1}H_j}X_m^{(i)}> \frac{8n(t_i-t_{i-1})}{\ell!}\sum_{j=1}^{i-1}\lambda_j^\ell, ~E\right) \notag\\
&\leq  2\P\left(Z> \frac{8n(t_i-t_{i-1})}{\ell!}\sum_{j=1}^{i-1}\lambda_j^\ell\right) \notag\\
&\leq \exp\left(-\frac{n(t_i-t_{i-1})}{\ell!}\sum_{j=1}^{i-1}\lambda_j^\ell \right) \leq \exp\left(-\frac{n\lambda_1^\ell}{\ell!}\right)\notag\\
&\leq \exp\left(-\exp\left(\log n-\frac{\ell^3}{\log^\beta\! n}-\ell\log \ell\right)\right)\notag\\
&=\exp\left(-n^{1-o(1)}\right),
\end{align}
where the penultimate transition uses the fact that $\lambda_1=r_1^*/n>\exp\big(-\frac{\ell^2}{4\log^\beta\! n}\big)$, where $r_1^*$ is given in \eqref{eq:r_i^*}, and the bound $\ell!<\ell^\ell$, and the last transition uses \eqref{eq:technical-bound-1} and \eqref{eq:technical-bound-3}. Using the fact that $0<\lambda_i<1$ and that $k<\ell$, we have
\begin{align}\label{eq:bound ri-2-a}
\frac{8n(t_i-t_{i-1})}{\ell!}\sum_{j=1}^{i-1}\lambda_j^\ell&<\frac{8kn\log^{\beta_{i-1}}\! n}{\ell!}<
\frac{8n\log^{\beta_{i-1}}\! n}{(\ell-1)!}<\frac{8n\log^{\beta_{i-1}}\! n}{\ell}\exp\left(-\frac{\ell \log \ell}{2}\right)\notag\\
&\le \frac{8n\log^{\beta_{i-1}}n}{\ell}\exp\left(-\frac{\ell^2}{5\log^{\beta_{i-1}}\! n}\right)=\frac{2}{5}r_i^*,
\end{align}
where the penultimate inequality uses Stirling's approximation and the last inequality follows the fact that $\ell<\log^{\beta_{i-1}} n$.
Combining \eqref{eq:bound ri-2} and \eqref{eq:bound ri-2-a}, we have
\begin{equation*}\label{eq:bound ri-2-b}
\P\left(r_{i, 2}> \frac{2}{5}r_i^*,~E\right)
= \exp\left(-n^{1-o(1)}\right).
\end{equation*}
This, together with \eqref{eq:bound hi} for $1\leq j\leq i-1$, implies that for $2\leq i\leq k$,
$$
\P\left(r_{i, 2}> \frac{2}{5}r_i^*\right)\leq\P\left(r_{i, 2}> \frac{2}{5}r_i^*,~ E\right)+\sum_{j=1}^{i-1}\P\left(|H_j|> \frac{4n\lambda_j^\ell}{\ell!}\right)= \exp\left(-n^{1/2-o(1)}\right).
$$
This, combined with \eqref{eq:bound ri-1} and \eqref{eq:r1-2}, yields
$$
\P(r_i>r_i^*)\leq \P\left(r_{i, 1}> \frac{3}{5}r_i^*\right)+\P\left(r_{i, 2}> \frac{2}{5}r_i^*\right)=\exp\left(-n^{1/2-o(1)}\right).
$$
This concludes the proof of the $i$-th case of \eqref{eq:r_i^*} condition on that \eqref{eq:bound hi} holds for $j<i$.

\textbf{Bounding $\P(|H_i|>2p_in)$ assuming that $r_i\leq r_i^*$ for $i\ge 1$.} 
Recall that $$H_i=\big\{j\in [n]: L_j^f(t_in)\geq L_0+2i\ell\big\}\setminus \bigcup_{i'<i}H^c_{i'}.$$ 
Let $j\in H_i$. We have $L_j^f(t_{i-1}n)<L_0+2(i-1)\ell$ (otherwise we would have $j\in H_{i-1}$). Let us show that $j$ must have received at least $\ell$ secondary allocations in the $i$-th stage.  During the $i$-th stage, if bin $j$ accepted less than $t_i-t_{i-1}+\ell$ primary allocations, it clearly must have received at least $\ell$ secondary allocations in order to belong to $H_i$. Otherwise, once $j$ accepted more than $t_i-t_{i-1}+\ell$ primary allocations (in the $i$-th stage), it rejects all further allocations unless its load is at most $-\log n$. Hence its load after accepting the last primary allocation must have been at most $-\log n$, so that in order to belong to $H_i$ it must have received at least $L_0+2i\ell+\log n$ secondary allocations.

Let $\big\{X_j^{(i)}\big\}_{j\in[n]}$ be independent $\poss(\lambda_i)$ random variables. Let $Y_j^{(i)}$ be the indicator function of the event that $X_j^{(i)}\geq \ell$. Then, $\big\{Y_j^{(i)}\big\}_{j\in[n]}$ are independent $\bernoulli(p_i)$ random variables, where $p_i=\P\big(X_1^{(i)}\geq \ell\big)$. Let $Y=\sum_{j=1}^nY_j^{(i)}$. By Lemma \ref{lem:poisson approximation} and Hoeffding's inequality,
\begin{equation}\label{eq:conditional bound hi}
\P(|H_i|>2p_in ,r_i\leq r_i^*)\leq 2\P\left(Y>2p_in\right)\leq 2e^{-2np_i^2}.
\end{equation}
Using the fact that $0<\lambda_i<1$, we have
\begin{equation*}
\frac{\lambda_i^\ell}{e\ell!}<p_i=e^{-\lambda_i}\sum_{j=\ell}^\infty\frac{\lambda_i^j}{j!}<\frac{2\lambda_i^\ell}{\ell!}.
\end{equation*}
This, together with \eqref{eq:conditional bound hi}, yields that, for $1\leq i\leq k-1$,
\begin{equation}\label{eq:conditional bound hi-1}
\P\left(|H_i|>\frac{4n\lambda_i^\ell}{\ell!} , r_i\leq r_i^*\right)<2\exp\left(-\frac{2n}{e^2}\left(\frac{\lambda_i^\ell}{\ell!}\right)^2\right).
\end{equation}
Since $\{\lambda_i\}_{i=1}^k$ is a decreasing sequence, we will upper bound the RHS of \eqref{eq:conditional bound hi-1} for $i=k-1$. Using the fact that $\ell!\leq e\sqrt \ell (\ell/e)^\ell$, by Stirling's approximation, we obtain
\begin{equation}\label{eq:conditional bound hi-2}
\frac{2n}{e^2}\left(\frac{\lambda_{k-1}^\ell}{\ell!}\right)^2\geq \frac{2n}{e^4\ell}\left(\frac{e\lambda_{k-1}}{\ell}\right)^{2\ell}=\frac{2}{e^4}\exp\left(\log n-\log\ell+2\ell\log \frac{e\lambda_{k-1}}{\ell}\right).
\end{equation}
Using $\ell=\lfloor\log^{\beta_k}\! n\rfloor<\log^{\beta_{k-1}}\! n$, \eqref{eq:r_i^*} and  \eqref{eq:bound hi}, we have
$
\lambda_{k-1}
>\exp\left(-\frac{\ell^2}{5\log^{\beta_{k-2}}n}\right),
$
and
\begin{align}\label{eq:conditional bound hi-3}
\ell\log \lambda_{k-1} &>\frac{-\ell^3}{5\log^{\beta_{k-2}}\! n}=-\frac{\ell^3}{5\log^{3\beta_k}\! n}\cdot(\log n)^{4\beta_k-\beta_{k-2}}\notag\\
&>-\frac{1}{4}(\log n)^{1+\frac{2(k+1)\varepsilon-(2\beta-1)}{2k+1}}=-\frac{\log n}{4},
\end{align}
where the second equality follows from that $\ell^3/\log^{3\beta_k}n=1-o(1)$ and $\beta_i=\beta-\frac{(2\beta-1-\varepsilon)i}{2k+1}$, and the last equality uses $\varepsilon= \frac{2\beta-1}{2(k+1)}$.
Combining \eqref{eq:conditional bound hi-1}, \eqref{eq:conditional bound hi-2}, \eqref{eq:conditional bound hi-3} and \eqref{eq:technical-bound-1}, we have
\begin{equation*}\label{eq:conditional bound hi-8}
\P\left(|H_i|> \frac{4n\lambda_i^\ell}{\ell!},r_i\leq r_i^*\right)\leq\exp\left(-n^{1/2-o(1)}\right).
\end{equation*}
This, together with \eqref{eq:bound ri}, implies that 
$$
\P\left(|H_i|> \frac{4n\lambda_i^\ell}{\ell!}\right)\leq\P\left(|H_i|> \frac{4n\lambda_i^\ell}{\ell!},r_i\leq r_i^*\right)+\P(r_i>r_i^*)\leq\exp\left(-n^{1/2-o(1)}\right).
$$
This concludes the proof of the $i$-th case of \eqref{eq:bound hi} given that the $i$-th case of \eqref{eq:bound ri} holds. This establishes the induction and thus the lemma.
\end{proof}

In the next lemma, we keep our notation  $r_i$ for the number of retries in the $i$-th stage, which proceeds from $t_{i-1}$ to $t_i$ and set $t_{k+1}:=t_k+\ell$. 

\begin{lem}\label{lem:retry maxload bound}
For $1\leq i\leq k$ we have
\begin{equation}\label{eq:retry maxload bound}
\P\left(\exists_{j\in[n]}L^f_{2, j}((t_{i-1},t_{i}])>t_{i+1}-t_i\right)\leq 2n^{-e^{2\sqrt{\log\log\log n}}}.
\end{equation}
\end{lem}

\begin{proof}
Denote $E=\big\{\exists_{j\in[n]}L^f_{2, j}((t_{i-1},t_{i}]>t_{i+1}-t_i\big\}$.
Recall that $r_i^*$ is defined in \eqref{eq:r_i^*}. Using the law of total probability, we have
\begin{equation}\label{eq:retry maxload bound-1}
\P\left(E\right)\leq \P\left(E~\middle |~r_i\leq r_i^*\right)+\P(r_i>r_i^*).
\end{equation}
We have already showed in Lemma \ref{lem:bound ri-hi} that
\begin{equation}\label{eq:retry maxload bound-2}
\P(r_i>r_i^*)\leq \exp\left(-n^{1/2-o(1)}\right).
\end{equation}
Next, we estimate the first term on the RHS of \eqref{eq:retry maxload bound-1}. Denote by $\{X_j^{(i)}\}_{j\in[n]}$ independent $\poss(\lambda_i)$ random variables, where $\lambda_i$ is given in \eqref{eq:bound hi}. By Lemma \ref{lem:poisson approximation}, we have
\begin{equation}\label{eq:retry maxload bound-3}
\P\left(E~\middle |~r_i\leq r_i^*\right)\leq 2\P\left(\exists_{j\in[n]}X_j^{(i)}> t_{i+1}-t_i\right).
\end{equation}
Using the face that $0<\lambda_i<1$ in \eqref{eq:technical-bound-4}, we have
\begin{align*}
\P\left(X_1^{(i)}\geq t_{i+1}-t_i\right) &<\frac{2\lambda_i^{ t_{i+1}-t_i}}{( t_{i+1}-t_i)!}\leq \left(\frac{e\lambda_i}{t_{i+1}-t_i}\right)^{t_{i+1}-t_i}\\
&\leq \exp\left(-\frac{\ell^2(t_{i+1}-t_i)}{4\log^{\beta_{i-1}}\! n}\right)\leq \exp\left(-\frac{\ell^2\log^{\beta_i}\! n}{5\log^{\beta_{i-1}}\! n}\right)\\
&\leq \exp\left(-\frac{(\log n)^{2\beta_k+\beta_i}}{6\log^{\beta_{i-1}}\! n}\right)=\exp\left(-\frac{1}{6}\log^{1+\varepsilon} n\right).
\end{align*}
The second inequality follows from Stirling's approximation $n!\geq \sqrt{2\pi n}(n/e)^n$ for $n\in\Z_+$. The transition to the second line uses the definition of $\lambda_i$ given in \eqref{eq:bound hi}. In the penultimate inequality, we use $t_i=\lfloor t-\log^{\beta_i}\! n\rfloor$, where $\beta_i=\beta-\frac{(2\beta-1-\varepsilon)i}{2k+1}$, and that $\log^{\beta_{i+1}}\! n=o(\log^{\beta_i}\! n)$. The last inequality uses the fact that $\ell= \lfloor\log^{\beta_k}\! n\rfloor$. Taking into account of $\varepsilon=\frac{2\beta-1}{2(k+1)}$, $k=\big\lfloor\frac{\log\log n}{3\log\log\log n}\big\rfloor$ and $\beta>\frac{1}{2}+\frac{1}{\sqrt{\log\log\log n}}$, we have $\log^\varepsilon n\geq e^{3\sqrt{\log\log\log n}}$. Taking the union bound, we have for $n$ large enough,
\begin{equation}\label{eq:retry maxload bound-4}
\P\left(\max_{j\in[n]}X_j^{(i)}> t_{i+1}-t_i\right)\leq n\exp\left(-\frac{1}{6}\log^{1+\varepsilon} n\right)\leq n^{-e^{2\sqrt{\log\log\log n}}}.
\end{equation}
The desired statement \eqref{eq:retry maxload bound} follows from \eqref{eq:retry maxload bound-1}--\eqref{eq:retry maxload bound-4}.
\end{proof}

Now we are ready to prove  Proposition \ref{prop:single time upper bound case 3}.

\begin{proof}[Proof of Proposition \ref{prop:single time upper bound case 3}]
We will estimate the maximum loads after $i$ stages for all $1\leq i\leq k$. By the definition of $H_i$, we have
\begin{equation}\label{eq:maxload-non-heavily-loaded-set}
\maxload^f_{(\cup_{j=0}^i H_j)^c}(t_in)\leq L_0+2i\ell.
\end{equation}

Next, we estimate the maximum load over $\cup_{j=1}^iH_j$ after $i$ stages. For $1\leq j\leq i\leq k$, we denote by $E_i^j=\big\{\maxload^f_{H_j}(t_in) > t_{i+1}-t_i+L_0+(2j-1)\ell\big\}$, where $t_{k+1}=t_k+\ell$. We will show that
\begin{equation}\label{eq:maxload-heavily-loaded-set}
\P\big(E_i^j\big)\leq (i-j+1)\cdot 2n^{-e^{2\sqrt{\log\log\log n}}}.
\end{equation}
We denote by $r_i$ the number of retries in the $i$-th stage. In the $i$-th stage, for a bin in $H_i$ to accept more than $t_i-t_{i-1}+\ell$ primary allocations, it is necessary that the load of this bin before accepting its last primary allocation is at most $-\log n$. Hence, we have
\begin{align*}
\maxload^f_{H_i}(t_in) &\leq \max\left\{\maxload^f_{H_i}(t_{i-1}n)+\ell, -\log n\right\}+\max_{p\in H_i}L^f_{2, p}((t_{i-1},t_{i}])\\
&\leq L_0+(2i-1)\ell+\max_{p\in H_i}L^f_{2, p}((t_{i-1},t_{i}]),
\end{align*}
where the second inequality uses the fact that $H_i\subseteq (\cup_{j=0}^{i-1} H_j)^c$ and the $i-1$ case of \eqref{eq:maxload-non-heavily-loaded-set}. Using the inequalities above and Lemma \ref{lem:retry maxload bound}, we obtain
\begin{equation}\label{eq:maxload-hi}
\P\big(E_i^i\big)\leq \P\left(\max_{p\in H_i}L^f_{2, p}((t_{i-1},t_{i}])> t_{i+1}-t_i\right)\leq 2n^{-e^{2\sqrt{\log\log\log n}}}.
\end{equation}
 For  $1\leq j\leq i-1$ and $i\geq 2$, the strategy guarantees that in the $i$-th stage, each bin of $H_j$ either accepts no primary allocations, or has a load at most $-\log n$ before accepting its last primary allocation. Hence, we have
$$
\maxload^f_{H_j}(t_in) \leq \max\left\{\maxload^f_{H_j}(t_{i-1}n)-(t_i-t_{i-1}), -\log n\right\}+\max_{p\in H_j}L^f_{2, p}((t_{i-1},t_{i}]).
$$
Hence, event $E_i^j$ occurs only if one of the two conditions holds: $\max_{p\in H_j}L^f_{2, p}(r_i)> t_{i+1}-t_i$ or  $\max\big\{\maxload^f_{H_j}(t_{i-1}n)-(t_i-t_{i-1}), -\log n\big\}>L_0+(2j-1)\ell$. The latter condition is equivalent to event  $E_{i-1}^j$. This and Lemma \ref{lem:retry maxload bound} imply that
$$
\P\big(E_i^j\big)\leq \P\big(E_{i-1}^j\big)+\P\left(\max_{p\in H_i}L^f_{2, p}((t_{i-1},t_{i}])>t_{i+1}-t_i\right)\leq \P\big(E_{i-1}^j\big)+2n^{-e^{2\sqrt{\log\log\log n}}}.
$$
Iterating this argument to obtain
$$
\P\big(E_i^j\big)\leq \P\big(E_j^j\big)+(i-j)\cdot n^{-e^{2\sqrt{\log\log\log n}}}\leq (i-j+1)\cdot 2n^{-e^{2\sqrt{\log\log\log n}}},
$$
where the second inequality follows from \eqref{eq:maxload-hi}. This concludes the proof of \eqref{eq:maxload-heavily-loaded-set}.

Now, we estimate the maximum load over $H_0$. In the first stage, each bin in $H_0$ either accepts no primary allocations or has a load at most $-\log n$ before accepting its last primary allocation. Hence, we have
\begin{equation}\label{eq:maxload-h0-iteration-a}
\maxload^f_{H_0}(t_1n) \leq \max\left\{\maxload^f_{H_0}(t_0n)-(t_1-t_0), -\log n\right\}+\max_{p\in H_0}L^f_{2, p}((t_0,t_{1}]).
\end{equation}
In general, in the $i$-th stage for $2\leq i\leq k$, for bin of $H_0$ to accept more than $t_i-t_{i-1}+\ell$ primary allocations, the load of this bin before accepting its last primary is at most $-\log n$. Hence, we obtain
\begin{equation}\label{eq:maxload-h0-iteration-b}
\maxload^f_{H_0}(t_in) \leq \max\left\{\maxload^f_{H_0}(t_{i-1}n)+\ell, -\log n\right\}+\max_{p\in H_0}L^f_{2, p}((t_{i-1},t_{i}]).
\end{equation}
Iteration of \eqref{eq:maxload-h0-iteration-b}, together with \eqref{eq:maxload-h0-iteration-a}, yields
\begin{align}\label{eq:maxload-h0}
\maxload^f_{H_0}(t_in) &\leq \max\left\{\maxload^f_{H_0}(t_1n)+\ell, -\log n\right\}+(i-2)\ell+\sum_{j=2}^i\max_{p\in H_0}L^f_{2, p}((t_{j-1},t_{j}]) \notag\\
&\leq \max\left\{\maxload^f_{H_0}(t_0n)-(t_1-t_0), -\log n\right\}\notag\\&\hspace{12pt}+(i-1)\ell+\sum_{j=1}^i\max_{p\in H_0}L^f_{2, p}((t_{j-1},t_{j}]) \notag\\
&\leq \sum_{j=1}^i\max_{p\in H_0}L^f_{2, p}((t_{j-1},t_{j}])+(i-1)\ell-\min\big\{(1-c-o(1))(t-t_0)), \log n\big\},
\end{align}
where the last inequality follows from the fact that $\maxload^f_{H_0}(t_0n)\leq c(t-t_0)$ for some constant $0<c<1$, and that  $t_1-t_0=(1-o(1))(t-t_0)$. Observe that $t_{i+1}-t_1=o(t-t_0)$, $t_{i+1}-t_1<\log n$ and $(i-1)\ell=o(t-t_0)$, $(i-1)\ell<\log n$. Hence, we have
$$
(t_{i+1}-t_i)+(i-1)\ell<\min\big\{(1-c-o(1))(t-t_0)), \log n\big\}.
$$
This, together with \eqref{eq:maxload-h0}, implies that
\begin{align}\label{eq:maxload-h0-a}
\P\left(\maxload^f_{H_0}(t_in)>0\right) &\leq \P\left(\sum_{j=1}^i\max_{p\in H_0}L^f_{2, p}((t_{j-1},t_{j}])>t_{i+1}-t_1\right)\notag\\
&\leq \sum_{j=1}^i\P\left(\max_{p\in H_0}L^f_{2, p}((t_{j-1},t_{j}])>t_{j+1}-t_j\right)\notag\\
&\leq i\cdot 2n^{-e^{2\sqrt{\log\log\log n}}},
\end{align}
where the last inequality follows from Lemma \ref{lem:retry maxload bound}. Combine inequalities \eqref{eq:maxload-non-heavily-loaded-set}, \eqref{eq:maxload-heavily-loaded-set}, \eqref{eq:maxload-h0-a} to obtain
$$
\P\left(\maxload^f(tn)>L_0+2k\ell\right)\leq k\cdot 2n^{-e^{2\sqrt{\log\log\log n}}}<n^{-e^{\sqrt{\log\log\log n}}}.
$$
This concludes the proof.
\end{proof}


\subsection{Case : \texorpdfstring{$t\ge \omega(\log n)$}{t≥ω(log n)}}\label{sec:case 4}

\begin{prop}\label{prop:single time upper bound case 4}
Denote by $f$ the $\big(\frac{1}{5},t-\tfrac7{\theta}\log n,\tfrac7{\theta}\log n,\ell,\ell\big)$-drift-threshold strategy with
$k$ and $\ell$ as in Proposition \ref{prop:single time upper bound case 3}. Then, for $n$ large enough, $f$ has 
$$
\P\left(\maxload^f(tn)>(2k+1)\ell\right)<n^{-1/7}.
$$ 
\end{prop}

\begin{proof}
We employ the aforementioned concatenated strategy described in Section~\ref{sec:combin}. 
Inequality \eqref{eq:maxload-drift} in Lemma \ref{lem:single time upper bound} yield
$$
\P\left(\maxload^f(t_0n)>\frac{6}{\theta}\log n\right)<n^{-1/6}.
$$
Employing once again the notation
$$
H_0=\left\{i\in[n]: L_i^f(t_0n)>\ell\right\},
$$
we apply Lemma \ref{lem:level set bound} to obtain
$$
\P\left(|H_0|> \frac{160}{\theta^2}ne^{-\frac{\theta\ell }{3}}\right)\leq 2\exp\left(-2n\left(\frac{80}{\theta^2}\right)^2e^{-\frac{2\theta \ell}{3}}\right)=\exp\left(-n^{1-o(1)}\right).
$$
The inequalities above imply that, with probability at least $1-\Theta(n^{-1/6})$, the conditions in Proposition \ref{prop:single time upper bound case 3} hold with $\eta=0$ (observe that $\alpha$ there, satisfies $\alpha=1+\tfrac{\log 7-\log\theta}{\log\log n}$.
Hence, with high probability, we can apply the multi-stage $(t,\ell,\ell)$-threshold strategy in Section \ref{sec:case 2} from time $t_0$ to time $t$.  Then we can apply Proposition \ref{prop:single time upper bound case 3} to conclude the proof.
\end{proof}


\section{Single-time load discrepancy: lower bound}\label{sec:single-time load lower bound}

In this section, we show that no two-thinning strategy can achieve a maximum load better than that in Theorem~\ref{thm:single-time load discrepancy}. Due to inequality \eqref{eq:integerization}, we can again assume that $m=tn$ for $t\in \N$. The lower bound in Theorem~\ref{thm:single-time load discrepancy} is an immediate consequence of the following statement applied on the process starting from time $\max\big\{\big\lfloor t-\tfrac{ \sqrt{\log n}}{50}\big\rfloor,0\big\}$.

\begin{prop}\label{prop:single-time lower bound case 1}
Given $t\le  \frac{\sqrt{\log n}}{50}$, we set 
$\ell=\sqrt{\frac{\log n}{12(\log\log n-2\log t)}}.$ Then any two-thinning strategy $f$ with any initial load vector $\{L_i(0)\}_{i\in [n]}\in\Z^n$ satisfies
\begin{equation}\label{eq:single-time lower bound case 1}
\P\left(\maxload^f(tn)<\ell\right)<3e^{-\sqrt n}.
\end{equation}
\end{prop}

\begin{proof}
If $\maxload^f(0)\ge t+\ell$, we will have $\maxload^f(t n)\ge \ell$ and inequality \eqref{eq:single-time lower bound case 1} trivially holds. Hence, we will assume that $\maxload^f(0)< t+\ell$. We denote $S=\big\{i\in [n]: L_i(0)\geq 0\big\}$ and $S^c=[n]\setminus S$. We first show that
\begin{equation}\label{eq:average level bound}
|S|\geq \frac{n}{t+\ell+1}.
\end{equation}
To see this, observe that
$$
0=\sum_{i\in[n]}L_i(0)=\sum_{i\in S}L_i(0)+\sum_{i\in S^c}L_i(0).
$$ 
This, together with our assumptions that $\{L_i(0)\}_{i\in [n]}\in\Z^n$ and $\maxload^f(0)< t+\ell$, yields
$$
|S^c|\leq\sum_{i\in S^c}|L_i(0)|=\sum_{i\in S}L_i(0)\leq |S|\cdot (t+\ell) .
$$
Then inequality \eqref{eq:average level bound} readily follows from the inequality above and $|S^c|=n-|S|$.

Next, we set $r^*=\lfloor |S|e^{-2 t I(\ell/ t)}/2\rfloor$, where $I(x)$ is given in Lemma \ref{lem:poss tail}. We denote by $r$ the number of retries up to time $t n$. By the law of total probability, we have
\begin{align}
\P\left(\maxload^f(tn)<\ell\right) &\le \P\left(\maxload^f_S(tn)<\ell\right) \notag \\
&=\P\left(\maxload^f_S(tn)<\ell, r<r^*\right) \label{eq:lower bound single time maxload a}\\
&~~~+\P\left(\maxload^f_S(tn)<\ell, r\geq r^*\right)
\label{eq:lower bound single time maxload b}.
\end{align}
We first estimate the probability in \eqref{eq:lower bound single time maxload a}. Recall that $\psi_S^{t+\ell}(tn)$ defined in \eqref{eq:primay allocation level set} represents the number of bins in $S$ that are suggested as primary allocations at least $t+\ell$ times up to time $tn$. Observe that if we retry fewer than $\psi_S^{t+\ell}(tn)$ balls, the maximum load  will be at least $\ell$. Hence, we have
\begin{equation}\label{eq:lower bound single time maxload a-1}
\P\left(\maxload_S^f(tn)<\ell, r<r^*\right)\leq \P\left(\psi_S^{t+\ell}(tn)<r^*\right).
\end{equation}
We denote by $\{X_i\}_{i\in [n]}$ independent $\poss(t)$ random variables. Let $W_i$ be the indicator function of the event $\{X_i\geq t+\ell\}$. Hence, $\{W_i\}_{i\in [n]}$ are independent Bernoulli random variables such that
\begin{equation}\label{eq:single-step for hoeffding}
p:=\P(W_i=1)=\P(X_i\geq t+\ell)\geq e^{-2tI(\ell/t)}\geq\frac{2r^*}{|S|}, \end{equation}
where the first inequality follows from Lemma \ref{lem:poss tail}. We then apply Lemma~\ref{lem:poisson approximation}, inequality  \eqref{eq:single-step for hoeffding} and Hoeffding's inequality to obtain

\begin{align}\label{eq:lower bound single time maxload a-2}
\P\left(\psi^{t+\ell}_S(tn)<r^*\right) &\leq2\P\left(\sum_{i\in S}W_i<r^*\right)\leq  2\P\left(\sum_{i\in S}W_i<\frac{|S|p}{2}\right)\notag
\\
&\leq 2\exp\left(-\frac{|S| p^2}{2}\right)\leq 2\exp\left(-\frac{|S|}{2}e^{-4tI(\ell/t)}\right)\notag\\
&\leq 2\exp\left(-\frac{n}{2(t+\ell+1)}\left(\frac{et}{\ell}\right)^{12\ell}\right)\notag\\
&=\exp\left(-n^{1-o(1)}\right),
\end{align}
where the penultimate transition follows from the upper bound of $I(x)$ in \eqref{eq:approx-I-1} and  the fact that $\ell>4t$ for $t\leq \frac{\sqrt{\log n}}{50}$.

Next we estimate the probability in \eqref{eq:lower bound single time maxload b}. Recall that $L_{2, i}([tn])$ defined in \eqref{eq:primary and secondary load} represents the number of balls that bin $i$ receives from secondary allocations. Then we have
\begin{equation}\label{eq:lower bound single time maxload b-1}
\P\left(\maxload_S^f(tn)<\ell, r\geq r^*\right)\leq \P\left(\max_{i\in S}L^f_{2, i}([tn])< t+\ell, r\geq r^*\right).
\end{equation}
Apply Lemma \ref{lem:max load} to obtain
\begin{equation}\label{eq:retry maxload}
\P\left(\max_{i\in S}L^f_{2, i}([tn])< t+\ell, r\geq r^*\right) \leq 2\exp\left(-\frac{|S|(r^*/n)^{t+\ell}}{e(t+\ell)!}\right).
\end{equation}
Using the upper bound of $I(x)$ in \eqref{eq:approx-I-1} and the fact that $\ell>4t$ for $t\leq \frac{\sqrt{\log n}}{50}$, we obtain $r^*>\frac{|S|}{2}\left(\frac{et}{\ell}\right)^{6\ell}$. This, together with
Stirling's approximation $k!\leq e\sqrt k (k/e)^k$, yields that for $n$ large enough
\begin{equation*}\label{eq:retry maxload-1}
\frac{(r^*/n)^{t+\ell}}{e(t+\ell)!}\geq \frac{1}{e^2\sqrt{t+\ell}}\left(\frac{e}{2(t+\ell)(t+\ell+1)}\right)^{t+\ell}\left(\frac{et}{\ell}\right)^{6\ell(t+\ell)}>\left(\frac{t}{\ell}\right)^{12 \ell^2}.
\end{equation*}
This, together with \eqref{eq:lower bound single time maxload b-1} and \eqref{eq:retry maxload}, yields
\begin{equation}\label{eq:lower bound single time maxload b-2}
\P\left(\maxload^f(tn)<\ell, r\geq r^*\right)\leq 2\exp\left(-\frac{n}{(t+\ell+1)}\left(\frac{t}{\ell}\right)^{12 \ell^2}\right)\leq 2e^{-\sqrt n},
\end{equation}
where the second inequality follows from the fact that 
\begin{equation*}\label{eq:retry maxload-2}
12 \ell^2\log\frac{\ell}{t} =\frac{\log n}{2}\left(1-\frac{\log (\log\log n-2\log t)+\log 12}{\log\log n-2\log t}\right)<\frac{\log n}{2}.
\end{equation*}
Then we can obtain \eqref{eq:single-time lower bound case 1} by combining \eqref{eq:lower bound single time maxload a}, \eqref{eq:lower bound single time maxload b}, \eqref{eq:lower bound single time maxload a-1}, \eqref{eq:lower bound single time maxload a-2} and  \eqref{eq:lower bound single time maxload b-2}.
\end{proof}
\section{All-time load discrepancy: upper bound}\label{sec:all time upper bound}

In the previous sections, we studied different thinning strategies which yield a good control of $\maxload^f(m)$, the maximum load at the end of the process. Here we are interested in thinning strategies that can control  $\maxload^f([m])$, the maximum load throughout the entire process.

As before, we assume that $m=tn$ for $t\in \N$. Clearly, $\maxload^f([m])\ge\maxload^f(m)$ and that $\maxload^f([m])$ is monotone non-decreasing function of $m$. On the other hand, we also have $ \maxload^f([m])\leq \maxload^f(m)+t$, where the RHS is the maximum number of balls in a single bin at the end of the process. Hence, for $t=O(\sqrt {\log n})$, we can apply the $(t+\ell)$-threshold strategy as per the analysis in Section \ref{sec:case 1} and obtain an optimal all-time maximum load (up to some multiplicative constants).  In the following couple of sections, we prove the upper bound in Theorem~\ref{thm:all-time load discrepancy} for $t=\omega(\sqrt {\log n})$.

\subsection{Case: 
\texorpdfstring{$\omega(\sqrt {\log n})\le t\le O(\log ^2 n/(\log\log n)^3)$}{ω(√(log n))≤t ≤ O((log² n)/(log log n)³)}}
\label{sec: all time upper bound case 1}

\begin{prop}\label{prop:all time upper bound case 1}
Suppose that $\omega(\sqrt {\log n})\leq t\leq \frac{\log ^2n}{(24\log\log n)^3}$. Set $\ell=(t\log n)^{1/3}$. We also assume that for all $i\in[n]$ the initial load satisfies $L_i(0)\leq L_0$ for some $L_0>0$. Then for any $c>0$ and sufficiently large $n$, the $\ell$-relative threshold strategy $f$ satisfies
\begin{equation}\label{eq:all time upper bound case 1}
\P\left( \maxload^f([tn])>L_0+(12c+9)\ell\right)\leq n^{-c}.
\end{equation}
\end{prop}

\begin{proof}
Observe that for any $s\in [t]$ and any $(s-1)n<k\leq sn$, 
$$
\maxload^f((s-1)n)-1\leq \maxload^f(k)\leq \maxload^f(sn)+1.
$$
Hence, it suffice to show  that
\begin{equation}\label{eq:all time upper bound case 1-a}
\P\left(\max_{s\in [t]}\maxload ^f(sn)>L_0+(12c+8)\ell\right)<n^{-c}.
\end{equation}
For $s\in [t]$, we denote by $r_s$ the number of retries in the $s$-th stage, i.e., in the time interval $(n(s-1),ns]$. 
On the one hand, if a bin $i\in [n]$ accepts more than $s+\ell$ primary allocations in the first $s$ stages, the load of this bin before accepting the last primary allocation has to be at most $-\log n$. For such a bin $i\in [n]$, we have
$$
L_i^f(sn)\leq L^f_{2, i}\left([sn]\right)-\log n+1,
$$
where the function $L^f_{2, i}$ given in \eqref{eq:primary and secondary load} is the number of balls bin $i$ receive from secondary allocations. On the other hand, if a bin $i$ accepts at most $s+\ell$ primary allocations in the first $s$ stages, we have
$$
L_i^f(sn)\leq L_i(0)+\ell+L^f_{2, i}\left([sn]\right)\leq L^f_{2, i}\left([sn]\right)+L_0+\ell.
$$
Write $E_s=\{r_k\leq r_k^*~\text{for all}~ 1\leq k\leq s\}$, where $r_k^*=6ne^{-k I(\ell/k)}/\log(1+\ell/k)$ and  $I(x)$ is given in Lemma \ref{lem:poss tail}. The inequalities above and the law of total probability imply that
\begin{align}\label{eq:all time upper bound case 1-b}
\P\left(\maxload ^f(sn)>L_0+(12c+8)\ell\right) &\leq \P\left(\max_{i\in [n]}L^f_{2, i}\left([sn]\right)>(12c+7)\ell\right) \notag\\
&\leq \P\left(\max_{i\in [n]}L^f_{2, i}\left([sn]\right)>(12c+7)\ell, ~E_s\right)+\P\big(E_s^c\big).
\end{align}

We first estimate $\P\big(E_s^c\big)$.  The definition of our $\ell$-relative threshold strategy given in Section \ref{sec:base strat} guarantees that if a retry occurs in the $k$-th stage, then it is necessary that the suggested bin has accepted at least $k-1+\ell$ primary allocations.
Hence, for a single bin, the number of retries in the $k$-th stage is either 0 or the difference between the number of times this bin was suggested as a primary allocation up to stage $k$ and $k-1+\ell$ provided that the difference is positive.
We write $\{X_i^k\}_{i\in[n]}$ for independent $\poss(k)$ random variables. Define $Y_i^k=\max\big\{0, X_i^k-k-\ell+1\big\}$ and $Y^k=\sum_{i=1}^nY_i^k$.  Lemmata \ref{lem:poisson approximation} \& \ref{lem:retry bound} yield 
$$
\P(r_k>r_k^*) \leq 2\P\big(Y^k>r_k^*\big)\leq 2\exp\left(-ne^{-k I(\ell/k)}\right).
$$
One can check that $I(x)/x$ is an increasing function. Then it is not hard to see that for any fixed $\ell>0$, the function $e^{-k I(\ell/k)}$ is increasing with respect to $k$. Hence, for all $k\in [t]$, we have
$$
\P(r_k>r_k^*)\leq2\exp\left(-ne^{-I(\ell)}\right)\leq2\exp\left(-n\left(\frac{e}{\ell}\right)^{3\ell}\right),
$$
where the last inequality follows from the upper bound of $I(x)$ in \eqref{eq:approx-I-1}. Our assumption of $t$ and the choice of $\ell$ yield $\ell\leq \frac{\log n}{24\log\log n}$ and hence
$$
n\left(\frac{e}{\ell}\right)^{3\ell}=\exp\left(\log n-3\ell\log\frac{\ell}{e}\right)>\sqrt n.
$$
Take the union bound to obtain (for $n$ large enough),
\begin{equation}\label{eq:all time upper bound case 1-c}
\P\big(E_s^c\big)\leq \sum_{k=1}^s\P(r_k>r_k^*)\leq 2se^{-\sqrt n}=e^{-(1-o(1))\sqrt n}.
\end{equation}

Now, we estimate the first term of \eqref{eq:all time upper bound case 1-b}. Recall that $r_k^*=6ne^{-k I(\ell/k)}/\log(1+\ell/k)$. We again use the fact that $I(x)/x$ is increasing to deduce that $r_k^*$ is an increasing function. Hence, when $E_s$ occurs, the total number of retries is no more than $tr_t^*$. We denote by  $\{Z_i\}_{i\in[n]}$ independent $\poss(\lambda)$ random variables, where
\begin{equation}\label{eq:lambda-upper}
\lambda=\frac{tr_t^*}{n}= \frac{6te^{-tI(\ell/t)}}{\log(1+\ell/t)}<\frac{12t^2}{\ell}\exp\left(-\frac{\ell^2}{4t}\right),
\end{equation}
where the inequality follows from the lower bound of $I(x)$ in \eqref{eq:approx-I} and $\log(1+x)\geq x/2$ for $0<x<1$ and the fact that $\ell\leq t$. Using Lemma \ref{lem:poisson approximation}, we obtain
\begin{equation}\label{eq:all time upper bound case 1-d}
\P\left(\max_{i\in [n]}L^f_{2, i}\left([sn]\right)>(12c+7)\ell, ~E_s\right)  \leq 2\P\left(\max_{i\in[n]}Z_i>(12c+7)\ell\right).
\end{equation}
Apply Lemma \ref{lem:poss tail} to obtain
\begin{align*}
\P(Z_1>(12c+7)\ell) &\leq\P(Z_1>\lambda+(12c+6)\ell) \leq e^{-\lambda I((12c+6)\ell/\lambda)}<\left(\frac{e\lambda}{(12c+6)\ell}\right)^{(12c+6)\ell}\notag\\
&<\exp\left(-\frac{(12c+6)\ell^3}{4t}+(12c+6)\ell\log\frac{2et^2}{(2c+1)\ell^2}\right),
\end{align*}
where the first inequality follows from that $\lambda<\ell$, the third inequality follows from the lower bound of $I(x)$ in \eqref{eq:approx-I-1}, and in the last inequality we use the upper bound on $\lambda$ in \eqref{eq:lambda-upper}.  Our choice of $\ell$ and the assumption on $t$ guarantees that $\ell^2>12t\log t$, which yields 
$$
\ell\log\frac{2et^2}{(2c+1)\ell^2}<\ell\log t<\frac{\ell^3}{12t}.
$$
Combine the two inequalities  above to obtain
$$
\P(Z_1>(12c+7)\ell)\leq\exp\left(-\frac{(2c+1)\ell^3}{t}\right)=n^{-(2c+1)}.
$$
This, together with \eqref{eq:all time upper bound case 1-d}, yields that
\begin{equation*}\label{eq:all time upper bound case 1-e}
\P\left(\max_{i\in [n]}L^f_{2, i}\left([sn]\right)>(12c+7)\ell, ~E_s\right) \leq 2n\P(Z_1>(12c+7)\ell) \leq 2n^{-2c}.
\end{equation*}
Combining the inequality above  with \eqref{eq:all time upper bound case 1-b}, \eqref{eq:all time upper bound case 1-c}, we obtain  that for any $s\in [t]$ and $n$ large enough, 
$$
\P\left(\maxload ^f(sn)>L_0+(12c+8)\ell\right)\leq e^{-(1-o(1))\sqrt n}+ 2n^{-2c}\leq 3n^{-2c}.
$$
Taking a union bound, we can obtain for $n$ large enough,
$$
\P\left(\max_{s\in [t]}\maxload ^f(sn)>L_0+(12c+8) \ell\right)\leq 3tn^{-2c}\leq n^{-c}.
$$
This proves \eqref{eq:all time upper bound case 1-a}, and hence \eqref{eq:all time upper bound case 1}.
\end{proof}


\subsection{Case: 
\texorpdfstring{$\omega(\log ^2 n/(\log\log n)^3))  \le t\le n^{O(1)}$}{ω((log² n)/(log log n)³)≤t ≤ n\^ O(1)}
}\label{subsection:case 2}

In this case, we utilize the varying drift strategy to control the all-time maximum load. We set $Z_k=i$ if the $k$-th point of $X(t)$ is a point of the process $X_i(t)$ define in Section~\ref{sec:base strat}. We will show that, with high probability, the random process $\{Z_k\}_{k\in\N}$ can be realized by some two-thinning strategy $f$ and that it achieves the desired bound.

\begin{prop}\label{prop:interval upper bound case 2}
Let $m,n\in \N$ sufficiently large and denote $d=\tfrac{\log m}{\log n}$. Let $\ell=\frac{2\log n}{\log\log n}$. The $\ell$-varying drift strategy $f$ defined above satisfies 
\begin{equation}\label{eq:interval upper bound case 2}
\P\left( \maxload^f([m])>(d+4)\ell\right)\leq \frac{2\log^3 n}{n}.
\end{equation}
\end{prop}

Next, we provide an estimate of the probability that the realizability criterion \eqref{eq:realizibility criterion} holds for a period of time, which implies that, with high probability, the process $\{Z_k\}_{k\in\N}$ can be realized by some two-thinning strategy $f$ for quasi-exponential time. 

\begin{lem}\label{lem:realizibility prob}
For any $T>0$ and sufficiently large $n$, we have
\begin{equation}\label{eq:realizibility prob}
\P\left(\exists t\in [0, T]\ :\ \Big|\Big\{i\in [n]: X_i(t)-t>\ell\Big\}\Big|>\frac{n}{\sqrt{\log n}} \right)\leq T\exp\left(-\frac{n}{2\log n}\right).
\end{equation}
\end{lem}

\begin{proof}
We first estimate the probability $\P(\sup_{s\in [t, t+1]}(X_i(s)-s)>\ell)$ for all $0\leq t\leq T-1$. We denote by $E=\{X_i(t)\leq t+\ell/2\}$. By the law of total probability, 
\begin{align}\label{eq:unit interval prob}
\P\left(\sup_{s\in [t, t+1]}(X_i(s)-s)>\ell\right) &\le\P\left(\sup_{s\in [t, t+1]}(X_i(s)-s)>\ell, ~E\right)+\P(E^c).
\end{align} 
Since $X_i(t)$, given in \eqref{eq:intensity xi a}, is $\frac{1}{\sqrt{\log n}}$-standardizing, we can apply inequality \eqref{eq:con-theta-b} in Corollary \ref{coro:con-theta-sum}
and Markov's inequality to obtain
\begin{equation}\label{eq:unit interval prob a}
\P(E^c)\leq\P\left(|X_i(t)-t|\geq \frac{\ell}{2}\right)\leq 80\log n\cdot \exp\left(-\frac{\sqrt{\log n}}{2 \log\log n}\right)<\frac{1}{4\sqrt{\log n}}.
\end{equation}
Next we bound $\P\left(\sup_{s\in [t, t+1]}(X_i(s)-s)>\ell, ~E\right)$.
Let $Y$ be a $\poss(1+\theta_1)$ variable. Observe that, by \eqref{eq:intensity xi a}, $X_i(t+1)-X_i(t)$ is stochastically dominated by $Y$. Hence, we have
\begin{align}\label{eq:unit interval prob b}
\P\left(\sup_{s\in [t, t+1]}(X_i(s)-s)>\ell, ~E\right) &\leq \P\left(Y>\frac{\ell}{2}\right) 
\leq \exp\left(-(1+\theta_1)I\left(\frac{\ell}{2(1+\theta_1)}\right)\right) \notag\\
&\leq \exp\left(-\frac{\ell}{3}\log\frac{\ell}{3e}\right)=n^{-\frac{2}{3}+o(1)},
\end{align}
where in the second inequality, the function $I$, appearing in Lemma \ref{lem:poss tail}, is the rate function of the deviation bound of Poisson random variables, and the last inequality follows from the fact that $I(x)>x\log (x/e)$ for $x>4$. Combine \eqref{eq:unit interval prob}, \eqref{eq:unit interval prob a} and  \eqref{eq:unit interval prob b} to obtain
$$
\P\left(\sup_{s\in [t, t+1]}(X_i(s)-s)>\ell\right)\leq \frac{1}{4\sqrt{\log n}}+n^{-\frac{2}{3}+o(1)}\leq \frac{1}{2\sqrt{\log n}}.
$$
We denote by $S(t)=\{i\in [n]: \sup_{s\in [t, t+1]}(X_i(s)-s)>\ell\}$. Let $W_i$ be the indicator function of the event $\{\sup_{s\in [t, t+1]}(X_i(s)-s)>\ell\}$. Hence, $\{W_i\}_{i\in[n]}$ are independent Bernoulli random variables such that
$$
\P(W_i=1)=\P\left(\sup_{s\in [t, t+1]}(X_i(s)-s)>\ell\right)\leq \frac{1}{2\sqrt{\log n}}.
$$
By Hoeffding's inequality, 
$$
\P\left(|S(t)|> \frac{n}{\sqrt{\log n}}\right)=\P\left(\sum_{i=1}^nW_i \geq\frac{n}{\sqrt{\log n}}\right)\leq \exp\left(-\frac{n}{2\log n}\right).
$$
The desired statement \eqref{eq:realizibility prob} follows by taking a union bound. 
\end{proof}


We are now ready to establish Proposition \ref{prop:interval upper bound case 2}.

\begin{proof}[Proof of Proposition \ref{prop:interval upper bound case 2}]
Set $T=m/n+\Delta$, where $\Delta=1+2\sqrt{\log n} \log(80\log n)$. Let $E$ be the event that $\{Z_i\}_{i\in \N}$ can be realized by some two-thinning strategy $f$. Lemma \ref{lem:realizibility prob} yields
\begin{equation}\label{eq:realizibility prob a}
\P(E^c)\leq T\exp\left(-\frac{n}{2\log n}\right).
\end{equation}
For each fixed $1\leq k\leq m$, we set $t^*=k/n+\Delta$. We write $F=\{X(t^*)\geq k\}$. The law of total probability yields
\begin{equation}\label{eq:single k large deviation}
\P\left( L_i^f(k)>(d+4)\ell\right)\leq \P\left( L_i^f(k)>(d+4)\ell, ~E\cap F\right) +\P(E^c)+\P(F^c).
\end{equation}
Since $X_i(t)$ given in \eqref{eq:intensity xi a} is $\frac{1}{\sqrt{\log n}}$-standarizing, we can apply the second inequality of \eqref{eq:con-theta-sum} in Corollary \ref{coro:con-theta-sum} and Markov's inequality to obtain
\begin{align}\label{eq:single k large deviation a}
\P(F^c) &= \P\left(X(t^*)<nt^*-n\Delta\right)\leq (80\log n)^n \exp\left(-\frac{n\Delta}{2\sqrt{\log n}}\right)=\exp\left(-\frac{n}{2\sqrt{\log n}}\right),
\end{align}
where the last equality follows from our choice of $\Delta$. The definition of $X_i(t)$ in  \eqref{eq:intensity xi a} implies that $X_i(t)-\ell$  is upper $(1-\frac{12}{\sqrt{\log n}})$-standardizing. One can check that the condition of inequality \eqref{eq:concentration-theta-one-sided-a} in Corollary \ref{coro:concentration-theta-standarizing-a} holds for $2\theta=1-\frac{12}{\sqrt{\log n}}$, $\lambda=\frac{\log\log n}{2}$. Hence, we apply inequality \eqref{eq:concentration-theta-one-sided-a} to obtain
\begin{equation}\label{eq:concentration-theta-one-sided-b}
\E\exp\left(\frac{\log\log n}{2}(X_i(t)-t-\ell)\right)<1+\frac{2e^{2\lambda}}{1-e^{-\lambda/2}}< \log^3 n.
\end{equation}
Whenever the event $E\cap F$ occurs, we have $L_i^f(k)+k/n\leq X_i(t^*)$. Inequality \eqref{eq:concentration-theta-one-sided-b} and Markov's inequality yield
\begin{align*}
\P\left( L_i^f(k)>(d+4)\ell, ~E\cap F\right) &\leq \P\left( X_i(t^*)>\frac{k}{n}+(d+4)\ell\right) \notag\\
&=\P\left( X_i(t^*)-t^*-\ell>(d+3)\ell-\Delta\right)\notag\\
&\leq \P\left(X_i(t^*)-t^*-\ell>(d+2)\ell\right)\notag\\
&\leq (\log n)^3 \cdot n^{-(d+2)}.
\end{align*}
This, together with \eqref{eq:realizibility prob a}, \eqref{eq:single k large deviation}, \eqref{eq:single k large deviation a}, yields that, for sufficiently large $n$,
$$
\P\left( L_i^f(k)>(d+4)\ell\right) \leq 2(\log n)^3 \cdot n^{-(d+2)}.
$$
Taking union bound over $m$ and $n$, we obtain
$$
\P\left( \maxload^f([m])>(d+4)\ell\right) \leq 2(\log n)^3mn^{-(d+1)}.
$$
Then, inequality \eqref{eq:interval upper bound case 2} follows from the fact that $m=n^d$
\end{proof}


\section{All-time load discrepancy: lower bound}\label{sec:all time lower bound}

Here we prove the lower bounds in Theorem~\ref{thm:all-time load discrepancy}.
We again assume that $m$ is divisible by $n$ and write $m=tn$ for some $t\in \Z$. Observe that the lower bound of the single-time maximum load in Theorem~\ref{thm:single-time load discrepancy} implies that of the all-time maximum load up to $t=O(\sqrt{\log n})$. 
Our next result covers the regime of $ \sqrt{\log n}< t< \log^2n/(24\log\log n)^3$.
This, together with the fact that the all-time maximum load is non-decreasing with respect to $t$, implies the lower bound of $\Theta\big(\tfrac{\log n}{\log\log n}\big)$ for $t\geq\log^2n/(24\log\log n)^3$. This completes the proof of the lower bounds in Theorem \ref{thm:all-time load discrepancy}.

\begin{prop}\label{prop:all-time lower bound case 1}
Suppose that $\sqrt{\log n}< t< \frac{\log ^2n}{(24\log\log n)^3}$. Set $\ell=\lfloor(t\log n)^{1/3}\rfloor$. Any two-thinning strategy $f$ satisfies that for $n$ large enough, 
\begin{equation}\label{eq:all-time lower bound case 1}
\P\left(\maxload^f([tn])<\ell\right)\leq \exp\left(-n^{1/5}\right).
\end{equation}
\end{prop}

\begin{proof}
We denote by $r$ the total number of retries and set $r^*=\frac{n}{2}e^{-\ell^2/t}$. Then we have
\begin{align}
\P\left(\maxload^f([tn])<\ell\right) &=\P\left(\maxload^f([tn])<\ell,  r< r^*\right) \label{eq:all-time lower bound case 1-a}\\
&~~~+\P\left(\maxload^f([tn])<\ell,  r\geq r^*\right) \label{eq:all-time lower bound case 1-b}.
\end{align}
We estimate \eqref{eq:all-time lower bound case 1-a}. Recall that $\psi^{t+\ell}(tn)$ defined in \eqref{eq:primay allocation level set} represents the number of bins that are suggested as primary allocations at least $t+\ell$ times after allocating $tn$ balls. If we retry less than $\psi^{t+\ell}(tn)$ balls, then we will have $\maxload ^f(tn)\geq\ell$. Hence we obtain
\begin{align}
\P\left(\maxload^f([tn])<\ell,  r< r^*\right)
&\leq\P\left(\maxload ^f(tn)<\ell,  r< r^*\right)\notag\\
&\leq\P\left(\psi^{t+\ell}(tn)<r^*\right). \label{eq:all-time lower bound case 1-d}
\end{align}
We denote by $\{X_i\}_{i\in[n]}$ independent $\poss(t)$ random variables. Write $Y_i$ for the indicator function of the event $\{X_i> t+\ell\}$. Hence, $\{Y_i\}_{i\in[n]}$ are independent $\bernoulli(p)$ random variables with
\begin{equation*}\label{eq:bound p 1}
p=\P(X_1>t+\ell)\geq e^{-2tI(\ell/t)}\geq e^{-\ell^2/t},
\end{equation*}
where the first inequality follows from Lemma \ref{lem:poss tail} and the second inequality uses the upper bound of $I(x)$ in \eqref{eq:approx-I} and the fact that $\ell<t$. 
Apply Lemma \ref{lem:poisson approximation} and Hoeffding's inequality to obtain 
\begin{align}\label{eq:all-time lower bound case 1-e}
\P\left(\psi^{t+\ell}(tn)<r^*\right) &\leq\P\left(\psi^{t+\ell}(tn)<\frac{pn}{2}\right)\leq 2\P\left(\sum_{i=1}^nY_i<\frac{pn}{2}\right)\notag\\
&\leq 2\exp\left(-\frac{p^2n}{2}\right)<\exp\left(-\frac{n}{2}e^{-2\ell^2/t}\right)\notag\\
&=\exp\left(-n^{1-o(1)}\right).
\end{align}

Next we estimate \eqref{eq:all-time lower bound case 1-b}. Recall that $R_k$ given in \eqref{eq:def-retry} is the number of retries after allocating $k$ balls. Define
$s_0=\inf\big\{s\in[t] : R_{s n}-R_{(s-1) n}\ge r^*/t\big\}$. Whenever the event $\{r\geq r^*\}$ occurs,  we have $s_0<\infty$. Write $S=\big\{i\in [n] : L_i^f((s_0-1)n)\geq 0\big\}$. As per \eqref{eq:average level bound},
we show that whenever the event $\big\{\maxload^f((s_0-1)n)< \ell\big\}$ occurs, we have
\begin{equation}\label{eq:average level bound II} 
|S|\geq  \frac{n}{\ell+1}.
\end{equation}
To see this, observe that
$$
0=\sum_{i\in[n]}L_i^f((s_0-1)n)=\sum_{i\in S}L_i^f((s_0-1)n)+\sum_{i\in S^c}L_i^f((s_0-1)n).
$$ 
This, together the fact that $\big\{L_i^f((s_0-1)n)\big\}_{i\in [n]}\in\Z^n$ and $\maxload^f((s_0-1)n)< \ell$, yields
$$
|S^c|\leq\sum_{i\in S^c}|L_i^f((s_0-1)n)|=\sum_{i\in S}L_i^f((s_0-1)n)\leq |S|\cdot (\ell+1).
$$
Then we can obtain \eqref{eq:average level bound II} using $|S^c|=n-|S|$. 

Apply Lemma \ref{lem:max load} to obtain
\begin{align}\label{eq:all-time lower bound case 1-f}
\P\left(\maxload^f([tn])<\ell,  r\geq r^*\right)&\leq \P\big(s_0<\infty,~\maxload^f(s_0n)<\ell\big)\notag\\
& \leq \P\left(s_0<\infty,~ \max_{i\in S}L^f_{2, i}\big(((s_0-1)n,s_0 n]\big)<\ell\right)\notag\\
& \leq 2\exp\left(-\frac{|S|}{e\ell!}\left(\frac{r^*}{tn}\right)^\ell\right).
\end{align}
Recall that $r^*=\frac{n}{2}e^{-\ell^2/t}$, $\ell=\lfloor(t\log n)^{1/3}\rfloor$ and $|S|\geq  n/(\ell+1)$. One can check that
\begin{equation}\label{eq:all-time lower bound case 1-g}
\frac{|S|}{e\ell!}\left(\frac{r^*}{tn}\right)^\ell\geq\frac{\sqrt n}{e(\ell+1)!}\left(\frac{1}{2t}\right)^\ell>\frac{\sqrt n t^{-\ell}}{(\ell+1)^{\ell+3/2}}>\frac{\sqrt n}{t^{3\ell}}>n^{1/4},
\end{equation}
where the second inequality uses Stirling's approximation $(\ell+1)!\leq e\sqrt{\ell+1}(\tfrac{\ell+1}{e})^{\ell+1}$; in the third inequality, we use the fact that $\ell<t$ and the last inequality follows from our choice of $\ell$ and the assumption on $t$. Combine \eqref{eq:all-time lower bound case 1-f} and \eqref{eq:all-time lower bound case 1-g} to obtain
$$
\P\left(\maxload^f([tn])<\ell,  r\geq r^*\right)\leq 2\exp\left(-n^{1/4}\right).
$$
This, together with \eqref{eq:all-time lower bound case 1-b}, \eqref{eq:all-time lower bound case 1-d}, \eqref{eq:all-time lower bound case 1-e}, yields
$$
\P\left(\maxload^f([tn])<\ell\right)\leq \exp\left(-n^{1-o(1)}\right)+2\exp\left(-n^{1/4}\right)<\exp\left(-n^{1/5}\right).
$$
This concludes the proof of \eqref{eq:all-time lower bound case 1}. 
\end{proof}


\section{Typical load discrepancy}\label{sec:typical load discrepancy}

In this section, we investigate two-thinning strategies for controlling the $\varepsilon$-typical maximum load $\maxload_\varepsilon^f([m])$. 
The main technical statement in this section is the following Proposition, which implies Theorem \ref{thm:typical load discrepancy}.

\begin{prop}\label{prop:typical-load-poly-time}
Fix $d\geq 1$. Set $\ell=(\log n)^{\frac{1}{2}+\frac{1}{\sqrt{\log\log\log n}}}$ and $\varepsilon=e^{-\frac{1}{2}\sqrt{\log\log\log n}}$. For sufficiently large $n\in\N$ and $m\le n^{d}$, there exists
a set $S\subset[m]$ with $|S|\ge(1-\varepsilon)m$
such that the $d$-multi-scaled long-term combined strategy $f$ satisfies
\begin{equation}\label{eq:typical-load-poly-time}
\P\left(\maxload^f(S)>\ell\right)\leq \frac{1}{n}.
\end{equation}
For $d\geq 2$ and general values of $m$, the $d$-multi-scaled long-term combined strategy $f$ satisfies
\begin{equation}\label{eq:typical-load-poly-time2}
\P\left(\maxload_\varepsilon^f([m])>\ell\right)\leq \frac{1}{n}.
\end{equation}
\end{prop}

The proof of this result requires the following four propositions, each of which tells us certain property of the process after a phase of an iteration. The proofs of these propositions are given in the following subsections. Throughout this section we use the notations in \eqref{eq:Q-L-A}, \eqref{eq:ms} and \eqref{eq:L0}.

\begin{prop}\label{prop:phase-0 low load}
Fix $d\geq 1$. Let $n\in\N$ be sufficiently large. Suppose that the initial load vector $\{L_i(0)\}_{i\in [n]}$ satisfies that $|\{i\in [n] : L_i(0)>L_0\}|\le 4000ne^{-L_0/15}$ and that $|L_i(0)|\le 100d\log n$ for all $i\in [n]$. Then  the multi-stage ($m_0/n, L_0, L_0$)-threshold strategy $f$ satisfies that
$$
\P\left(L_i^f(m_0)<-300d\log n~\text{or}~L_i^f(m_0)> L~\text{for some}~i\in [n]\right)\le n^{-e^{\sqrt{\log\log\log n}}}.
$$
\end{prop}
\begin{prop}\label{prop:typical-load}
Fix $c>0$. Let $m,n\in \N$ be sufficiently large. We write $\alpha=\frac{\log (m/n)}{\log\log n}$ and assume that 
$\alpha \in\big[\frac{1}{2}+\frac{3}{\sqrt{\log\log\log n}}, 1+\frac{1}{30\sqrt{\log\log\log n}}\big]$. Further, we denote by $\varepsilon=e^{-\frac{2}{3}\sqrt{\log\log\log n}}$ and $\ell=(\log n)^{\frac{1}{2}+\frac{1}{\sqrt{\log\log\log n}}}$. Suppose that 
the initial load vector $\{L_i(0)\}_{i\in [n]}$ satisfies that $L_i(0)\leq (\log n)^{\frac{1}{2}+\frac{1}{2\sqrt{\log\log\log n}}}$ for all $i\in[n]$. Then, there exists $A_m\subset [m]$ with $|A_m|\ge (1-\varepsilon)m$ such that the $0$-multi-scale strategy $f$ satisfies that
\begin{equation}\label{eq:typical-load}
\P\Big(\maxload^f(A_m)>\ell\Big)\leq n^{-c}.
\end{equation}
\end{prop}

\begin{prop}\label{prop:phase-1 minload}
Fix $d\geq 1$. Let $n\in\N$ be sufficiently large. Suppose that the initial load vector $\{L_i(0)\}_{i\in [n]}$ satisfies that $-300d\log n\leq L_i(0)\le Q$ for all $i\in[n]$. 
Then the $Q$-multi-scale strategy $f$ satisfies that
\begin{equation}\label{eq:initial loads bound}
\P\left(\max_{i\in [n]}|L_i^f(m_1)|>A\right)\leq n^{-3d}.
\end{equation}
\end{prop}

\begin{prop}\label{prop:phase-2 typicaly m2}
Fix $d\geq 1$. Let $n\in\N$ be sufficiently large. Suppose that the initial load vector $\{L_i(0)\}_{i\in [n]}$ satisfies that $|L_i(0)|\le A$ for all $i\in[n]$. Then the $1/5$-drift strategy $f$ satisfies that
$$
\P\left(\max_{i\in[n]}|L^f_i(m_2)|> 100d\log n\text{ or }\left|\left\{i\in [n] : L^f_i(m_2)>L_0\right\}\right|> 4000ne^{-L_0/15}\right)\le 2n^{-3d}.
$$
\end{prop}
\begin{proof}[Proof of Proposition~\ref{prop:typical-load-poly-time}]
Observe that Lemma~\ref{lem:expected normalization time} guarantees that the third phase of each iteration eventually terminates so that there are almost surely infinitely many iterations. Set $M_{1,0}= M_{1,1}=0$, $M_{1,2}=m_1$, $ M_{1,3}=m_1+m_{2,1}$. For $j\ge 2$ and $k\in \{0,1,2,3\}$, we define 
$$ M_{j,0}= M_{j-1,3},~~  M_{j,1}=M_{j,0}+m_0,~~ M_{j,2}=M_{j,1}+m_1,~~
M_{j,3}=M_{j,2}+m_{2,j}.
$$ 
Hence, for $k\in\{0, 1, 2\}$, $ M_{j,k}$ is the starting time of the $(k+1)$-th phase in the $j$-th iteration.  For $j\in\N$, we define events
\begin{align*}
    E_j &=\left\{-300d\log n\leq L_i^f\left(M_{j,1}\right)\leq L~\text{for all}~i\in [n]\right\},\\
    F_j &=\left\{\max_{i\in [n]}\big|L_i^f\left(M_{j,2}\right)\big|\leq A\right\},\\
    G_j &=\{m_{2,j}=m_2\}.
\end{align*}
 Our strategy guarantees that the load vector $\big\{L^f_i(M_{j,0})\big\}_{i\in[n]}$ at the beginning of the $j$-th iteration satisfies that
$$\max_{i\in[n]}\big|L_i^f(M_{j,0})\big|\le  100d\log n\ \ \text{ and }\ \ \left|\left\{i\in [n] : L^f_i(M_{j,0})>L_0\right\}\right|\le 4000ne^{-L_0/15}.$$
Hence, we apply Proposition~\ref{prop:phase-0 low load} to obtain for all $j>1$ that
\begin{equation}\label{eq:ejc-1}
\P(E_j^c)\le n^{-e^{\sqrt{\log\log\log n}}}.
\end{equation} 
This inequality trivially holds for $E_1^c$. By Proposition~\ref{prop:typical-load}, we have for all $j\ge1$ that
\begin{equation}\label{eq:typical load given ej}
\P\left(\maxload^f\left(A_{m_1}+M_{j,1}  \right)>\ell\ |\ E_j\right) \le n^{-3d}.
\end{equation}
By Proposition~\ref{prop:phase-1 minload}, we have for all $j\ge1$ that
\begin{equation}\label{eq:fjc given fj}
\P\left(F_j^c\ |\ E_j\right) \le n^{-3d}.
\end{equation}
By Proposition~\ref{prop:phase-2 typicaly m2}, we have for all $j\ge1$ that
\begin{equation}\label{eq:gjc given fj}
 \P\left(G_j^c\ |\ F_j\right) \le 2 n^{-3d}.
\end{equation}
Set $M=m_0+m_1+m_2$. On the event $\bigcap_{j\in[\kappa]} G_j$, we have 
$$ \bigcup_{j=1}^{\kappa} \left(A_{m_1}+M_{j,1} \right)=\bigcup_{j=0}^{\kappa-1}(A_{m_1}+j M).
$$
Set $S=\bigcup_{j=0}^{\kappa-1}(A_{m_1}+j M)$. Putting together \eqref{eq:ejc-1}, \eqref{eq:typical load given ej}, 
\eqref{eq:fjc given fj},\eqref{eq:gjc given fj} and taking the union bound, we now get
$$
\P\left(\maxload^f(S) >\ell\right) \le 5\kappa n^{-3d}.
$$
For $m\leq n^d$, we take $\kappa=\lfloor m/M \rfloor \le n^d$, so that the probability above is less than $1/n$. 

Next, we complete the proof of \eqref{eq:typical-load-poly-time} by showing that $|S|\ge (1-\varepsilon/4)m$.
Notice that $S$ is a disjoint union of copies of $A_{m_1}$ shifted by multiples of $M$. Hence, it suffices to show that $|A_{m_1}|>(1-\varepsilon/4)M$.
Indeed, we have shown in Proposition \ref{prop:typical-load}  that  $|A_{m_1}|>(1-\varepsilon')m_1$ with $\varepsilon'=e^{-\frac{2}{3}\sqrt{\log\log\log n}}$. By the definitions of $m_0, m_1, m_2$ in \eqref{eq:ms}, we have 
\begin{equation}\label{eq:m_1/M}
\frac{m_1}{M}=1-\frac{m_0+m_2}{M}=1-O\Big((\log n)^{\frac{1-\alpha_{i_{\max}+1}}{2}}\Big)>1-O\Big((\log n)^{-\frac{1}{60\sqrt{\log\log\log n}}}\Big),
\end{equation}
where the equality follows from \eqref{eq:last scale}. These, together with $\varepsilon\gg \varepsilon'$, yield that for sufficiently large $n$ we have $|A_{m_1}|>(1-2\varepsilon')M>(1-\varepsilon/4)M$ and hence that $|S|\ge (1-\varepsilon/4)m$.

We now prove \eqref{eq:typical-load-poly-time2}. We say that the $j$-th iteration is \emph{bad} if either $E_j^c$, $F_j^c$, $G_j^c$ happened or $\maxload^f\left(A_{m_1}+M_{j,1}\right)>\ell$; otherwise we say that it is \emph{good}. We denote by $J$ the set of bad iterations among the first $\kappa$ iterations. By definition, each good iteration has length at most  $M=m_0+m_1+m_2$ and the maximum load over $\cup_{j\in [\kappa]\setminus J} \left(M_{j,1} + A_{m_1} \right)$ is bounded above by $\ell$. Hence, we have
\begin{align}\label{eq:bad times}
\left|\left\{m'<m: \maxload^f(m')>\ell\right\}\right| &\leq \sum_{j\in J}(m_0+m_1+m_{2, j})+(\kappa-|J|)(M-|A_{m_1}|) \notag\\
&\leq \sum_{j\in J}(m_0+m_1+m_{2, j})+\frac{\varepsilon m}{2},
\end{align}
where the second inequality follows from $\kappa\le m/M$,  $|A_{m_1}|>(1-\varepsilon/4)M$ and 
\eqref{eq:m_1/M}. 

We now estimate the first term of \eqref{eq:bad times}. As we have just seen, the probability of an iteration being bad is bounded above by $5 n^{-3d}$ and hence $\E|J|\leq 5\kappa n^{-3d}$. Then we apply Markov's inequality to obtain
\begin{equation}\label{eq:bound |J|}
\P\left(|J|>5\kappa n^{-2d}\right)\leq n^{-d}.
\end{equation}
This, together with $m_0<m_1$ and $\kappa<m/m_1$, yields
\begin{equation}\label{eq:bad time a}
\P\left((m_0+m_1)|J|>10mn^{-2d}\right)\leq n^{-d}.
\end{equation}

We now estimate $\sum_{j\in J}m_{2, j}$. Note that the load vector at the beginning of the third phase of each iteration satisfies
$$
\max_{i\in [n]}\big|L_i^f(M_{j, 2})\big|\le 100d \log{n}+m_0+m_1=o(n^2).
$$
We apply Lemma~\ref{lem:expected normalization time} to obtain
$\E(m_{2,j}) \le n^3$ and hence
\begin{equation*}
    \E\left(\sum_{j\in J}m_{2, j}\right)=\E\left(\sum_{j\in [\kappa]} m_{2, j} \ind_{j\in J}\right)\le 5 \kappa n^{3-3d}.
\end{equation*}
Then we apply Markov's inequality to obtain
\begin{equation}\label{eq:bad time b}
\P\left(\sum_{j\in J}m_{2, j}>5 m n^{3-2d}\right) \le \frac{\kappa}{m n^d} \le n^{-d}.
\end{equation}
For $d\geq 2$ and sufficiently large $n$, we combine \eqref{eq:bad times}, \eqref{eq:bad time a} and \eqref{eq:bad time b} to obtain
$$
\P\left(\big|\big\{m'<m: \maxload^f(m')>\ell\big\}\big|>\varepsilon m\right)<\frac{1}{n}. 
$$
This concludes the proof of \eqref{eq:typical-load-poly-time2}.
\end{proof}

\subsection{Proof of Proposition~\ref{prop:typical-load}}
We first make some technical observations on the parameters used in the $Q$-multi-scale strategy given in Section \ref{sec:combin}. Recall that $\alpha_1=\frac{1}{2}+\frac{2}{\lfloor\sqrt{\log\log\log n}\rfloor+1/4}$, $L=(\log n)^{\frac{1+\alpha_1}{3}}$, $k=\big\lfloor\frac{\log\log n}{3\log\log\log n}\big\rfloor$, $N_{i}=\lceil\frac{L}{3k\ell_i}\rceil$ and $Q^{i,j}=(2k+1)(j-1)\ell_i$. We first have for $i\in\N, j\in [N_i]$ that
\begin{equation}\label{eq:Qij-L}
Q^{i, j}<L.
\end{equation}
Observing from \eqref{eq:alpha-i-iteration} that $\{\alpha_i\}_{i\in\N}$ is a non-decreasing sequence, we have for $i\geq 1$ and sufficiently large $n$ that
\begin{align}
(\log n)^{\alpha_i'-\alpha_i} &=(\log n)^{-\frac{1}{5}\cdot\frac{2\alpha_i-1-\varepsilon_i}{2k+1}}=(\log n)^{-\left(\frac{1}{5}-o(1)\right)\frac{\alpha_i-1/2}{k+1/2}} \notag\\
&\le (\log n)^{-\left(\frac{1}{5}-o(1)\right)\frac{\alpha_1-1/2}{k+1/2}} 
\leq (\log n)^{-\frac{2/5-o(1)}{\sqrt{\log\log\log n}}\cdot\frac{1}{k+1/2}}\notag\\
&=(\log n)^{-\left(\frac{6}{5}-o(1)\right)\frac{\sqrt{\log\log\log n}}{\log\log n}}\notag\\
&<e^{-\sqrt{\log\log\log n}}.\label{eq:alpha-i-vs-alpha-i'-a}
\end{align}
Recall that $i_{\max}=\max\{i\in \N: \alpha_i\leq 1\}$. Using \eqref{eq:alpha-i-iteration}, we have for all $i\leq i_{\max}$ that
\begin{align}
N_i&=\Big\lceil\frac{L}{3k\ell_i}\Big\rceil =\frac{(1+o(1))L}{3k\ell_i}=\frac{(1+o(1))}{3k}(\log n)^{\frac{2\alpha_1-1}{6}-\frac{\alpha_i-1/2+k\varepsilon_i}{2k+1}}=(\log n)^{\frac{2\alpha_1-1}{6}-O\left(\frac{1}{k}\right)}.\label{eq:N-i-bd}
\end{align}
Using \eqref{eq:alpha-i-iteration} and \eqref{eq:alpha-i-vs-alpha-i'-a} we observe that $N_i=(1-o(1))(\log n)^{\alpha_{i+1}-\alpha_i}$. This, together with \eqref{eq:N-i-bd}, yields the iteration formula
\begin{align}
\alpha_{i+1}&=\alpha_i+\frac{2\alpha_1-1}{6}-O\left(\frac{1}{k}\right).\label{eq:alpha-i-iteration-a}
\end{align} 
This, along with the definition of $i_{\max}$, implies that
\begin{equation}\label{eq:i-max-bound}
i_{\max}\le(1+o(1)) \frac{6(1-\alpha_1)}{2\alpha_1-1}=\frac{3+o(1)}{2\alpha_1-1}<\sqrt{\log\log\log n}.
\end{equation}
In addition, we have
\begin{align}\label{eq:alpha-i-vs-alpha-i'}
(\log n)^{\alpha_i-\alpha_{i+1}'}=(\log n)^{\alpha_i-\alpha_{i+1}+\frac{1}{5}\cdot\frac{2\alpha_{i+1}-1-\varepsilon_{i+1}}{2k+1}}=(\log n)^{-\frac{2\alpha_1-1}{6}+\Theta\left(\frac{1}{k}\right)}=o(1).
\end{align}

The main technical instrument for establishing Proposition~\ref{prop:typical-load} is the following lemma, the proof of which is provided in the next subsection. 

\begin{lem}\label{lem:low typical load segments}
Consider the $Q$-multi-scale strategy with the initial load vector $\{L_p(0)\}_{p\in [n]}$ satisfying  $L_p(0)\le Q\le L$ for all $p\in[n]$. Fix $c>0$. Set $\ell=(\log n)^{\frac{1}{2}+\frac{1}{\sqrt{\log\log\log n}}}$. For any $s=\sum_{i=1}^{i_{\max}} j_in(\lfloor\log^{\alpha_i}\!n\rfloor+\lfloor\log^{\alpha'_i}\!n\rfloor)$  with $0\leq j_i\leq N_i-1$, we have 
\begin{equation}\label{eq:low typical load segments}
\P\Big(\maxload^f([s,s+n\lfloor\log^{\alpha_1}\! n\rfloor])>\ell\Big)<  n^{-c}.
\end{equation}
\end{lem}

\begin{proof}[Proof of Proposition~\ref{prop:typical-load}]
We will show that the $Q$-multi-scale strategy with $Q\leq L$ satisfies the statement in Proposition~ \ref{prop:typical-load}. Recall that $i_{\max} \!=\max\{i\in\N\!:\! \alpha_i\leq 1\}$. We first show  for sufficiently large $n$ that
\begin{equation}\label{eq:last scale}
\alpha_{i_{\max}+1}\ge  1+\frac{1}{30\sqrt{\log\log\log n}}.
\end{equation}
To this end, we iterate equation \eqref{eq:alpha-i-iteration-a} to obtain
\begin{equation}\label{eq:alpha_plus_one}
\alpha_{i_{\max}+1}=\alpha_1+i_{\max} \left(\frac{2\alpha_1-1}{6}-O\left(\frac{1}{k}\right)\right)=\frac{1}{2}+(i_{\max}+3)\cdot \frac{2\alpha_1-1}{6}-O\left(\frac{i_{\max}}{k}\right).
\end{equation}
The monotonicity of $\{\alpha_i\}_{i\in\N}$ and the definition of $i_{\max}$ implies that $\alpha_{i_{\max}+1}>1$. This inequality, equation \eqref{eq:alpha_plus_one} and the fact that $i_{{\max}}$ is an integer yield that
\begin{equation}\label{eq:lower bound i-max}
i_{\max}+3\geq \left\lceil\frac{6}{2\alpha_1-1}\left(\frac{1}{2}+O\left(\frac{i_{\max}}{k}\right)\right)\right\rceil.
\end{equation}
Recall that $\alpha_1=\frac{1}{2}+\frac{2}{\lfloor\sqrt{\log\log\log n}\rfloor +1/4}$, $k=\big\lfloor\frac{\log\log n}{3\log\log\log n}\big\rfloor$ and the bound $i_{\max}<\sqrt{\log\log\log n}$ given in \eqref{eq:i-max-bound}. Then, for sufficiently large $n$, we can further write inequality \eqref{eq:lower bound i-max} as
\begin{align*}
    i_{\max}+3&\ge \left\lceil\frac{3\lfloor \sqrt{\log\log\log n}\rfloor+3/4+o(1)}4\right\rceil\ge \frac{3\lfloor \sqrt{\log\log\log n}\rfloor +1}{4}.
\end{align*}
Plugging this into \eqref{eq:alpha_plus_one}, we obtain
$$
\alpha_{i_{\max}+1}\geq 1+\frac{1}{24\lfloor \sqrt{\log\log\log n}\rfloor+6} -O\left(\frac{i_{\max}}k\right)\ge 1+\frac{1}{30\sqrt{\log\log\log n}}.
$$
This proves \eqref{eq:last scale}. 

We next prove the main statement \eqref{eq:typical-load}. For $m\in\N$, we define the set 
\begin{equation*}\label{eq:good segments}
A_m=\cup_{s\in J_m}\big\{m'\in \N: s\leq m'\leq \max\big\{m, s+n\lfloor\log^{\alpha_1}\!n\rfloor\big\}\big\},
\end{equation*}
where $J_m$ is defined as
$$
J_m=\left\{s=\sum_{i\in\N}{j_in(\lfloor\log^{\alpha_i}\!n\rfloor+\lfloor\log^{\alpha'_i}\!n\rfloor)}: 0\leq j_i\leq N_i-1,~s<m\right\}.
$$
Observe that by the condition of Proposition~\ref{prop:typical-load}, we have 
$$\alpha=\frac{\log (m/n)}{\log\log n}\le 1+\tfrac{1}{30\sqrt{\log\log\log n}} \le \alpha_{i_{{\max}}+1},$$
and hence
$$ \frac{m}{n}=\log^{\alpha}n\le(\log n)^{\alpha_{i_{{\max}}+1}}=N_{i_{{\max}}}(\lfloor(\log n)^{\alpha_{i_{{\max}}}}\rfloor+\lfloor(\log n)^{\alpha'_{i_{{\max}}}}\rfloor).
$$
Together with \eqref{eq:last scale} we thus have
$$
J_m=\left\{s=\sum_{i=1}^{i_{\max}}j_in(\lfloor\log^{\alpha_i}\!n\rfloor+\lfloor\log^{\alpha'_i}\!n\rfloor): 0\leq j_i\leq N_i-1,~s<m\right\}.
$$
For any fixed constant $c>0$, we apply Lemma~\ref{lem:low typical load segments} and the union bound argument to obtain
\begin{align}\label{eq:low typical load segments-b}
\P\Big(\maxload^f(A_m)\geq\ell\Big) 
&\leq \sum_{s\in J_m}\P\Big(\maxload^f([s,s+n\lfloor\log^{\alpha_1}\! n\rfloor])\geq \ell\Big) \notag\\
&\leq |J_m|\cdot n^{-2c}<n^{-c},
\end{align}
where the last equality follows from that $|J_m|\leq m/n<\log^2 n$. 

We next show that $|A_m|/m\geq1-\varepsilon$.
We define $i^*=\max\{i\in\N: n\log^{\alpha_i}n< m\}$. It is clear from the condition of Proposition~\ref{prop:typical-load} and \eqref{eq:last scale} that $i^*\le i_{\max}$. We further denote 
\begin{align*}
\xi &=\min\big\{\xi'\in\N : \xi'n(\lfloor\log^{\alpha_{i^*}}n\rfloor+\lfloor\log^{\alpha'_{
i^*}}n \rfloor)\ge m\big\},\\ m_{\xi} & =\xi n(\lfloor\log^{\alpha_{i^*}}n\rfloor+\lfloor\log^{\alpha'_{
i^*}}n \rfloor).
\end{align*}
Observe from the definition of $i^*$ that $\xi\le N_{i^*}$.
This, along with \eqref{eq:alpha-i-vs-alpha-i'-a}, 
implies that $m_{\xi}<2m$. Hence it suffices to show that
\begin{equation}\label{eq:AMXI}
\frac{|A_{m_{\xi}}|}{m_{\xi}}\geq 1-\frac{\varepsilon}2.
\end{equation}
For $1\le i\le i^*$, we define
\begin{align*}
J^*_{i}&=\left\{\sum_{i'=i}^{i^*}j_{i'}n(\lfloor\log^{\alpha_{i'}}\!n\rfloor+\lfloor\log^{\alpha'_{i'}}\!n\rfloor): 0\leq j_{i'}\leq N_{i'}-1\right\},\\
B_i&=\bigcup_{j=0}^{N_i-1}\left(jn(\lfloor\log^{\alpha_i}\!n\rfloor+\lfloor\log^{\alpha'_i}\!n\rfloor)+C_i\right), ~\text{where}\\
C_i&=\left(0,n\lfloor\log^{\alpha_i}n\rfloor\right].
\end{align*}
Observe that $B_i+(J^*_{i+1}\cap[m_{\xi}])=C_{i}+(J^*_{i}\cap[m_{\xi}])$ consists of a disjoint union of shifted copies of $B_i$ and that $C_{i+1}+(J^*_{i+1}\cap[m_{\xi}])$ consists of a disjoint union of shifted copies of $C_{i+1}$. We thus obtain
$$
\frac{|C_i+(J^*_{i}\cap[m_{\xi}])|}{|C_{i+1}+(J^*_{i+1}\cap[m_{\xi}])|}=
\frac{|B_i+(J^*_{i+1}\cap[m_{\xi}])|}{|C_{i+1}+(J^*_{i+1}\cap[m_{\xi}])|}
=\frac{|B_i|}{|C_{i+1}|}.
$$
By \eqref{eq:alpha-i-iteration} and \eqref{eq:alpha-i-vs-alpha-i'-a}, we obtain
$$
\frac{|B_i|}{|C_{i+1}|}=\frac{\lfloor\log^{\alpha_i}\! n\rfloor}{\lfloor\log^{\alpha_i}\! n\rfloor+\lfloor\log^{\alpha_i'}\! n\rfloor}
=1-\frac{\lfloor\log^{\alpha_i'}\! n\rfloor}{\lfloor\log^{\alpha_i}\! n\rfloor+\lfloor\log^{\alpha_i'}\! n\rfloor}
>1-(\log n)^{\alpha_i'-\alpha_i}>1-\delta,
$$
where $\delta=e^{-\sqrt{\log\log\log n}}$.
Moreover, we have 
$$
\frac{|C_{i^*}+(J^*_{i^*}\cap [m_{\xi}])|}{m_{\xi}}=\frac{\lfloor \log^{\alpha_{i^*}}n \rfloor}{
\lfloor \log^{\alpha_{i^*}}n\rfloor + \lfloor\log^{\alpha_{i^*}'}n\rfloor}>1-\delta.
$$
Iterating these observations we obtain
\begin{align*}
\frac{|A_{m_\xi}|}{m_{\xi}}
&=\frac{|C_1+(J^*_{1}\cap[m_{\xi}])|}{m_{\xi}}=\frac{|C_{i^*}+(J^*_{i^*}\cap[m_{\xi}])|}{m_{\xi}}\cdot
\frac{|C_1+(J^*_{1}\cap[m_{\xi}])|}{|C_{i^*}+(J^*_{i^*}\cap[m_{\xi}])|}\\
&=\frac{|C_{i^*}+(J^*_{i^*}\cap[m_{\xi}])|}{m_{\xi}}\cdot
\prod_{i=1}^{i^*-1}\frac{|C_i+(J^*_{i}\cap[m_{\xi}])|}{|C_{i+1}+(J^*_{i+1}\cap[m_{\xi}])|}\\
&\ge 
(1-\delta)^{i_{\max}}>1-\sqrt{\log\log\log n}\ e^{-\sqrt{\log\log\log n}},
\end{align*}
where the inequalities follow from the fact that $i^*\leq i_{\max}<\sqrt{\log\log\log n}$. This completes the proof of \eqref{eq:AMXI}.
\end{proof}

\subsubsection{Proof of Lemma~\ref{lem:low typical load segments}}

For $i\in\N, j\in [N_i]$, we write $s^j_i=(j-1)n(\lfloor\log^{\alpha_i}\!n\rfloor+\lfloor\log^{\alpha_i'}\!n\rfloor)$ and
$t^j_i=s^j_i+n\lfloor\log^{\alpha_i}\!n\rfloor$.
Hence, $(s^j_i,t^j_i]$ and $(t_i^j, s_i^{j+1}]$ are discrete time intervals in the $(i+1)$-th scale where we apply the $j$-th iteration of the $i$-th scale strategy and the $j$-th iteration of the regulating multi-stage threshold strategy, respectively.
Fix $c>0$. We set $\alpha_0=\frac{1+\alpha_1}{3}+\frac{\log (12c+9)}{\log\log n}$ so that $\log^{\alpha_0}n=(12c+9)L$. One can check that $\alpha_0<\alpha_1$ for $n$ large enough. We introduce the following events
\begin{equation}
\begin{split}%
\vdE{i}{j}{Q} &=\left\{\maxload^f([s^j_i,t^j_i])\leq \log^{\alpha_{i-1}}n+Q^{i,j}+Q\right\}, \label{eq:events-Eij-Fj}\\
\vF{i}{j}{Q} &=\left\{\maxload^f\big(s^{j+1}_i\big)\leq Q^{i,j+1}+Q\right\},\\
\vG{j}{i}{Q}&=\left\{|\vH{i}{j}{Q}|\leq 3n\exp\left(-\tfrac{\ell_i^2}{4\log^{\alpha_i}n}\right)\right\},~\text{where}\\
\vH{i}{j}{Q} &=\left\{p\in[n]: L_p(t^j_i)\geq Q^{i,j}+\ell_i+Q\right\}.
\end{split}
\end{equation}

In this subsection, in order to simplify the notations,  we denote by $\overline E$ the complement of the event $E$. The following result plays a key role in establishing Lemma~\ref{lem:low typical load segments}.

\begin{lem}\label{lem:key lem for typical iter}
Consider the $Q$-multi-scale strategy with the initial load vector $\{L_p(0)\}_{p\in [n]}$ satisfying  $L_p(0)\le Q\le 2L\sqrt{\log\log\log n}$ for all $p\in[n]$. Fix $c>0$. For sufficiently large $n$ and all $i\in\N$ such that $\alpha_i \le 1$, we have
\begin{equation}\label{eq:i-th scale}
\P\left(\ \bigcup_{j\in[N_i]}{\overline{{\vdE{i}{j}{Q}\cap \vF{i}{j}{Q}\cap \vG{i}{j}{Q}}}}\ \right)\le n^{-c}.
\end{equation}
\end{lem}

\begin{proof}[Proof of Lemma~\ref{lem:low typical load segments}]
Recall the notation $Q^{i,j}=(2k+1)(j-1)\ell_i$.
The statement \eqref{eq:low typical load segments} is a consequence of the following stronger statement
\begin{equation}\label{eq:low typical load segments-a}
\P\left(\maxload^f([s,s+n\lfloor\log^{\alpha_1}\! n\rfloor])>Q+\log^{\alpha_0}\!n+\sum_{i\in [i_{\max}]} Q^{i,j_i+1}\right)<n^{-c+o(1)}.
\end{equation}
Recall that $L=(\log n)^{\frac{1}{2}+\frac{2}{3(\lfloor\sqrt{\log\log\log n}\rfloor+1/4)}}$, $Q\leq L$ and $\log^{\alpha_0}\! n=(12c+9)L$. These, together with \eqref{eq:i-max-bound} and \eqref{eq:Qij-L}, yield
$$
Q+\log^{\alpha_0}\! n+\sum_{i\in [i_{\max}]}Q^{i,j_i+1}\leq Q+(12c+9)L+i_{\max} L<2L\sqrt{\log\log\log n}<\ell.
$$ 
For $0\leq i\leq i_{\max}$, we write 
 $$
 s_i=\sum_{h=i+1}^{i_{\max}}{j_hn(\lfloor\log^{\alpha_h}\!n\rfloor+\lfloor\log^{\alpha'_h}\!n\rfloor)}
 $$ 
 so that $s_0=s$ and $s_{i_{\max}}=0$. We further denote $Q_i=Q+\sum_{h=i+1}^{i_{\max}}Q^{h,j_h+1}$ so that $Q_{i_{\max}}=Q$ and define $K_i=\big\{\maxload^f(s_i)\le Q_i\big\}$. Keeping the notations $P_{Q'}(\cdot)$ as in the proof of Lemma \ref{lem:key lem for typical iter}, we have
\begin{equation}\label{eq:given K1}
\P\Big(\maxload^f([s,s+n\lfloor\log^{\alpha_1}\! n\rfloor])>Q_0+\log^{\alpha_0}\! n\, |\, K_{1}\Big)< P_{Q_1}\left(\,\overline{{\vdE{1}{j_1+1}{Q_1}}}\,\right).
\end{equation}
Since $Q_1<2L\sqrt{\log\log\log n}$, we may apply Lemma~\ref{lem:key lem for typical iter} to obtain
\begin{equation}\label{eq:prob E1}
P_{Q_1}\left(\,\overline{{\vdE{1}{j_1+1}{Q_1}}}\,\right)<  n^{-2c}.
\end{equation}
Next, we estimate $\P(K_{1})$. As mentioned in the proof of Lemma \ref{lem:key lem for typical iter}, the $j$-th iteration of the $i$-th scale of the $Q$-multi-scale strategy is identical to the first iteration of the $i$-th scale of the $(Q+Q^{i,j})$-multi-scale strategy. This self-similar property implies that  
$$
\P\big(\,\overline{K_i}\cap  K_{i+1}\big)< P_{Q_{i+1}}\left(\,\overline{\vF{i+1}{j_{i+1}}{Q_{i+1}}}\,\right).
$$
Using this inequality and the fact $K_{i_{\max}}=\big\{\maxload^f(0)\le Q\big\}$ which is trivially satisfied by the starting conditions, we obtain
\begin{align}\label{eq:prob K1}
\P\left(\,\overline{K_1}\,\right) &= \P\left(\bigcup_{i=1}^{i_{\max}-1}(\overline{K_i}\cap K_{i+1})\right) \le \sum_{i=1}^{i_{\max}-1}
P_{Q_{i+1}}\left(\,\overline{\vF{i+1}{j_{i+1}}{Q_{i+1}}}\,\right) < (i_{\max}-1) \cdot n^{-2c},
\end{align}
where the last inequality uses Lemma~\ref{lem:key lem for typical iter}, which is applicable since  $Q_i<2L\sqrt{\log\log\log n}$. Combining \eqref{eq:given K1}, \eqref{eq:prob E1}, \eqref{eq:prob K1} and the fact that ${i_{\max}}<\sqrt{\log\log\log n}$, we have for sufficiently large $n$ that
$$
\P\left(\maxload^f([s,s+n\lfloor\log^{\alpha_1}\! n\rfloor])>Q+\log^{\alpha_0}\! n\right)< i_{\max}\cdot n^{-2c}< n^{-c}.
$$
This concludes the proof of  \eqref{eq:low typical load segments-a}, and hence \eqref{eq:low typical load segments}. 
\end{proof}

We now present a couple of auxiliary lemmata that are used in our proof of Lemma~\ref{lem:key lem for typical iter}. The first lemma provides an upper bound on the number of bins with loads above certain level.

\begin{lem}\label{lem:heavilily-loaded-bins}
Let $t>\ell>0$, $h, r^*>0$ and $p\in[0, 1]$. Let $\{L_i(0)\}_{i\in [n]}$ be an initial load vector such that $L_i(0)\leq L_0$ for all $i\in[n]$.  Let $f$ be any two-thinning strategy, which satisfies that $\P\big(\maxload^f([tn])\leq h\big)\geq 1-p$. Define $H=\big\{i\in[n]: L_i(tn)\geq L_0+\ell\big\}$. Then we have
$$
\P\left(|H|>2n\exp\left(-\frac{\ell^2}{4t}\right)+r^*\right)\leq 2\exp\left(-2n\exp\left(-\frac{\ell^2}{2t}\right)\right)+4\exp\left(-\frac{n(r^*/tn)^h}{e(h+1)!}\right)+p.
$$
\end{lem}

Given an event $E$ and $Q\ge 0$, we write $P_{Q}(E)$ for the maximum probability of $E$ under the $Q$-multi-scale strategy with the initial maximum load bounded above by $Q$. Then the second lemma is as follows.

\begin{lem}\label{lem:key lem for typical iter aux}
Consider the $Q$-multi-scale strategy with the initial load vector  $\{L_p(0)\}_{p\in [n]}$ satisfying  $L_p(0)\le Q\le 3L\sqrt{\log\log\log n}$ for all $p\in[n]$. Fix $c>0$. For  sufficiently large $n$ and all $i\in\N$ such that $\alpha_i \le 1$ and all $j\in [N_i]$, we have
\begin{align}
P_Q\left(\,\overline{{\vdE{1}{1}{Q}}}\,\right)&<  n^{-c}, \label{eq:First E}\\
P_Q\left(\,\overline{\vF{i}{j}{Q}}, \vdE{i}{j}{Q}, \vG{i}{j}{Q}\right)&< n^{-e^{\sqrt{\log\log\log n}}}, \label{eq: j step induction}\\
P_Q\left(\,\overline{\vG{i}{j}{Q}},  \vdE{i}{j}{Q},\vF{i}{j-1}{Q}\right)&< 2\exp\left(-n^{1/2-o(1)}\right). \label{eq:G from E}
\end{align}

\end{lem}

With these two lemmata at hand, we now prove Lemma~\ref{lem:key lem for typical iter}.

\begin{proof}[Proof of Lemma~\ref{lem:key lem for typical iter}]
For $i\in\N, j\in [N_i]$, we write 
\[
\vU{i}{j}{Q}= \bigcup_{j'\le j}\overline{\vdE{i}{j'}{Q}\cap \vF{i}{j'}{Q} \cap \vG{i}{j'}{Q}}\]
with $\vU{i}{0}{Q}=\emptyset$. 
Observe that for $i\in\N, j\in [N_i]$ we have
\begin{equation*}
   \vU{i}{j}{Q}=
\left(\,\overline{\vF{i}{j}{Q}}\cap \vdE{i}{j}{Q}\cap \vG{i}{j}{Q}\right)\cup
\left(\,\overline{\vG{i}{j}{Q}}\cap  \vdE{i}{j}{Q}\cap \vF{i}{j-1}{Q}\right)
\cup
\left(\,\overline{\vdE{i}{j}{Q}}\cap  \vF{i}{j-1}{Q}\right)
\cup \vU{i}{j-1}{Q}.
\end{equation*}
Notice that this indeed holds for $j=1$ since the initial load condition implies that $\vF{i}{0}{Q}=\Omega$. This, along with \eqref{eq: j step induction}, \eqref{eq:G from E} from Lemma~\ref{lem:key lem for typical iter aux}, yields
\begin{align}
P_{Q}\left(\vU{i}{j}{Q}\right)&\le 
P_Q\left(\,\overline{\vF{i}{j}{Q}}, \vdE{i}{j}{Q}, \vG{i}{j}{Q}\right)+\
P_Q\left(\,\overline{\vG{i}{j}{Q}},  \vdE{i}{j}{Q},\vF{i}{j-1}{Q}\right)\notag\\
&\phantom{\le}\,+
P_Q\left(\,\overline{\vdE{i}{j}{Q}}, \vF{i}{j-1}{Q}\right)
+P_{Q}\left(\vU{i}{j-1}{Q}\right)\notag\\
&\le 
P_Q\left(\,\overline{\vdE{i}{j}{Q}}, \vF{i}{j-1}{Q}\right)
\!+\!P_{Q}\left(\vU{i}{j-1}{Q}\right)\!+\!n^{-e^{\sqrt{\log\log\log n}}}\!+\!
2\exp\left(-n^{1/2-o(1)}\right).\label{eq:iteration step}
\end{align}
Observe that the $j$-th iteration of the $i$-th scale of the $Q$-multi-scale strategy is identical to the first iteration of the $i$-th scale of the $(Q+Q^{i,j})$-multi-scale strategy. Recall that the event $\vF{i}{j-1}{Q}$ asserts that the load at time $s^j_i$ is at most $Q+Q^{i,j}$. Hence we have

\begin{equation}\label{eq:selfsim}
P_{Q}\left(\,\overline{{\vdE{i}{j}{Q}}}, \vF{i}{j-1}{Q} \right)\le P_{Q+Q^{i,j}}\left(\,\overline{{\vdE{i}{1}{Q+Q^{i,j}}}}\,\right).
\end{equation}
Iteration of (\ref{eq:iteration step}) and the above inequality yield
\begin{equation}
P_{Q}\left(\vU{i}{j}{Q}\right) 
\le \sum_{j'\le j}P_{Q+Q^{i,j'}}\left(\,\overline{{\vdE{i}{1}{Q+Q^{i,j'}}}}\,\right)+n^{-\omega(1)}
.\label{eq:i-th scale-a}
\end{equation}

In order to iterate this inequality, we now show that for all $Q'>0$ and $i\ge 2$ the following inclusion inequality holds
\begin{equation}\label{eq:inclu-rela}
\overline{{\vdE{i}{1}{Q'}}} \subset \vU{i-1}{N_{i-1}}{Q'}.
\end{equation}
To see this, we define the event
$$
\vddE{i}{j}{Q} =\left\{\maxload^f\big((s^j_i,s^{j+1}_i]\big)\leq \log^{\alpha_i'}n+Q^{i,j+1}+Q\right\}.
$$
The statement \eqref{eq:inclu-rela} follows from the monotonicity of $\vU{i}{j}{Q'}$ and the following inclusion relations
\begin{align}
\overline{\vdE{i}{1}{Q'}} &\subset \left(\cup_{j\in[N_{i-1}]}\overline{\vddE{i-1}{j}{Q'}}\,\right) \label{eq:E to E-tilde}, \\
\overline{\vddE{i}{j}{Q'}} &\subset \left(\,\overline{\vdE{i}{j}{Q'}}\cup \overline{\vF{i}{j}{Q'}}\,\right)\subset \vU{i}{j}{Q'}. \label{eq: E to U}
\end{align}
To see \eqref{eq:E to E-tilde}, observe that  $\overline{\vdE{i}{1}{Q'}}$ asserts that over $(0,t_i^1]=\cup_{j\in [N_{i-1}]}(s^j_{i-1},s^{j+1}_{i-1}]$, the maximum load is greater than $\log^{\alpha_{i-1}}n+Q'$, while $\overline{\vddE{i-1}{j}{Q'}}$ asserts that over $(s^j_{i-1},s^{j+1}_{i-1}]$,  the maximum load is greater than $\log^{\alpha_{i-1}'}n+Q^{i-1, j+1}+Q'$. Using \eqref{eq:alpha-i-vs-alpha-i'-a}, we have $\log^{\alpha_{i-1}'}n=o(\log^{\alpha_{i-1}}n)$ and by \eqref{eq:Qij-L} we have $Q^{i-1, j+1}<L=o(\log^{\alpha_{i-1}}n)$. These observations yield \eqref{eq:E to E-tilde}. To see \eqref{eq: E to U}, observe that
$$\maxload^f(s_i^{j+1})\ge \maxload^f\big((t^j_i,s^{j+1}_i]\big)-|(t^j_i,s^{j+1}_i]|=\maxload^f\big((t^j_i,s^{j+1}_i]\big)-\log^{\alpha_i'}n.$$
Hence, whenever $\vF{i}{j}{Q'}$ occurs, we have $\maxload^f\big((t^j_i,s^{j+1}_i]\big)\leq \log^{\alpha_i'}n+Q^{i, j+1}+Q'$. This, along with $\vE{i}{j}{Q'}$, implies that $\big(\vE{i}{j}{Q'}\cap \vF{i}{j}{Q'}\big)\subset \vddE{i}{j}{Q'}$, which is equivalent to the first inclusion inequality in \eqref{eq: E to U}.
The second inclusion inequality in \eqref{eq: E to U} is trivial. 

Then we can use \eqref{eq:inclu-rela} to iterate \eqref{eq:i-th scale-a} and obtain
\begin{align*}
    P_Q\left(\vU{i}{j}{Q}\right)\le \sum_{j_i\leq j}\sum_{j_{i-1}\leq N_{i-1}}\cdots\sum_{j_1\leq N_1}P_{Q+Q^{i,j_i}+\cdots+Q^{1, j_1}}\left(\,\overline{\vE{1}{1}{Q+Q^{i,j_i}+\cdots+Q^{1, j_1}}}\,\right)+n^{-\omega(1)}.
\end{align*}
One can use \eqref{eq:i-max-bound} and  \eqref{eq:Qij-L} to check that 
$$
Q+Q^{i,j_i}+\cdots+Q^{1, j_1}\leq Q+iL\leq Q+i_{\max}L\leq 3L\sqrt{\log\log\log n}. 
$$
Then we apply \eqref{eq:First E} from Lemma~\ref{lem:key lem for typical iter aux}, \eqref{eq:N-i-bd} and \eqref{eq:i-max-bound} to obtain for sufficiently large $n$ that
\begin{align*}
    P_Q\left(\vU{i}{j}{Q}\right) &\le \left(\prod_{i'=1}^{i}N_{i'}\right)n^{-2c}+n^{-\omega(1)}\le(\log n)^{i_{\max}\cdot\frac{2\alpha_1-1}{6}}n^{-2c}+n^{-\omega(1)}\\
    &= (\log n)^{O(1)}\cdot n^{-2c}+n^{-\omega(1)}\leq n^{-c}.
\end{align*}
This concludes the proof.
\end{proof}
\subsubsection{Proofs of Lemmata \ref{lem:heavilily-loaded-bins} and \ref{lem:key lem for typical iter aux}}

\begin{proof}[Proof of Lemma~\ref{lem:heavilily-loaded-bins}]
We denote by $r$  the total number of retries up to time $tn$ and by $H'$ the set of bins which are suggested as primary allocations at least $t+\ell$ times by time $tn$. Then, we have
$$
|H|\leq |H'|+r. 
$$
Hence, we have
\begin{equation}\label{eq:bound-H}
\P\left(|H|>2n\exp\left(-\frac{\ell^2}{4t}\right)+r^*\right)\leq \P\left(|H'|>2n\exp\left(-\frac{\ell^2}{4t}\right)\right)+\P(r>r^*).
\end{equation}

We now estimate the first term. We denote by $\{X_i\}_{i\in[n]}$ independent $\poss(t)$ random variables. Write $Y_i$ for the indicator function of the event $\{X_i\geq t+\ell\}$ and $Y=\sum_{i=1}^nY_i$. By Lemma \ref{lem:poss tail}, we have
$$
\P(Y_i=1)=\P(X_i\geq t+\ell)\leq e^{-tI(\ell/t)}\leq \exp\left(-\frac{\ell^2}{4t}\right),
$$
where the second inequality follows from the lower bound of $I(x)$ in \eqref{eq:approx-I} and the assumption that $\ell/t<1$. Lemma \ref{lem:poisson approximation} and Hoeffding's inequality imply that
\begin{equation}\label{eq:bound-H'}
\P\left(|H'|>2n\exp\left(-\frac{\ell^2}{4t}\right)\right)\leq 2\P\left(Y>2n\exp\left(-\frac{\ell^2}{4t}\right)\right)\leq 2\exp\left(-2n\exp\left(-\frac{\ell^2}{2t}\right)\right).
\end{equation}

Next, we estimate the second term. Set $E=\{\maxload^f([tn])\leq h\}$.  By the law of total probability,
\begin{align}\label{eq:bound-r}
\P(E) &= \P(E, r\ge r^*)+\P(E, r < r^*)\leq \P(E, r\ge r^*)+\P(r< r^*)\notag\\
&=\P(E, r\ge r^*)+1-\P(r\ge r^*).
\end{align}

Recall that $R_k$ given in \eqref{eq:def-retry} is the number of retries after allocating $k$ balls. We denote by
$s_0=\inf\big\{s\in[t] : R_{s n}-R_{(s-1) n}\ge r^*/t\big\}$. Whenever $\{r\geq r^*\}$ occurs,  we have $s_0<\infty$. Write $S=\big\{i\in [n] : L_i^f((s_0-1)n)\geq 0\big\}$. As per \eqref{eq:average level bound},
we show that whenever $E$ occurs, then
\begin{equation}\label{eq:average level bound III} 
|S|\geq  \frac{n}{h+1}.
\end{equation}
To see this, observe that
$$
0=\sum_{i\in[n]}L_i^f((s_0-1)n)=\sum_{i\in S}L_i^f((s_0-1)n)+\sum_{i\in S^c}L_i^f((s_0-1)n).
$$ 
This, together with the fact that $\big\{L_i^f((s_0-1)n)\big\}_{i\in [n]}\in\Z^n$ and $\maxload^f((s_0-1)n)< h$, yields
$$
|S^c|\leq\sum_{i\in S^c}|L_i^f((s_0-1)n)|=\sum_{i\in S}L_i^f((s_0-1)n)\leq |S|\cdot (h+1).
$$
Then we can obtain \eqref{eq:average level bound III} using $|S^c|=n-|S|$. 

Denote by $\{Z_i\}_{i\in[n]}$ independent $\poss\left(r^*/tn\right)$ random variables. By Lemma \ref{lem:poisson approximation} and Lemma \ref{lem:max load},
\begin{equation}\label{eq:bound-r-a}
\P(r>r^*, E)\leq 2\P\left(\max_{i\in S}Z_i\leq h\right)\leq 4\exp\left(-\frac{n(r^*/tn)^h}{e(h+1)!}\right).
\end{equation}
Inequalities \eqref{eq:bound-r}, \eqref{eq:bound-r-a} and the fact that $\P(E)\geq 1-p$ imply that
$$
\P(r>r^*)\leq 4\exp\left(-\frac{n(r^*/tn)^h}{e(h+1)!}\right)+p.
$$
We can conclude the proof by combining this with \eqref{eq:bound-H} and \eqref{eq:bound-H'}.
\end{proof}

\begin{proof}[Proof of Lemma~\ref{lem:key lem for typical iter aux}.]
\textbf{Proof of \eqref{eq:First E}.} The statement readily follows from the application of Proposition~\ref{prop:all time upper bound case 1} with the parameters $L_0:=Q$, $t:=\lfloor\log^{\alpha_1}\! n\rfloor$, $\ell:=L=(\log n)^{\frac{1+\alpha_1}{3}}$ and our definition of $\alpha_0$ such  that $\log^{\alpha_0}
n=(12c+9)L$.

\textbf{Proof of \eqref{eq: j step induction}.} The statement follows from the application of Proposition \ref{prop:single time upper bound case 3} with the parameters $t_0:=t_i^j$, $t:=s_i^{j+1}$, $\alpha:=\alpha_i'$, $\eta:=\alpha_i-\alpha_i'$, $L_0:=Q^{i,j}+\ell_i+Q$. Hence it suffice to show that the conditions of Proposition \ref{prop:single time upper bound case 3} are satisfied. 

We first verity the technical requirement $\eta\le \frac{\alpha-1/2}{4k-2}$, which is assumed in our definition of the multi-stage threshold strategy in Section~\ref{sec:combin}. Using $\eta=\alpha_i-\alpha_i', \alpha=\alpha_i'$ and \eqref{eq:alpha-i-prime-iteration}, we can rewrite this requirement as 
$$
\frac{\alpha_i-1/2-\varepsilon_i/2}{5k+5/2}\le\frac{\alpha_i-1/2}{4k-1},
$$
which clearly holds.

We next show that both assumptions in Proposition \ref{prop:single time upper bound case 3} hold when  $\vE{i}{j}{Q}$ and $\vG{i}{j}{Q}$ occur. Given the event $\vG{i}{j}{Q}$, the second assumption trivially holds. We now verify the first assumption that
$\maxload^f(t_i^j)=o(t-t_0)$.
Assuming the event $\vE{i}{j}{Q}$, we have
$$
\maxload^f(t_i^j)\leq \log^{\alpha_{i-1}}n+Q^{i,j}+Q\leq \log^{\alpha_{i-1}}n+L+Q, 
$$
where the last inequality follows from \eqref{eq:Qij-L}. Recall $\alpha_1=\frac{1}{2}+\frac{2}{\lfloor\sqrt{\log\log\log n}\rfloor+1/4}$, $L=(\log n)^{\frac{1+\alpha_1}{3}}$, $Q\leq 3L\sqrt{\log\log\log n}$ and $k=\big\lfloor\frac{\log\log n}{3\log\log\log n}\big\rfloor$.
We have
$$
L+Q<4L\sqrt{\log\log\log n}=(\log n)^{\frac{1}{2}+\frac{2}{3\sqrt{\log\log\log n}}+O(\frac{\log\log\log\log n}{\log\log n})},
$$
while
$$
t-t_0=\log^{\alpha_i'}n=(\log n)^{\alpha_i-O(\frac{1}{k})}>(\log n)^{\alpha_1-O(\frac{1}{k})}=(\log n)^{\frac{1}{2}+\frac{2}{\sqrt{\log\log\log n}+1/4}-O(\frac{1}{k})}.
$$
These, together with \eqref{eq:alpha-i-vs-alpha-i'}, verify the first assumption of Proposition \ref{prop:single time upper bound case 3}. Hence, we can apply Proposition \ref{prop:single time upper bound case 3} to obtain \eqref{eq: j step induction}.

\textbf{Proof of \eqref{eq:G from E}.} Recall our definition $\vG{i}{j}{Q}=\big\{|\vH{i}{j}{Q}|\leq 3n\exp\big(\!-\tfrac{\ell_i^2}{4\log^{\alpha_i}\!n}\big)\big\}$. We introduce
$$
\vtG{i}{j}{Q}=\left\{|\vH{i}{j}{Q}|\leq 2n\exp\left(\!-\frac{\ell_i^2}{4\log^{\alpha_i}\! n}\right)+n\exp\left(\!-\frac{\log n}{2(\log^{\alpha_{i-1}}\! n+L+Q)}\right)\right\}.
$$
We will show that $\vtG{i}{j}{Q}\subset \vG{i}{j}{Q}$ and that 
\begin{equation}\label{eq:G-tilde from E}
P\left(\,\overline{\vtG{i}{j}{Q}}, \vE{i}{j}{Q}, \vF{i}{j-1}{Q}\right)\le 2\exp\left(-n^{1/2-o(1)}\right),
\end{equation}
which implies  \eqref{eq:G from E}.

To see $\vtG{i}{j}{Q}\subset \vG{i}{j}{Q}$, it suffice to show 
that $\frac{\ell_i^2}{\log^{\alpha_i}\! n}=o\left(\frac{\log n}{\log^{\alpha_{i-1}}\! n+L+Q}\right)$. We recall that $\alpha_0=\frac{1+\alpha_1}{3}+\Theta(\frac{1}{\log\log n})$, $\alpha_1=\frac{1}{2}+\Theta(\frac{1}{\sqrt{\log\log\log n}})$, $\ell_i=(\log n)^{\frac{1}{2}+O(\frac{1}{k})}$ and $k=\big\lfloor\frac{\log\log n}{3\log\log\log n}\big\rfloor$. Hence, using again $\log^{\alpha_0}\!n=(12c+9)L$, we have
\begin{align*}
\frac{\ell_1^2}{\log^{\alpha_1}\!n}\Big/\frac{\log n}{\log^{\alpha_0}\! n+L+Q}
&=\frac{\log^{\alpha_0}\! n+L+Q}{\log^{\alpha_0}\! n}\cdot \frac{\ell_1^2/\log^{\alpha_1}\! n}{\log n/\log^{\alpha_0}\! n}\\
&=\frac{(12c+10)L+Q}{(12c+9)L}\cdot(\log n)^{-\frac{2\alpha_1-1}{3}+O(\frac{1}{k})}\\
&<O\big(\sqrt{\log\log\log n}\big)\cdot(\log n)^{-\frac{2\alpha_1-1}{3}+O\left(\frac{1}{k}\right)}=o(1),
\end{align*}
where the inequality follows from that $Q\leq3L\sqrt{\log\log\log n}$.
For $i\geq 2$, we use the fact that  $\log^{\alpha_{i-1}}\! n+L+Q<2\log^{\alpha_{i-1}}\! n$ to obtain
\begin{align*}
\frac{\ell_i^2}{\log^{\alpha_i}\!n}\Big/\frac{\log n}{\log^{\alpha_{i-1}}\! n+L+Q}&<\frac{2\ell_i^2} {(\log n)^{1+\alpha_i-\alpha_{i-1}}} =2(\log n)^{-(\alpha_i-\alpha_{i-1})+O(\frac{1}{k})}\\
&=(\log n)^{-\frac{2\alpha_1-1}{6}+O\left(\frac{1}{k}\right)}=o(1),
\end{align*}
where the second identity follows from \eqref{eq:alpha-i-iteration-a}. 

Towards showing inequality \eqref{eq:G-tilde from E}, we observe that given $\vE{i}{j}{Q}$ and $\vF{i}{j-1}{Q}$, we can 
apply Lemma~\ref{lem:heavilily-loaded-bins} to the process started at time $s_i^j$ with  $L_0=Q^{i, j}+Q$, $t=\lfloor\log^{\alpha_i}\! n\rfloor$, $\ell=\ell_i$, $p=0$, $h=\log^{\alpha_{i-1}}\! n+Q^{i, j}+Q$  and $r^*=n\exp\big(\!-\frac{\log n}{2(\log^{\alpha_{i-1}}\! n+L+Q)}\big)$ to obtain \begin{align}
&P\left(\left\{|H_i^j|>2n\exp\left(-\frac{\ell_i^2}{4\log^{\alpha_i}\! n}\right)+n\exp\left(-\frac{\log n}{2(\log^{\alpha_{i-1}}\! n+L+Q)}\right)\right\}\cap  \vE{i}{j}{Q} \cap \vF{i}{j-1}{Q}\right) \notag\\
&\hspace{12pt}\le  2\exp\left(-2n\exp\left(-\frac{\ell_i^2}{2\lfloor\log^{\alpha_i}\! n\rfloor}\right)\right)+
4\exp\left(-\frac{\sqrt{n}(\log n)^{-\alpha_i(\log^{\alpha_{i-1}}n+L+Q) }}{e\lceil \log^{\alpha_{i-1}}n+L+Q\rceil !}\right),\label{eq:lem app main}
\end{align}
where, in the second term of \eqref{eq:lem app main}, we use the fact that $h<\log^{\alpha_{i-1}} n+L+Q$. For the first term of \eqref{eq:lem app main}, we have
\begin{align}\label{eq:bound-prob-lemma 9.2-a}
2\exp\left(-2n\exp\left(-\frac{\ell_i^2}{2\lfloor\log^{\alpha_i}\! n\rfloor}\right)\right)=\exp\left(-n^{1-o(1)}\right).
\end{align}
The second term of \eqref{eq:lem app main} is increasing with respect to $\alpha_{i-1}$, which, in turn, is increasing with respect to $i$. Hence, we can assume that $i\geq 2$ and use $\log^{\alpha_{i-1}}\! n+L+Q<2\log^{\alpha_{i-1}}\! n$ to obtain 
\begin{align}
4\exp\left(-\frac{\sqrt{n}(\log n)^{-\alpha_i(\log^{\alpha_{i-1}}\!n+L+Q) }}{e\lceil \log^{\alpha_{i-1}}n+L+Q\rceil !}\right)
&\leq4\exp\left(-\frac{\sqrt{n}(\log n)^{-2\alpha_i(\log n)^{\alpha_{i-1}} }}{(2\log^{\alpha_{i-1}}\! n)^{2(\log n)^{\alpha_{i-1}}}}\right)
 \notag\\
&\leq4\exp\left(-\frac{\sqrt{n}}{(2\log n)^{2(\alpha_i+\alpha_{i-1})(\log n)^{\alpha_{i-1}}}}\right) \notag\\
&\leq4\exp\left(-\frac{\sqrt{n}}{\exp(5\log\log n\cdot \log^{\alpha_{i-1}}\! n)}\right) \notag\\
&=\exp\left(-n^{1/2-o(1)}\right),\label{eq:bound-prob-lemma 9.2-b}
\end{align}
where the first inequality follows from Stirling's approximation $n!\leq e\sqrt{n}(n/e)^n$, 
and the last inequality -- from the observation that $\alpha_{i-1}<1-\frac{1}{4\sqrt{\log\log\log n}}$, which, in turn, follows from the fact that $\alpha_i\leq 1$ and  \eqref{eq:alpha-i-iteration-a}.
Plugging and \eqref{eq:bound-prob-lemma 9.2-a}, \eqref{eq:bound-prob-lemma 9.2-b} into \eqref{eq:lem app main}, inequality \eqref{eq:G-tilde from E}, and hence \eqref{eq:G from E}, follows.
\end{proof}

\subsection{Proof of Proposition~\ref{prop:phase-1 minload}}

\begin{proof}[Proof of Proposition~\ref{prop:phase-1 minload}]  
To establish equality \eqref{eq:initial loads bound}, it would clearly suffice to show the following estimates
\begin{align}
\P\Big(\maxload^f(m_1)>A\Big) &\leq n^{-4d}, \label{eq:all-time load m1}\\
\P\left(\min_{i\in[n]}L_i^f(m_1)<-A\right) &\leq n^{-4d}. \label{eq: target lower bound}
\end{align}
We first show that inequality \eqref{eq:all-time load m1} follows from Lemma \ref{lem:key lem for typical iter}. Our choice of $m_1$ in \eqref{eq:ms} guarantees that the allocation of $m_1$ balls using the $Q$-multi-scale strategy ends up with $N_{i_{\max}}$ complete iterations of the $i_{\max}$-th scale strategy followed by the regulating multi-stage threshold strategy. Recall the definition of $\vF{i}{j}{Q}$ given in \eqref{eq:events-Eij-Fj} and apply Lemma \ref{lem:key lem for typical iter} to obtain
$$
\P\left(\maxload^f(m_1)>Q+Q^{i_{\max},N_{i_{\max}}+1}\right)=\P\left(\,\overline{\vF{i_{\max}}{N_{i_{\max}}}{Q}}\,\right)\leq n^{-4d}.
$$
Using \eqref{eq:Qij-L}, we have    $Q+Q^{i_{\max},N_{i_{\max}}+1}<Q+L<A$. The two inequalities above yield \eqref{eq:all-time load m1}.

Next, we estimate $\P(L_i^f(m_1)<-A)$, which, together with the union bound argument, implies inequality \eqref{eq: target lower bound}. For each $i\in [n]$, we denote 
$$
k_i=\sup\{k\in[1, m_1] : L_{i}(k)\ge -300d\log n\}
$$
and write $F_i$ for the event $\{-\infty<k_i<m_1-\log n\}$.
Observe that, given our assumptions on $L_i(0)$, on $F_i^c$ we have $L_i^f(m_1)\ge -A$ almost surely. We denote by $\cF_k$ the filtration generated by $\big\{L_i^f(p)\!: 1\leq p\leq k, i\in [n]\big\}$.
By Chernoff's argument, we thus have that for any $\lambda>0$,
\begin{align}\label{eq:target lower bound-a}
\P\big(L_i^f(m_1)<-A\,|\,F_i,\cF_{k_i}\big)&\leq  e^{-\lambda A}\cdot \E \big[e^{-\lambda L_i^f(m_1)}\,|\, F_i,\cF_{k_i}\big].
\end{align}
We write $p_{k,\cF_{k-1}}=\P\big(L_i^f(k)-L_i^f(k-1)=1-1/n\,|\, \cF_{k-1}\big)$, i.e., the probability that the $k$-th ball has been allocated to the $i$-th bin conditioned on the load vector in time $k-1$. Our strategy never retries a ball if its primary allocation is a bin with load below $-\log n$. This and the definition of $k_i$ imply that $p_{k,\cF_{k-1}}\geq 1/n$ for all $k_i<k\le m_1$. We now compute
\begin{align*}
\E \big[e^{-\lambda L_i^f(m_1)}\,|\,F_i, \cF_{k_i}\big] &= \E\big[e^{-\lambda L_i^f(m_1-1)}\cdot \E \big[e^{-\lambda (L_i^f(m_1)-L_i^f(m_1-1))}\,|\, \cF_{m_1-1}\big]\,|\,F_i,\cF_{k_i}\big] \\
&\le\E \big[e^{-\lambda L_i^f(m_1-1)} \cdot \big(p_{m_1, \cF_{m_1-1}} e^{-\lambda (1-1/n)}+(1-p_{m_1, \cF_{m_1-1}})e^{\lambda/n}\big)\,|\,F_i, \cF_{k_i}\big] \\
&=\E \big[e^{-\lambda L_i^f(m_1-1)}\cdot e^{\lambda/n}\cdot \big(1-(1-e^{-\lambda})p_{m_1, \cF_{m_1-1}}\big)\,|\,F_i, \cF_{k_i}\big]\\
&\leq \E \big[e^{-\lambda L_i^f(m_1-1)}\,|\,F_i, \cF_{k_i}\big]\cdot e^{\lambda/n}\cdot \left(1-\frac{1-e^{-\lambda}}{n}\right),
\end{align*}
Iterate this inequality to obtain 
\begin{align}\label{eq:mgf-load}
\E \big[e^{-\lambda L_i^f(m_1)}\,|\,F_i,  \cF_{k_i}\big] &\leq \E \big[e^{-\lambda L_i^f(k_i)}\,|\,F, \cF_{k_i}\big]\cdot \left(e^{\lambda/n}\cdot \left(1-\frac{1-e^{-\lambda}}{n}\right)\right)^{m_1-k_i} \notag\\
&= e^{-\lambda L_i^f(k_i)}\cdot \left(e^{\lambda/n}\cdot \left(1-\frac{1-e^{-\lambda}}{n}\right)\right)^{m_1-k_i} \notag\\
&\leq e^{300\lambda d \log n}\cdot \left(e^{\lambda/n}\cdot \left(1-\frac{1-e^{-\lambda}}{n}\right)\right)^{m_1},
\end{align}
where the second inequality follows from that $L_i^f(k_i)\geq-300d\log n$ and that $e^{\lambda/n} \left(1-\frac{1-e^{-\lambda}}{n}\right)$ is increasing for $\lambda>0$.
Combining \eqref{eq:target lower bound-a} and \eqref{eq:mgf-load}, we obtain  
\begin{align*}
\P\big(L_i^f(m_1)<-A\,|\,F_i, \cF_{k_i}\big) &\leq e^{-\lambda A}\cdot e^{300\lambda d \log n}\cdot e^{\lambda m_1/n}\cdot \left(1-\frac{1-e^{-\lambda}}{n}\right)^{m_1}\\
&=\exp\left(-\lambda A+300\lambda d\log n+\frac{\lambda m_1}{n}+m_1\log\left(1-\frac{1-e^{-\lambda}}{n}\right)\right)\\
&\leq \exp\left(-\lambda A+300\lambda d \log n+\frac{\lambda m_1}{n}-m_1\cdot\frac{1-e^{-\lambda}}{n}\right)\\
&= \exp\left(-\lambda (A-300d\log n)+\frac{m_1}{n}(\lambda-1+e^{-\lambda})\right)\\
&\leq \exp\left(-\lambda (A-300d\log n)+\frac{\lambda^2m_1}{2n}\right),
\end{align*}
where the second inequality uses $\log(1-x)\leq -x$ for $0\leq x\leq 1$, and the last inequality follows from that $e^{-x}<1-x+x^2/2$ for $x>0$.  We plug $\lambda=\frac{n(A-300d\log n)}{m_1}$ into the above inequality to obtain
\begin{align*}
\P\big(L_i^f(m_1)<-A\,|\,F_i,\cF_{k_i}\big)&\leq \exp\left(-\frac{n(A-300d\log n)^2}{2m_1}\right)=\exp\left(-\frac{(1-o(1))A^2}{2\log^\alpha\! n}\right)<n^{-5d}.
\end{align*}
We recall that 
$\P\big(L_i^f(m_1)<-A\big)= \P(L_i^f(m_1)<-A, F)\le\P(L_i^f(m_1)<-A\,|\,F)$. Hence
inequality \eqref{eq: target lower bound} follows from taking a union bound of the above inequality over $i\in[n]$.
\end{proof}


\subsection{Proofs of Propositions \ref{prop:phase-0 low load} and \ref{prop:phase-2 typicaly m2}}

\begin{proof}[Proof of Proposition~\ref{prop:phase-0 low load}] The statement follows as easy consequence of  Proposition~\ref{prop:single time upper bound case 3} with the parameters $t=m_0/n, \alpha=\frac{\log (m_0/n)}{\log\log n}, \eta=0$. We first show that Proposition~\ref{prop:single time upper bound case 3} is applicable with the aforementioned parameters. Recall that $m_0=\lfloor200dn\log n\rfloor$. This, together with the assumption that $|L_i(0)|\le 100d\log n$ for all $i\in [n]$, implies that $\maxload^f(0)\leq t/2$. Hence, the first condition of Proposition~\ref{prop:single time upper bound case 3} is satisfied. 
Notice that $\alpha\ge1$ and that $L_0=(\log n)^{\frac{1}{2}+\Theta(\frac{1}{k})}$, where $k=\big\lfloor\frac{\log\log n}{3\log\log\log n}\big\rfloor$. We thus have $L_0\gg (L_0)^2/\log^{\alpha} n$. This, along with the assumption that $|\{i\in [n] : L_i(0)>L_0\}|\le 4000ne^{-L_0/15}$, guarantees the validity of the second condition of Proposition~\ref{prop:single time upper bound case 3}. Thus we can apply Proposition \ref{prop:single time upper bound case 3} to obtain that
$$
\P\left(\maxload^f(m_0)>(2k+2)L_0\right)\leq n^{-e^{\sqrt{\log\log\log n}}}.
$$
Observe that $(2k+2)L_0=(\log n)^{\frac{1}{2}+O\big(\frac{\log\log\log n}{\log\log n}\big)}$ and that $L=(\log n)^{\frac{1}{2}+\Theta\big(\frac{1}{\sqrt{\log\log\log n}}\big)}$. Hence, we obtain
\begin{equation}\label{eq:maxload M-b-2}
\P\left(\maxload^f(m_0)>L\right)\leq n^{-e^{\sqrt{\log\log\log n}}}.
\end{equation}
Notice that the load of each bin can decrease by at most $m_0/n\leq 200d\log n$ after the allocation of $m_0$ balls. Since $\max_{i\in [n]}|L_i^f(0)|\le 100d\log n$, we have
$$\min_{i\in[n]}L_i^f(m_0)\ge -300d\log n.$$ This, together with \eqref{eq:maxload M-b-2}, concludes the proof of Proposition \ref{prop:phase-0 low load}.
\end{proof}

\begin{proof}[Proof of Proposition~\ref{prop:phase-2 typicaly m2}] We apply Lemma \ref{lem:single time upper bound} and Lemma \ref{lem:level set bound} to obtain for sufficiently large $n$ that
$$
\P\left(\max_{i\in[n]}|L_i^f(m_2)|>100d\log n\right)\leq n^{-3d},
$$
$$
\P\left(\big|\big\{i\in[n]: L_i^f(m_2)>L_0\big\}\big|> 4000ne^{-L_0 /15}\right) \le \exp\left(-n^{1-o(1)}\right).
$$
Then we can conclude the proof by taking the union bound.
\end{proof}

\end{document}